\documentclass[11pt]{article}
\usepackage[utf8]{inputenc}
\DeclareUnicodeCharacter{202F}{\,}
\usepackage{geometry} % (optional; keep if you plan to control margins here)

\usepackage{amsmath,amssymb,amsfonts,amsthm}
\numberwithin{equation}{section}

\usepackage{graphicx}
\usepackage{float}
\usepackage{pdflscape}
\usepackage{comment}
\usepackage{wrapfig}
\usepackage{subcaption, cite}
\usepackage[table]{xcolor}
\usepackage{array,tabularx,longtable,booktabs,multirow,makecell}

\renewcommand{\arraystretch}{1.5}    % row height scaling
\newcolumntype{s}{>{\columncolor{blue!10}}c}
\arrayrulecolor{gray!10}

\usepackage{enumerate}
\usepackage{enumitem}

\usepackage{authblk}
\usepackage{titlesec}
\usepackage{pdflscape}
\usepackage{adjustbox}
\usepackage{caption}
\usepackage{varioref}
\usepackage{hyperref}
\hypersetup{
	colorlinks=true,
	linkcolor=violet,
	citecolor=cyan,
	urlcolor=teal
}
\usepackage[sort&compress]{cleveref}
\crefname{figure}{figure}{figures}
\Crefname{figure}{Figure}{Figures}
\def\RR{\hbox{I\kern-.2em\hbox{R}}}
\newcommand{\myqed}{\hbox to 0pt{}\hfill$\rlap{$\sqcap$}\sqcup$\vspace{3mm}}

\newcommand{\estci}[3]{%
	\shortstack[c]{#1\\[-0.4mm]\scriptsize [#2,\;#3]}%
}
\usepackage[table]{xcolor}
\usepackage{booktabs}
\usepackage{tabularx}
\usepackage{array}
\usepackage{makecell}
\usepackage{threeparttable}
\usepackage{float}
\usepackage{rotating}
\definecolor{TableBlue}{RGB}{38,78,121}
\definecolor{HeaderFill}{RGB}{224,234,244}

\definecolor{SourceFill}{RGB}{248,226,226}
\definecolor{SourceText}{RGB}{145,35,35}

\definecolor{SinkFill}{RGB}{226,237,247}
\definecolor{SinkText}{RGB}{35,82,130}

\definecolor{ExporterFill}{RGB}{238,230,243}
\definecolor{ExporterText}{RGB}{102,56,120}

\definecolor{ImporterFill}{RGB}{228,241,233}
\definecolor{ImporterText}{RGB}{38,105,68}

\newcommand{\SourceCell}{%
	\cellcolor{SourceFill}\textcolor{SourceText}{\textbf{Source}}}

\newcommand{\SinkCell}{%
	\cellcolor{SinkFill}\textcolor{SinkText}{\textbf{Sink}}}

\newcommand{\ExporterCell}{%
	\cellcolor{ExporterFill}\textcolor{ExporterText}{\textbf{Exporter}}}

\newcommand{\ImporterCell}{%
	\cellcolor{ImporterFill}\textcolor{ImporterText}{\textbf{Importer}}}
\definecolor{ScenarioBlue}{RGB}{38,78,121}
\definecolor{SubheaderBlue}{RGB}{224,234,244}
\definecolor{AlternateRow}{RGB}{246,248,250}
\definecolor{BaselineFill}{RGB}{232,240,248}
\definecolor{ReductionGreen}{RGB}{35,110,70}
\definecolor{IncreaseRed}{RGB}{160,45,45}
\newtheorem{Th}{Theorem}

\graphicspath{{Figure/}}
\date{}

\begin{document}
%	\title{Spatial Heterogeneity, Vaccination, and Interdivisional Connectivity in the 2026 Measles Outbreak in Bangladesh: A Multipatch Modelling Study}
	\title{A Multipatch Model for Evaluating Interdivisional Connectivity and Vaccination Dynamics in the 2026 Bangladesh Measles Outbreak}
%	\title{Multipatch Dynamics and Spatial Control of the 2026 Measles Outbreak in Bangladesh}
%	\title{Multipatch Dynamics, Vaccination, and Spatial Control of the 2026 Measles Outbreak in Bangladesh}

	\author[1, 2]{\small Faizunnesa Khondaker\thanks{Email: faizunnesa@math.jnu.ac.bd}}
	\author[1]{\small Md. Kamrujjaman\thanks{Corresponding author   Email: kamrujjaman@du.ac.bd}}
	%	\author[2]{\small  Md. Shahidul Islam\thanks{Email: mshahid@du.ac.bd}}

	\affil[1]{\footnotesize Department of Mathematics, University of Dhaka, Dhaka 1000, Bangladesh}
	\affil[2]{\footnotesize Department of Mathematics, Jagannath University, Dhaka 1100, Bangladesh}
	
	\affil[ ]{\textbf{Corresponding author:} Md. Kamrujjaman, \href{mailto:kamrujjaman@du.ac.bd}{kamrujjaman@du.ac.bd}}
	\maketitle
	
	\vspace{-1.0cm}
	\noindent\rule{6.35in}{0.02in}\\

	\begin{abstract}
		The 2026 measles resurgence in Bangladesh motivates a spatial modelling framework that accounts for heterogeneous local transmission, vaccination activity, and interdivisional connectivity. We formulate a multipatch susceptible-vaccinated-infectious-recovered model for the eight administrative divisions of Bangladesh. Division-specific parameters are estimated from reported measles cases \cite{DGHS} and smoothed vaccination data using bounded nonlinear least squares, with uncertainty quantified through 1000 parametric bootstrap replicates. The fitted divisional systems are then connected through compartment-specific movement and cross-patch infection coefficients  adjusted for connectivity and the source division's relative reproduction number. The fitted model captures the main division-specific epidemic patterns. Estimated local reproduction numbers range from \(1.1386\) in Khulna (95\% CI:\(0.5703\)--\(1.3834\)) to \(4.3491\) in Dhaka (95\% CI: \(1.4449\)--\(5.6307\)), while Chattogram also exhibits high transmission potential (\(\mathcal{R}_0^i=4.2275\); 95\% CI: \(4.0019\)--\(4.3185\)). Baseline spatial coupling increases the network-wide epidemic peak from \(190.96\) to \(313.20\) infectious individuals and more than doubles cumulative infections from \(1086.42\) to \(2345.35\). Relative to baseline coupling, 50\% and 80\% restrictions reduce the epidemic peak by \(26.69\%\) and \(34.37\%\), and cumulative infections by \(25.07\%\) and \(41.48\%\), respectively. Conversely, a 50\% increase in spatial coupling raises the peak by \(30.14\%\) and cumulative infections by \(23.54\%\). Although local transmission remains dominant in every division, cross-patch infection accounts for \(28.83\%\) of new infections in Rajshahi at the baseline epidemic peak. Source-sink analysis identifies Dhaka as the dominant cross-patch source and Chattogram as a smaller secondary source. These findings suggest that, under the assumed connectivity structure, coordinated vaccination and local outbreak control should be complemented by targeted interventions along major interdivisional transmission pathways. 		
		\end{abstract}

	\noindent\textbf{Keywords:} measles; multipatch model; vaccination; human movement; cross-patch infection; reproduction number; Bangladesh.
	
	\noindent\rule{\textwidth}{0.4pt} % optional separator

		\section*{Highlights}
	
	\begin{itemize}[leftmargin=*, label=$\bullet$]
		\item Formulated a multipatch SVIR model for Bangladesh's eight divisions, calibrated via bounded nonlinear least squares on case and vaccination data with {1000} bootstrap replicates.
		\item Identified high spatial transmission heterogeneity, with local $\mathcal{R}_0^i$ ranging from {1.1386} ({95\%} CI: {0.5703}--{1.3834}) in Khulna to {4.3491} ({95\%} CI: {1.4449}--{5.6307}) in Dhaka.
		\item Demonstrated that baseline interdivisional coupling more than doubles cumulative national infections (from {1086.42} to {2345.35}) and elevates the peak infectious count to {313.20}.
		\item Showed that {50\%} and {80\%} travel restrictions reduce cumulative infections by {25.07\%} and {41.48\%}, whereas a {50\%} coupling increase raises cumulative cases by {23.54\%}.
		\item Uncovered significant cross-patch dynamics--accounting for {28.83\%} of peak infections in Rajshahi--and identified Dhaka as the primary interdivisional transmission source.
	\end{itemize}
	
	\clearpage
	\section{Introduction}
	
	Measles is an acute viral disease that can affect individuals of all ages and remains a leading cause of death among young children worldwide. It is exceptionally contagious and spreads through airborne particles and respiratory droplets. Symptoms generally appear 10--14 days after exposure, while infected individuals can transmit the virus from approximately four days before to four days after rash onset, allowing transmission before the disease is readily recognized. Although most patients recover, measles can cause pneumonia, severe diarrhoea, encephalitis, blindness, and death, particularly among young, malnourished, and unvaccinated children \cite{WHO2026Bangladesh}.
	
Despite the availability of a safe and effective vaccine, measles has resurged worlwide, including in high-income countries that had previously achieved elimination. WHO estimates that approximately 11 million infections and 95,000 measles related deaths occured globally in 2024. During the same year, only 84\% of children received the first measles-containing vaccine dose and 76\% received the second, both remaining well below the 95\% coverage required to interrupt transmission. Moreover, 59 countries reported large or disruptive outbreaks, demonstrating the consequences of persistent immunization gaps. Even where national coverage appears relatively high, aggregated estimates may conceal under-vaccinated communities in which susceptible individuals accumulate, permitting localized outbreaks and continued transmission \cite{WHO2025GlobalMeasles}. 
 Under the national Expanded Programme on Immunization in Bangladesh, children receive two measles--rubella vaccine doses at 9 and 15 months of age. However, MR1 coverage declined from 88.6\% in 2019 to 86.0\% in 2023, while MR2 coverage fell more sharply from 89.0\% to 80.7\%. These shortfalls are estimated to leave approximately 10 million children without MR1 protection and 20 million without MR2 protection \cite{Situational}. Together with missed or delayed vaccination and increasing vaccine hesitancy associated with misinformation, these accumulated immunity gaps have contributed to the recent resurgence of measles \cite{UNICEF2026MeaslesFAQ}. WHO reports that, on 4 April 2026, Bangladesh’s National IHR Focal Point notified the organization of a nationwide increase in measles cases affecting all eight divisions and 58 of the country’s 64 districts, with children younger than five years comprising \(79\%\) of the cases reported during the initial phase of the outbreak \cite{WHO2026Bangladesh}. From 15 March to 3 August, Bangladesh reports 131,050 suspected cases, 16,408 confirmed cases, 751 suspected deaths, and 96 confirmed deaths \cite{DGHS}. In response, an emergency campaign beginning on 5 April in 30 high-risk upazilas expands nationwide \cite{child}, targeting more than 17.8 million children aged 6--59 months and aiming for at least 95\% coverage \cite{UNICEF}.

The scale and geographical variation of this response highlight the need for quantitative frameworks to evaluate how vaccination, local transmission, and population connectivity shape outbreak dynamics. Mathematical modelling has therefore been widely applied across infectious diseases to investigate transmission thresholds, vaccination effects, spatial spread, and control strategies. In~\cite{Wang2008}, \(R_0\) is established as the invasion threshold for periodic compartmental models, while the time-averaged autonomous system may underestimate or overestimate infection risks in more general cases. Al Darabsah~\cite{Darabsah} examines the effect of the latent period and imperfect vaccination on measles dynamics and critical vaccination coverage. Kanchanarat et al.~\cite{Kanchanarat} use data from Thailand to assess how pulse-vaccination timing influences measles extinction and immunization scheduling. Spatial studies further demonstrate that transmission depends on both local epidemiological conditions and population connectivity.
	Lieberthal and Gardner~\cite{Lieberthal} show that superspreader risk is jointly determined by local transmission potential and network connectivity. A multipatch framework incorporating heterogeneous mobility and secondary contacts is developed in~\cite{Zhu} to examine movement-mediated regional transmission. Using a mobility-integrated renewal framework, Roy et al.~\cite{Roy} find that infections acquired within and outside an individual's home region jointly shape transmission and that excluding mobility can bias regional \(R_t\) estimates. Related vector-borne disease models provide further evidence of the complex effects of spatial movement. Host dispersal is shown by Gao and Lou~\cite{Gao2025} to either increase or decrease local and global disease prevalence, while spatial heterogeneity may reduce host prevalence but increase vector prevalence and disease persistence. In~\cite{Gao2024}, host commuting and vector migration are jointly modelled, revealing that nonhomogeneous mixing generally strengthens persistence, whereas proportional host-vector distributions minimize \(R_0\). From a stochastic perspective, Nipa and Allen~\cite{Nipa} use a multitype branching-process approximation and the backward Kolmogorov diffrential equation to estimate outbreak probability as a function of imported infections, seasonality, dispersal, and interpatch connectivity. Several epidemiological modelling studies have also examined disease dynamics, intervention thresholds, and optimal-control strategies for reducing transmission and population-level burden \cite{Khondaker2026Sci, Khondaker2026plos, Khondaker2025, Khondaker2025acta, Lahodny, Khondaker2022abr, Khondaker2022am}.

	Motivated by these epidemiological and spatial modelling studies, we formulate a division-level multipatch SVIR model that integrates local measles transmission, time-dependent vaccination, interdivisional population movement, and cross-patch infection pressure. Section~\ref{formulation} presents the model formulation, while Section~\ref{stability} examines the equilibria and stability properties, including the disease-free equilibrium in Section~\ref{dfe} and the basic reproduction number in Section~\ref{reproduction}. Numerical findings and discussions are presented in Section~\ref{numerical}, followed by the main conclusions in Section~\ref{conclusion}. 
	Additional figures and tables supporting the analysis are provided in Appendix~\ref{figs} and Appendix~\ref{tables}, respectively.
	%Additional figures and tables supporting the analysis are provided in the Appendix \ref{figs} Supplementary Figures and Appendix \ref{tables} Supplementary Tables sections of the Appendix.
	
	\section{Model formulation}\label{formulation}
	Recent surveillance during the 2026 measles outbreak in Bangladesh reveals substantial differences in disease burden and vaccination activity among administrative divisions \cite{DGHS}. The combined effects of heterogeneous vaccination coverage, high population density, and continued interdivisional mobility motivate a multipatch model incorporating local transmission, time-dependent vaccination, physical movement, and cross-patch infection pressure. The model has a common structure across all divisions, while the local parameters are calibrated separately using data from each division.\\
	The total population is divided into \(n\) geographically distinct but epidemiologically connected patches, with each patch representing one administrative division. Such a metapopulation structure allows local disease dynamics to interact through human movement and infection exchange between regions \cite{bharti, apolloni}. In patch $i, S_i(t), V_i(t), I_i(t),$ and $R_i(t)$ denote the susceptible, vaccinated, infectious, and recovered populations, respectively. Susceptible individuals enter at rate \(\Lambda_i\), acquire infection at rate \(\lambda_iS_i\), receive vaccination at rate \(\phi_i(t)S_i\), and experience natural mortality at rate \(\mu_iS_i\). The figure \ref{diagram} shows the spatial-flow diagram of the model. 
	\begin{figure}[H]
		\centering
		\includegraphics[width=5.5 in]{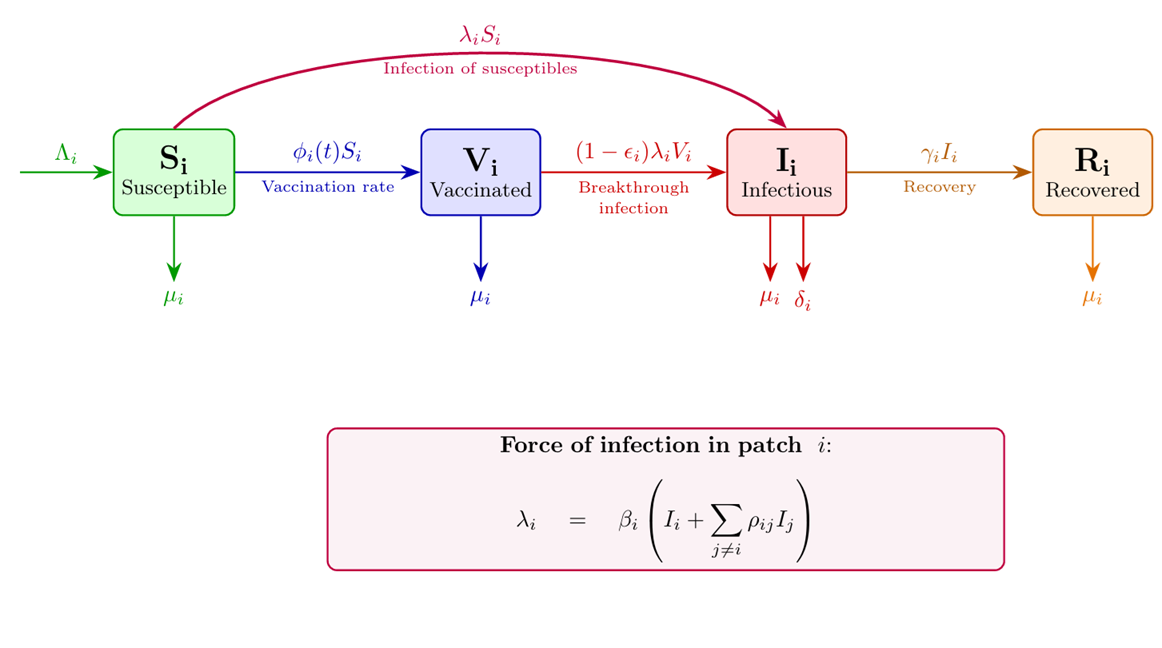}
		\caption{Transmission and spatial-flow structure of the multipatch SVIR model for measles dynamics.}	
		\label{diagram}
	\end{figure}
	Although the emergency vaccination campaign targets children aged 6--59 months, the model represents the total population of each division without explicit age structure. Therefore, $\phi_i(t)$ is interpreted as an effective population-level vaccination rate derived from the temporal intensity of the child-targeted campaign. The fitted scaling parameter $\alpha_i$ empirically translates the normalized vaccination pattern into its average division-wide effect and does not separately estimate the proportion of vaccine-eligible children. Consequently, $\phi_i(t)S_i$ represents an aggregated reduction in population susceptibility rather than uniform vaccination across all age groups. This formulation permits population-wide transmission and movement to be analysed, while age-specific differences in susceptibility, vaccine eligibility, and mobility remain outside the scope of the model.
	
 Vaccinated individuals retain partial protection. The parameter $\epsilon_i$ denotes vaccine effectiveness; therefore, \(1-\epsilon_i\) represents the remaining susceptibility to breakthrough infection, which occurs at rate \((1-\epsilon_i)\lambda_iV_i\) \cite{Liu}. Newly infected individuals enter \(I_i\) from both the susceptible and vaccinated compartments and leave through recovery at rate \(\gamma_i\), disease-induced mortality at rate \(\delta_i\), or natural mortality at the rate \(\mu_i\). Recovered individuals enter \(R_i\) at rate \(\gamma_iI_i\), where recovery is assumed to confer immunity, consistent with the standard biological description of measles infection \cite{bharti}. \\
 Human mobility is incorporated separately for each epidemiological compartment. For \(X\in{S, V, I, R}, d_{ij}^X\) denotes the per-capita movement rate from source patch \(j\) to receiving patch \(i\); consequently, \(d_{ij}^X X_j\) represents inflow to patch \(i\), whereas \(d_{ji}^X X_i\) represents outflow from patch \(i\) to patch \(j\). Allowing movement rates to vary among compartments reflects the possibility that mobility behavior differs according to epidemiological status, particularly because infectious individuals may travel less during illness. The force of infection in patch \(i\) is
 \begin{align*}
 	\lambda_i=\beta_i\left(I_i+\sum_{j\neq i} \rho_{ij}I_j\right), 
 	\end{align*}
 	where \(\beta_i\) is the effective transmission coefficient, \(I_i\) represents local infectious pressure, and \(\rho_{ij}I_j\) represents infection pressure exerted on receiving patch \(i\) by infectious individuals associated with source patch \(j\). Thus, \(\rho_{ij}\) captures cross-patch transmission without requiring the physical relocation of those infectious individuals, whereas \(d_{ij}^I\) describes their actual movement between patches. This distinction allows the model to capture both the physical movement of individuals and movement-mediated contact across divisional boundaries, two important mechanisms of spatial measles transmission \cite{bharti}. These biological and spatial processes are summarized in the flow diagram (figure~\ref{diagram}) and mathematically represented by the system of differential equations presented below
\begin{equation}
	\text{Patch } i \quad
	\left\{
	\begin{aligned}
		\dot{S}_i
		&=
		\Lambda_i-\lambda_iS_i-\phi_i(t)S_i-\mu_iS_i
		+\sum_{\substack{j=1,\;j\neq i}}^{n}
		\left(
		d_{ij}^{S}S_j-d_{ji}^{S}S_i
		\right),\\[2mm]
		\dot{V}_i
		&=\phi_i(t)S_i-(1-\epsilon_i)\lambda_iV_i-\mu_iV_i
		+\sum_{\substack{j=1,\;j\neq i}}^{n}
		\left(
		d_{ij}^{V} V_j-d_{ji}^{V}V_i
		\right),\\[2mm]		
		\dot{I}_i
		&=
		\lambda_iS_i+(1-\epsilon_i)\lambda_iV_i
		-(\gamma_i+\delta_i+\mu_i)I_i
		+\sum_{\substack{j=1,\;j\neq i}}^{n}
		\left(
		d_{ij}^{I}I_j-d_{ji}^{I}I_i
		\right),\\[2mm]		
		\dot{R}_i
		&=
		\gamma_iI_i-\mu_iR_i
		+\sum_{\substack{j=1,\;j\neq i}}^{n}
		\left(
		d_{ij}^{R}R_j-d_{ji}^{R}R_i
		\right).
	\end{aligned}
	\right.
	\label{model}
\end{equation}
and
\begin{equation}
	\lambda_i
	=
	\beta_i
	\left(
	I_i+\sum_{\substack{j=1,\;j\neq i}}^{n}\rho_{ij}I_j
	\right).
	\label{foi}
\end{equation}
%\rho_{ji}=measures the influence/movement/contact from patch j toward patch i.
% ρi​j=infection influence from patch j→i.
where
\[
S_i(0)>0,\quad
V_i(0)\geq0,\quad
I_i(0)\geq0,\quad
R_i(0)\geq0,
\]
and
\[
N_i=S_i(0)+V_i(0)+I_i(0)+R_i(0).
\]
Together, equations~\eqref{model} and \eqref{foi}, subject to the stated nonnegative initial conditions, define the multipatch SVIR framework. This system provides the basis for the subsequent disease-free equilibrium, reproduction number, and numerical analyses.
\subsection{Equilibria and Stability}\label{stability}
\subsubsection {Disease-free equilibrium}\label{dfe}
Since the vaccination rate $\phi_i(t)$ is assumed to be time-dependent, the proposed model is a non-autonomous system. Consequently, the system does not admit a constant disease-free equilibrium in the classical sense. To facilitate analytical derivation of the disease-free equilibrium (DFE) and the basic reproduction number, the time-dependent vaccination rate $\phi_i(t)$ is replaced by its mean value over the fitting period, denoted by $\bar{\phi}_i$. This provides a constant representative vaccination rate for each division and converts the non-autonomous system into an autonomous system. At the DFE, there is no infection in any patch. Hence, 
\begin{align}
	I_i^{0}=0,\qquad R_i^{0}=0,\qquad\lambda_i^{0}=0,\qquad i=1,2,\ldots,n.
\end{align}
Therefore, the DFE takes the form
\begin{align}
\mathcal{E}^{0}=\left( S_1^{0}, V_1^{0},0,0,\ldots,S_n^{0},V_n^{0},0,0\right) .
\end{align}
Substituting $I_i=0,\;R_i=0,\;$ and $\lambda_i=0$ into system \eqref{model}, we obtain the disease-free subsystem 
\begin{align}
	\left\{
	\begin{aligned}
		 \Lambda_i-\bar{\phi}_iS_i^{0}-\mu_iS_i^{0}+\sum_{\substack{j=1,\;j\neq i}}^{n}
		\left(d_{ij}^{S}S_j^{0}-d_{ji}^{S}S_i^{0} \right)&=0,\\[2mm]
		\bar{\phi}_iS_i^{0}-\mu_iV_i^{0}+\sum_{\substack{j=1,\;j\neq i}}^{n}
		\left(d_{ij}^{V}V_j^{0}-d_{ji}^{V}V_i^{0}\right)&=0,
	\end{aligned}
		\right.\qquad i=1,2,\ldots,n.
\end{align}
Now, we define 
\begin{align*}
	S^{0}=(S_1^{0},S_2^{0},\ldots,S_n^{0})^{T},\qquad V^{0}=(V_1^{0},V_2^{0},\ldots,V_n^{0})^{T},
	\end{align*}
	and
	\[\Lambda=(\Lambda_1,\Lambda_2,\ldots,\Lambda_n)^{T}.\]
	Further, let
	\[\phi=\operatorname{diag}(\bar{\phi}_1,\bar{\phi}_2,\ldots ,\bar{\phi}_n).\]
	Then the susceptible subsystem can be written in matrix form as
	\[H_S S^{0}=\Lambda,\]
	where
	\[ H_S=
	\begin{pmatrix}
		\mu_1+\bar{\phi}_1+\displaystyle\sum_{j\neq1}d_{j1}^{S}
		&-d_{12}^{S} & \cdots &-d_{1n}^{S}\\
		-d_{21}^{S} &\mu_2+\bar{\phi}_2+\displaystyle\sum_{j\neq2}d_{j2}^{S}& \cdots & -d_{2n}^{S}\\
		\vdots & \vdots &\ddots & \vdots \\
		-d_{n1}^{S} & -d_{n2}^{S} & \cdots & \mu_n+\bar{\phi}_n+\displaystyle\sum_{j\neq n}d_{jn}^{S}
		\end{pmatrix}.
	\]
	Hence, 
	\[S^{0}=H_S^{-1}\Lambda.\]
	Similarly, the vaccinated subsystem can be expressed as
	\[H_VV^{0}=\phi S^{0},
	\]
	where
	\[
	H_V=
	\begin{pmatrix}
		\mu_1+\displaystyle\sum_{j\neq1}d_{j1}^{V}
		&
		-d_{12}^{V}
		& \cdots &
		-d_{1n}^{V}
		\\
		-d_{21}^{V}
		&
		\mu_2+\displaystyle\sum_{j\neq2}d_{j2}^{V}
		& \cdots &
		-d_{2n}^{V}
		\\
		\vdots & \vdots & \ddots & \vdots
		\\
		-d_{n1}^{V}
		&
		-d_{n2}^{V}
		& \cdots &
		\mu_n+\displaystyle\sum_{j\neq n}d_{jn}^{V}
	\end{pmatrix}.
	\]
	Therefore,
	\[
	V^{0}=H_V^{-1}\phi S^{0}.
	\]
	Consequently, the disease-free equilibrium of system \eqref{model} is given by
	\[
	\mathcal{E}^{0}=
	\left(
	S_1^{0},V_1^{0},0,0,
	\ldots,
	S_n^{0},V_n^{0},0,0
	\right),
	\]
	where
	\[
	S^{0}=H_S^{-1}\Lambda,
	\qquad
	V^{0}=H_V^{-1}\phi S^{0}.
	\]
	Observe that the matrices $H_S$ and $H_V$ possess positive diagonal entries and non-positive off-diagonal entries. Under the standard assumptions that all demographic and migration parameters are nonnegative with $\mu_i>0$, both $H_S$ and $H_V$ are nonsingular $M$-matrices \cite{Berman}. Consequently,
	\[
	H_S^{-1}\geq0,
	\qquad
	H_V^{-1}\geq0,
	\]
	which guarantees that the disease-free equilibrium exists uniquely and remains biologically feasible.
	Define the total population in patch \(i\) by \[N_i(t)=S_i(t)+V_i(t)+I_i(t)+R_i(t),\]
	and the total population over all patches by
	\[N(t)=\sum_{i=1}^{n}N_i(t).\]
	Adding all equations of system \eqref{model}, we obtain
	\[
	\begin{aligned}
		\frac{dN(t)}{dt}&=\sum_{i=1}^{n}\Lambda_i-\sum_{i=1}^{n}\mu_i\left(S_i+V_i+I_i+R_i \right)-\sum_{i=1}^{n}\delta_iI_i\\ &\quad+\sum_{i=1}^{n}\sum_{\substack{j=1,\;j\neq i}}^{n}\Big[ (d_{ij}^{S}S_j-d_{ji}^{S}S_i)+(d_{ij}^{V}V_j-d_{ji}^{V}V_i)\\&\quad\quad+(d_{ij}^{I}I_j-d_{ji}^{I}I_i)+(d_{ij}^{R}R_j-d_{ji}^{R}R_i)\Big] 
	\end{aligned}
	\]
	Since the migration terms cancel pairwise when summed over all patches, we have \[\frac{dN(t)}{dt}=\sum_{i=1}^{n}\Lambda_i-\sum_{i=1}^{n}\mu_i N_i-\sum_{i=1}^{n}\delta_iI_i\]
	Let \[\tilde\Lambda=\sum_{i=1}^{n}\Lambda_i,\quad \mu^{*}=\min_{1\leq i\leq n}\mu_i.\]
	Since \(\delta_iI_i\geq 0\), it follows that \[\frac{dN(t)}{dt}\leq \tilde\Lambda-\mu^{*}N(t).\]
	By the comparison principle, 
	\[
	\limsup_{t\to\infty}N(t)\leq \frac {\tilde\Lambda}{\mu^{*}}.
	\]
Let \(S=(S_1,\ldots,S_n)^T\) and \(V=(V_1,\ldots,V_n)^T\). Since \(H_SS^0=\Lambda\) and the infection terms are nonnegative,

$$
\frac{dS}{dt}\leq\Lambda-H_SS=H_S(S^0-S).
$$
Because \(-H_S\) is a stable Metzler matrix, {the comparison principle \cite{Smith}} gives \(0\leq S(t)\leq S^0\), provided that \(0\leq S(0)\leq S^0\). Similarly, using \(H_VV^0=\phi S^0\) and \(S(t)\leq S^0\),

$$
\frac{dV}{dt}\leq\phi S^0-H_VV=H_V(V^0-V),
$$
and hence \(0\leq V(t)\leq V^0\), provided that \(0\leq V(0)\leq V^0\). Therefore, the feasible region is
$$
\Omega=
\left\{
(S,V,I,R)\in\mathbb{R}_{+}^{4n}:
N(t)\leq\frac{\widetilde{\Lambda}}{\mu^{*}},
\quad 0\leq S_i\leq S_i^0,
\quad 0\leq V_i\leq V_i^0,
\quad i=1,\ldots,n
\right\}.
$$
For initial conditions in \(\Omega\), this compact region is positively invariant. Hence, the model solutions remain nonnegative and bounded for all \(t\geq0\), ensuring that the model is mathematically and epidemiologically well posed.

\subsection{Basic Reproduction Number}\label{reproduction}
	The basic reproduction number is derived using the next-generation matrix method \cite{Van}. %At the disease-free equilibrium,
%	\[
%	E^{0}=
%	\left(
%	S_1^{0},V_1^{0},0,0,\ldots,S_n^{0},V_n^{0},0,0
%	\right).
%	\]
The vector of infected variables is
	\[
	I=(I_1,I_2,\ldots,I_n)^T.
	\]
From the infected equation of the model \eqref{model}, we have
	\[
	\dot{I}_i
	=
	\lambda_iS_i+(1-\epsilon_i)\lambda_iV_i
	-(\gamma_i+\delta_i+\mu_i)I_i
	+\sum_{\substack{j=1,\;j\neq i}}^{n}
	\left(
	d_{ij}^{I}I_j-d_{ji}^{I}I_i
	\right).
	\]
The force of infection in patch \(i\) is
	\[
	\lambda_i
	=
	\beta_i
	\left(
	I_i+\sum_{\substack{j=1,\;j\neq i}}^{n}\rho_{ij}I_j
	\right).
	\]
	At the disease-free equilibrium \((\mathcal{E}^0)\), all infectious variables are zero and hence \(\lambda_i = 0\). However, to construct the next-generation matrix, the new-infection terms are linearized around the disease-free equilibrium by treating the infectious variables as small perturbations. Evaluating the susceptible and vaccinated populations at their disease-free values gives the following linearized new-infection term in patch \(i\):
	\begin{align}
	\mathcal{F}_i
	=
	\beta_i
	\left(S_i^0+(1-\epsilon_i)V_i^0\right)
	\left(
	I_i+\sum_{\substack{j=1,\;j\neq i}}^{n}\rho_{ij}I_j
	\right).
	\end{align}
	Let,
	\[
	A_i=S_i^0+(1-\epsilon_i)V_i^0.
	\]
	Then
	\[
	\mathcal{F}_i
	=
	\beta_i A_i
	\left(
	I_i+\sum_{\substack{j=1,\;j\neq i}}^{n}\rho_{ij}I_j
	\right).
	\]
	Therefore, the new infection matrix \(\mathcal{F}\) is given by
	\[
	\mathcal{F}=
	\begin{pmatrix}
		\beta_1A_1
		&
		\beta_1A_1\rho_{12}
		&
		\cdots
		&
		\beta_1A_1\rho_{1n}
		\\
		\beta_2A_2\rho_{21}
		&
		\beta_2A_2
		&
		\cdots
		&
		\beta_2A_2\rho_{2n}
		\\
		\vdots & \vdots & \ddots & \vdots
		\\
		\beta_nA_n\rho_{n1}
		&
		\beta_nA_n\rho_{n2}
		&
		\cdots
		&
		\beta_nA_n
	\end{pmatrix}.
	\]
	Next, the transition matrix \(\mathcal{V}\), describing recovery, disease-induced death, natural death, and movement of infectious individuals, is
	\[
	\mathcal{V}=
	\begin{pmatrix}
		\gamma_1+\delta_1+\mu_1+\displaystyle\sum_{j\neq1}d_{j1}^{I}
		&
		-d_{12}^{I}
		&
		\cdots
		&
		-d_{1n}^{I}
		\\
		-d_{21}^{I}
		&
		\gamma_2+\delta_2+\mu_2+\displaystyle\sum_{j\neq2}d_{j2}^{I}
		&
		\cdots
		&
		-d_{2n}^{I}
		\\
		\vdots & \vdots & \ddots & \vdots
		\\
		-d_{n1}^{I}
		&
		-d_{n2}^{I}
		&
		\cdots
		&
		\gamma_n+\delta_n+\mu_n+\displaystyle\sum_{j\neq n}d_{jn}^{I}
	\end{pmatrix}.
	\]
Let \(v_{ij}\) denote the \((i,j)\)-th entry of \(\mathcal{V}\).
For each column \(j\), \(v_{jj}
=
\gamma_j+\delta_j+\mu_j
+\sum_{\substack{i=1\\i\neq j}}^{n}d_{ij}^{I}\), while
\(
\sum_{\substack{i=1\\i\neq j}}^{n}|v_{ij}|
=
\sum_{\substack{i=1\\i\neq j}}^{n}d_{ij}^{I}.
\)
Therefore,
\(
v_{jj}
-
\sum_{\substack{i=1\\i\neq j}}^{n}|v_{ij}|
=
\gamma_j+\delta_j+\mu_j>0.
\)
Hence, \(\mathcal{V}\) is a strictly column-diagonally dominant
matrix with positive diagonal entries and nonpositive off-diagonal entries, that is, a \(Z\)-matrix. Therefore,
it is a nonsingular \(M\)-matrix, and consequently, \(\mathcal{V}^{-1}\geq 0\).
	The next-generation matrix is therefore
	\[
	K=\mathcal{F}\mathcal{V}^{-1}.
	\]
	Consequently, the global basic reproduction number of the connected multi-patch system is defined by
	\begin{align}
		\mathcal{R}_0=\rho(\mathcal{F}\mathcal{V}^{-1}),	
	\end{align}
	where \(\rho\) denotes the spectral radius.
	To examine the transmission potential within an individual patch, we now consider the local basic reproduction number. This corresponds to the situation where inter-patch movement and cross-patch infection are ignored. Hence, for an isolated patch \(i\), where mobility and cross-patch transmission are neglected, the infected equation reduces to 
	\[\dot I_i=\beta_i\left(S_i+(1-\epsilon_i)V_i\right)I_i-(\gamma_i+\delta_i+\mu_i)I_i. \]
	At the disease-free equilibrium of patch \(i\), \(S_i=S_i^{0},\) and \(V_i=V_i^{0}.\)
	Applying the next-generation matrix approach \cite{Van}, the local basic reproduction number for patch \(i\) is obtained as 
	\begin{align}
		\mathcal{R}_{0}^{i}=\frac{\beta_i\left(S_i^0+(1-\epsilon_i)V_i^0\right)}{\gamma_i+\delta_i+\mu_i},\;i=1, 2,\ldots,n.
		\end{align}
	The quantity \(\mathcal{R}_{0}^{i}\) measures the transmission potential of measles in patch \(i\) when considered in isolation. If \(\mathcal{R}_{0}^{i}>1\), measles can invade that isolated patch. In the connected multipatch system, however, invasion is governed by the global reproduction number \(\mathcal{R}_{0}\), which depends jointly on local transmission, infectious movement, and cross-patch infection.
\begin{comment}
	\begin{Th}
		Suppose that \(\gamma_i=\gamma,\;\delta_i=\delta,\;\mu_i=\mu,\;i=1,2,\ldots,n.\) Then \[\min_j c_j \min_i \mathcal{R}_{0}^{i}\leq \mathcal{R}_{0} \leq \max_j c_j \max_i \mathcal{R}_{0}^{i},\] where \[c_j=\sum_{i=1}^{n}C_{ij},\qquad C_{ij}=
		\begin{cases}
			1, & i=j,\\
			\rho_{ij}, & i\neq j.
		\end{cases}\]
		In particular, if \(\sum_{i=1}^{n}C_{ij}=1,\) then 
		\(\min_i \mathcal{R}_{0}^{i}\leq \mathcal{R}_{0} \leq  \max_i \mathcal{R}_{0}^{i}.\)
	\end{Th}
	\begin{proof}
		Since \(\mathcal F_{ij}=B_iC_{ij},\) the sum of the \(j\)-th column of \(\mathcal F\) satisfies 
		\[\min_iB_ic_j\leq\sum_{i=1}^{n}B_iC_{ij}\leq\max_iB_ic_j.\]
		Hence, the column sums of \(K=\mathcal{F}\mathcal{V}^{-1}\) satisfy 
		\[\min_jc_j\min_i\mathcal{R}_{0}^{i}\leq\rho(K)\leq\max_jc_j\max_i\mathcal{R}_{0}^{i}.\]
	Moreover, if the coupling coefficients satisfy \(c_j=1,\;j=1,2,\ldots,n,\) then  
	\begin{align}
		\min_i \mathcal{R}_{0}^{i}\leq \mathcal{R}_{0} \leq  \max_i \mathcal{R}_{0}^{i}.
	\end{align}
	Therefore, the global reproduction number is bounded by the local reproduction numbers through the coupling structure of the connected patches. This shows that the overall transmission dynamics depend on both local outbreak intensity and inter-patch interactions.	
	\end{proof}
	Since the disease transmission dynamics are governed by the infectious classes, the stability of the disease-free equilibrium depends on whether small perturbations in the infected population decay or grow with time. 
\end{comment}

\subsection*{Stability Analysis}
\begin{Th}
	The disease-free equilibrium, $\mathcal{E}^{0}=\left(S_1^{0},V_1^{0},0,0,\ldots,S_n^{0},V_n^{0},0,0\right)$ of system \eqref{model} is locally asymptotically stable if \(\mathcal{R}_{0}<1,\) and unstable if \(\mathcal{R}_{0}>1\).
\end{Th}
\begin{proof}
	The linearized infected subsystem at \(\mathcal{E}_0\) is
	\begin{align} 
		\dot{I}=(\mathcal{F}-\mathcal{V})I,
	\end{align}
		where $I=(I_1, I_2,\ldots,I_n)^{T},$ \(\mathcal{F}\) denotes the matrix of new infection terms, and \(\mathcal{V}\) represents the transition matrix associated with recovery, natural death, disease-induced death, and migration of infectious individuals.			
	Let \(s(\mathcal{F}-\mathcal{V})\) denote the spectral bound of \(\mathcal{F}-\mathcal{V}\), that is, the largest real part among its eigenvalues. Since \(\mathcal{F}\geq0\) and \(\mathcal{V}\) is a nonsingular \(M\)-matrix, the standard next-generation matrix result~\cite{Van} gives \[ s(\mathcal{F}-\mathcal{V})<0 \quad\Longleftrightarrow\quad \rho(\mathcal{F}\mathcal{V}^{-1})<1, \] and \[ s(\mathcal{F}-\mathcal{V})>0 \quad\Longleftrightarrow\quad \rho(\mathcal{F}\mathcal{V}^{-1})>1. \] Since  \(\mathcal{R}_{0} =\rho(\mathcal{F}\mathcal{V}^{-1}),\) 
	if \(\mathcal{R}_{0}<1,\) then each infected individual generates less than one secondary infection on average. Consequently, all eigenvalues of the matrix $(\mathcal{F}-\mathcal{V})$ have negative real parts, implying that solutions of the linearized infected subsystem decay exponentially to zero. Therefore, the disease-free equilibrium \(\mathcal{E}^{0}\) is locally asymptotically stable.
			
Conversely, if \(\mathcal{R}_{0}>1,\) then infection can invade the population, and at least one eigenvalues of $(\mathcal{F}-\mathcal{V})$ has positive real part. Hence, perturbations around the disease-free equilibrium grow with time, implying that \(\mathcal{E}^{0}\) is unstable. Therefore, the disease-free equilibrium \(\mathcal{E}^{0}\) is locally asymptotically stable whenever \(\mathcal{R}_{0}<1\), and unstable whenever \(\mathcal{R}_{0}>1\).
\end{proof}
\begin{Th}
	Assume that the transmission--movement network is irreducible. If \(\mathcal{R}_0<1\), then the disease-free equilibrium \(\mathcal{E}^{0}\) of the autonomous system associated with \eqref{model} is globally asymptotically stable in the feasible region \(\Omega\).				
\end{Th}
\begin{proof}
Since \(\mathcal{R}_{0}=\rho(\mathcal{F}\mathcal{V}^{-1})\), the matrices \(\mathcal{F}\mathcal{V}^{-1}\) and \(\mathcal{V}^{-1}\mathcal{F}\) have the same spectral radius. By {the Perron--Frobenius theorem \cite{Berman}}, there exists a positive left eigenvector \(\vartheta=(\vartheta_1, \vartheta_2, \ldots, \vartheta_n)\) of the nonnegative irreducible matrix \(\mathcal{V}^{-1}\mathcal{F}\) corresponding to \(\mathcal{R}_{0}\) such that
\[\vartheta \mathcal{V}^{-1}\mathcal{F}=\mathcal{R}_{0}\vartheta .\]
Define \(\kappa=(\kappa_1, \kappa_2,\ldots,\kappa_n)= \vartheta \mathcal{V}^{-1}.\) Since \(\vartheta>0\) and \(\mathcal{V}^{-1}\geq0\), it follows that \(\kappa_i>0,\;i=1, 2, \ldots, n.\)
Now consider the Lyapunov function
\begin{align}
	\mathcal{L}(I)=\sum_{i=1}^{n}\kappa_i I_i,
\end{align}
				
Clearly, \(\mathcal{L}(I)\geq0,\) with equality only when \(I_i=0,\; i=1, 2, \ldots, n.\)
Differentiating \(\mathcal{L}(I)\) along the trajectories of system \eqref{model} yields
\begin{align}
	\dot{\mathcal{L}}=\sum_{i=1}^{n}\kappa_i\dot I_i.
\end{align}
Using the infected subsystem of \eqref{model}, we obtain
\begin{align*}
\dot{\mathcal{L}}=&\sum_{i=1}^{n}\kappa_i\Bigg[\beta_i
\left(S_i+(1-\epsilon_i)V_i \right)
\left(I_i+\sum_{\substack{j=1,\;j\neq i}}^{n}\rho_{ij}I_j \right) \\
&-(\gamma_i+\delta_i+\mu_i)I_i+ \sum_{\substack{j=1,\;j\neq i}}^{n}\left(d_{ij}^{I}I_j-d_{ji}^{I}I_i \right)\Bigg].
\end{align*}	
	Since the feasible region $\Omega$ is positively invariant, we have \(S_i\leq S_i^{0}\)	and \(V_i\leq V_i^{0}\).
	Hence, \(S_i+(1-\epsilon_i)V_i\leq S_i^{0}+(1-\epsilon_i)V_i^{0}\).
	Therefore, \[\dot{\mathcal{L}}\leq\kappa(\mathcal{F}-\mathcal{V})I.\]		
	Using the definition \(\kappa=\vartheta \mathcal{V}^{-1}\),  
	we obtain 
\begin{align*}
	\dot{\mathcal{L}}&\leq \vartheta \mathcal{V}^{-1}(\mathcal{F}-\mathcal{V})I\\ &=(\mathcal{R}_{0}-1)\vartheta I. 
\end{align*}
	Since \(\mathcal{R}_{0}<1\), \(\vartheta>0\), and \(I\geq0\), it follows that \(	\dot{\mathcal{L}}\leq0,\) with equality only when \(I=0\). Hence, by LaSalle's invariance principle 
	\cite{LaSalle}, every solution initiating in \(\Omega\) approaches the 
	infection-free invariant set, and consequently,
	\[
	I_i(t)\longrightarrow0,
	\qquad i=1,2,\ldots,n.
	\]
	Since \(\lambda_i(t)\) depends linearly on the infectious populations, 
	\(I_i(t)\to0\) also implies that \(\lambda_i(t)\to0\). The recovered 
	subsystem can be written as
	\[
	\dot R_i
	=
	\gamma_iI_i-\mu_iR_i
	+\sum_{\substack{j=1,\;j\neq i}}^{n}
	\left(d_{ij}^{R}R_j-d_{ji}^{R}R_i\right).
	\]
	Its coefficient matrix is stable because it has the same nonsingular 
	\(M\)-matrix structure as the other transition matrices. Since its input 
	\(\gamma_iI_i(t)\) tends to zero, it follows that
	\[
	R_i(t)\longrightarrow0,
	\qquad i=1,2,\ldots,n.
	\]
Finally, as \(I_i(t)\to0\) and \(\lambda_i(t)\to0\), the susceptible and 
vaccinated subsystems approach the disease-free subsystem. Since \(H_S\) 
and \(H_V\) are nonsingular \(M\)-matrices, this subsystem is asymptotically 
stable and therefore \(S(t)\longrightarrow S^0,\) and \(V(t)\longrightarrow V^0.\)	Thus, every solution initiating in \(\Omega\) converges to	\(\mathcal{E}^{0}\). Therefore, \(\mathcal{E}^{0}\) is globally asymptotically stable whenever \(\mathcal{R}_{0}<1\).	
\end{proof}
\begin{Th}
	Assume that the mobility coefficients satisfy $d_{ij}^{I}>0,\; i,j=1,2,\ldots,n,\;i\neq j.$ If $\mathcal{R}_{0}>1,$ then system \eqref{model} is uniformly persistent in the infectious classes. More precisely, there exists a positive constant $\xi$ such that every solution with initially positive infection satisfies \[\liminf_{t\to\infty} I_i(t)\geq\xi,\;i=1,2,\ldots,n.\]
	Furthermore, system \eqref{model} possesses at least one endemic equilibrium.			
\end{Th}
			
\begin{proof}
	Consider the state space \[\mathcal{D}=\left\lbrace(S, V, I, R)\in \Omega:S_i, V_i, I_i, R_i\geq0,i=1,2,\ldots,n \right\rbrace, \] together with the interior region \[\mathcal{D}^{\circ}=\left\lbrace(S, V, I, R)\in\mathcal{D}:I_i>0,i=1,2,\ldots,n\right\rbrace . \]
	Its persistence boundary is \(\partial\mathcal{D}^{\circ}=\mathcal{D}\setminus\mathcal{D}^{\circ}.\) The positivity of solutions ensures that both $\mathcal{D}$ and $\mathcal{D}^{\circ}$ are forward invariant under the semiflow generated by system \eqref{model}. Moreover, the previous boundedness results imply that the system possesses a compact attracting set.
	We next characterize the invariant part of the boundary. Suppose that a trajectory remains in \(\partial\mathcal{D}^{\circ}\). If there exists some patch $i_0$ and some time $t_0>0$ such that \[I_{i_0}(t_0)>0,\]
	then the positivity of the migration coefficients $d_{ij}^I$ implies that infectious individuals spread to the other patches. Indeed, for any patch $j\neq i_0$ satisfying $I_{j}(t_0)=0$, 
%	\[ \left.\frac{dI_j}{dt}\right|_{t=t_0} = \sum_{\substack{k=1,\;k\neq j}}^{n} d_{jk}^{I}I_k(t_0) \geq d_{ji_0}^{I}I_{i_0}(t_0)>0. \]
	\[\frac{dI_j(t_0)}{dt}=\sum_{\substack{k=1,\;k\neq j}}^{n}d_{jk}^{I}I_k(t_0)\geq d_{ji_0}^{I}I_{i_0}(t_0)>0.\]
	Hence, $I_j(t)>0$ for $t>t_0$, sufficiently close to $t_0.$ Repeating the argument for all patches shows that infection eventually becomes positive throughout the network, contradicting the assumption that the orbit stays on the boundary. Therefore, every invariant orbit contained in \(\partial\mathcal{D}^{\circ}\) must satisfy $I_i(t)\equiv 0,\;i=1,2,\ldots,n.$
	Consequently, the only equilibrium on the boundary is the disease-free equilibrium \[\mathcal{E}^{0}=\left( S_1^{0}, V_1^{0},0,0,\ldots,S_n^{0},V_n^{0},0,0\right) .\]
	Moreover, every orbit that remains on the boundary converges to \(\mathcal{E}^{0}\), because the corresponding disease-free subsystem has \(\mathcal{E}^{0}\) as its unique globally attracting equilibrium. We now show that $\mathcal{E}^{0}$ repels trajectories originating in $\mathcal D^{\circ}.$ Suppose, contrary to our claim, that there exists a solution $\varphi_t(x_0),\;x_0\in \mathcal D^{\circ},$ such that $\varphi_t(x_0)\to \mathcal{E}^{0}$ as $t\to\infty.$ Then for sufficiently small $\Psi>0,$ there exists $T>0$ such that
	\[S_i(t)>S_i^{0}-\Psi,\qquad V_i(t)>V_i^{0}-\Psi,\qquad t>T. \]
	Using these estimates in the infected subsystem yields	
\begin{align*}
	\dot{I_i}\geq& \beta_i\left[ S_i^{0}+(1-\epsilon_i)V_i^{0}-\Psi(2-\epsilon_i)\right]\left(I_i+\sum_{\substack{j=1,\;j\neq i}}^{n}\rho_{ij}I_j \right)\\
	&-(\gamma_i+\delta_i+\mu_i)I_i+\sum_{\substack{j=1,\;j\neq i}}^{n} \left( d_{ij}^II_j-d_{ji}^{I}I_i.\right) 
\end{align*}
	This inequality may be written in matrix form as
	\[\dot I\geq(\mathcal F-\mathcal V-\Psi\mathcal C)I,\] where $\mathcal C$ is a nonnegative matrix. Since $\mathcal{R}_0>1,$ we have $s(\mathcal F-\mathcal V)>0.$ By continuity of the spectral bound, one can choose $\Psi>0$ sufficiently small such that
	\(s(\mathcal F-\mathcal V-\Psi\mathcal C)>0.\) For sufficiently small \(\Psi\), the matrix \(\mathcal F-\mathcal V-\Psi\mathcal C\) is irreducible and Metzler. Since \(I(T)\gg0\), the comparison principle implies that \(I(t)\) is bounded below by a positive exponentially growing solution of the corresponding comparison system. Therefore, \(I(t)\) cannot converge to zero. This contradicts the assumption that the trajectory approaches $\mathcal{E}^{0}.$ Hence
	\[W^s(\mathcal{E}^{0})\cap\mathcal{D}^{\circ}=\emptyset,\] where $W^s(\mathcal{E}^{0})$ denotes the stable set of $\mathcal{E}^{0}.$ Since \(\mathcal{E}^{0}\) is isolated, every invariant boundary orbit converges to it, and no orbit starting in \(\mathcal{D}^{\circ}\) approaches it, the persistence theorem of \cite{Wang2004} guarantees that the system is uniformly persistent. Consequently, there exists a constant \(\xi>0\) such that \[\liminf_{t\to\infty}I_i(t)\geq \xi,\qquad i=1,2,\ldots,n.\]
		
	Furthermore, the uniform persistence property and ultimate boundedness of the system imply the existence of at least one endemic equilibrium
	\begin{align}
		\mathcal{E}^*=\left(S_1^*, V_1^*,I_1^*,R_1^*,\ldots, S_n^*, V_n^*,I_n^*,R_n^*\right), 
	\end{align}
	with \[I_i^*>0,\qquad i=1,2,\ldots,n.\]
	This completes the proof.
\end{proof}

	\section{Numerical Analysis and Discussion}\label{numerical}
		Given the substantial geographical variation in the outbreak, the numerical analysis proceeds in two stages to distinguish locally driven transmission from movement-mediated effects.
	In the first stage, the eight administrative divisions of Bangladesh--Dhaka, Chattogram, Rajshahi, Barishal, Mymensingh, Rangpur, Khulna, and Sylhet--are treated independently by setting the compartment-specific movement rates, \(d_{ij}^X, X\in \left\lbrace S, V, I, R\right\rbrace \) and cross-patch infection coefficients, $\rho_{ij},$ to zero. In the second stage, the calibrated divisional models are connected within a multipatch framework to assess  how interdivisional connectivity alters epidemic peaks, peak timing, cumulative infections, and division-specific disease burden. We begin with the division-specific analysis, for which figures ~\ref{division1} and ~\ref{division2} present the observed vaccination and measles dynamics together with the corresponding model fits.\\
	For each division, the SVIR model (\ref{model}) is fitted to the observed daily measles cases using bounded nonlinear least squares. Daily vaccination data are smoothed using a seven day moving average and normalized to construct the time-dependent vaccination rate, $\phi(t),$ with $\alpha$ estimated as its scaling parameter. The parameters $\beta, \delta, \epsilon,$ and $\alpha$ are estimated, while the remaining parameters are fixed at their specified values (see table~\ref{tab1}). 
	Parameter uncertainty is assessed using 1000 parametric bootstrap replicates. In each replicate, a synthetic case series is generated from a negative-binomial distribution based on the fitted case trajectory, and the model parameters are re-estimated. The DFE and corresponding reproduction number are then recalculated for each valid replicate. The 95\% confidence intervals are determined from the 2.5th and 97.5th percentiles of the resulting bootstrap distributions. \\
	Figures ~\ref{division1} and ~\ref{division2} present the division-specific vaccination patterns, observed and fitted daily measles cases, and residuals obtained without spatial coupling. Across all eight divisions, vaccination activity increases rapidly during the second half of April, peaks between late April and early May, and declines markedly by mid-May. The seven day smoothed daily vaccination series peaks at approximately $2.6-2.8\times10^5$ in Dhaka and Chattogram (figure~\ref{division1}a,b), $1.43\times10^5$ in Rangpur (figure~\ref{division2}b), $1.0-1.3\times10^5$ in Rajshahi, Khulna (figure~\ref{division1}c,d), Mymensingh, and Sylhet (figure~\ref{division2}c,d), and $5.91\times10^4$ in Barishal (figure~\ref{division2}a). Nevertheless, the measles patterns remain heterogeneous. Dhaka records the highest burden, reaching approximately 200 daily cases in May, while Chattogram shows an increasing trend from June onward and peaks at about 60 cases. Rajshahi shows substantial early activity followed by later peaks approaching 30 cases. Smaller late peaks occur in Barishal, Sylhet, Rangpur, and Khulna, reaching approximately 35, 22, 21, and 16 daily cases, respectively, whereas Mymensingh generally remains below 15 cases per day. Thus, despite the broadly synchronized vaccination campaign, the magnitude, timing, and persistence of measles transmission vary substantially among divisions.
	
	\begin{figure}[H]
		\centering
		
		\begin{subfigure}[t]{0.48\textwidth}
			\centering
			\includegraphics[width=\linewidth]{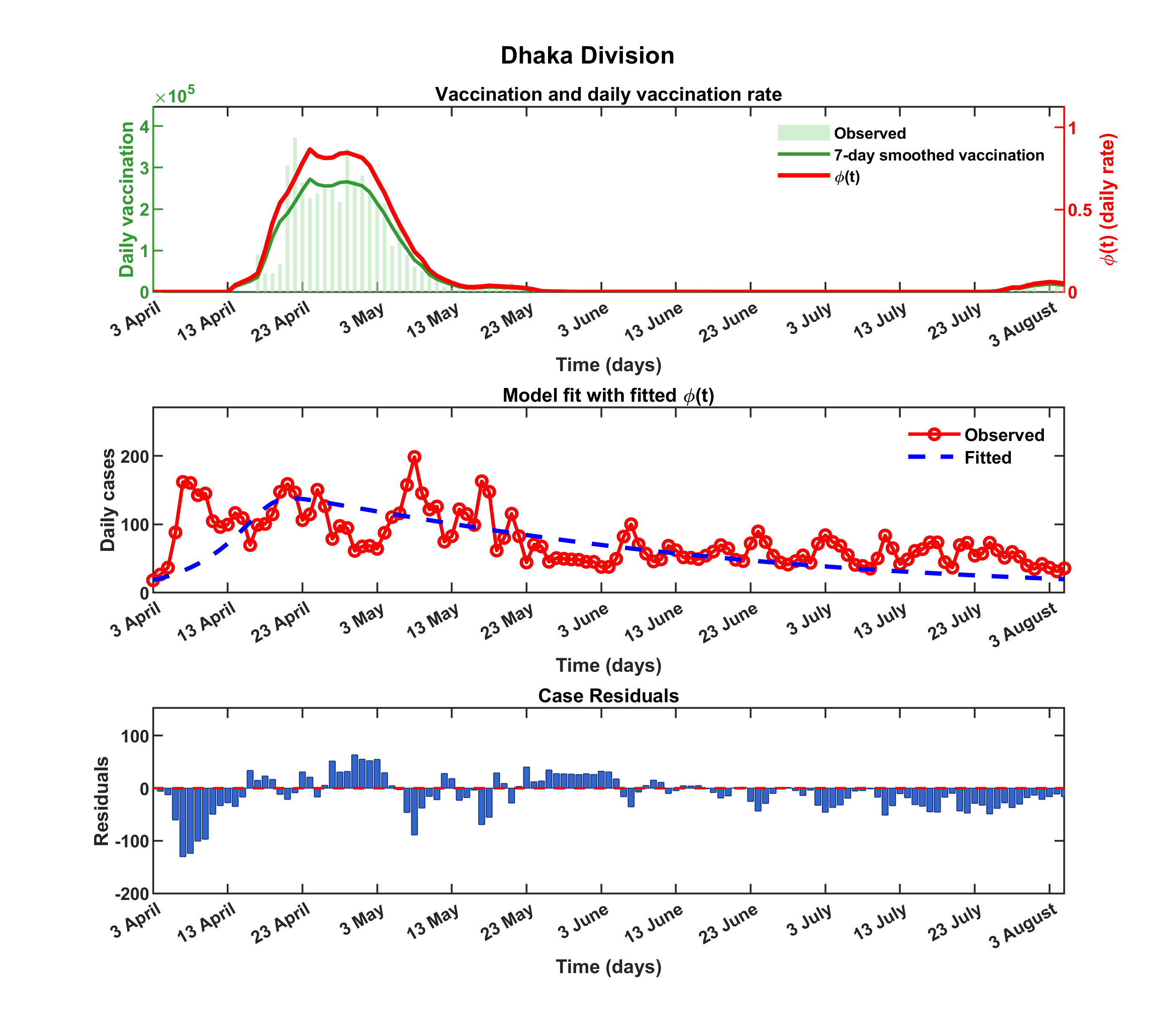}
			\caption{Dhaka Division}
			\label{dhaka}
		\end{subfigure}
		\hfill
		\begin{subfigure}[t]{0.48\textwidth}
			\centering
			\includegraphics[width=\linewidth]{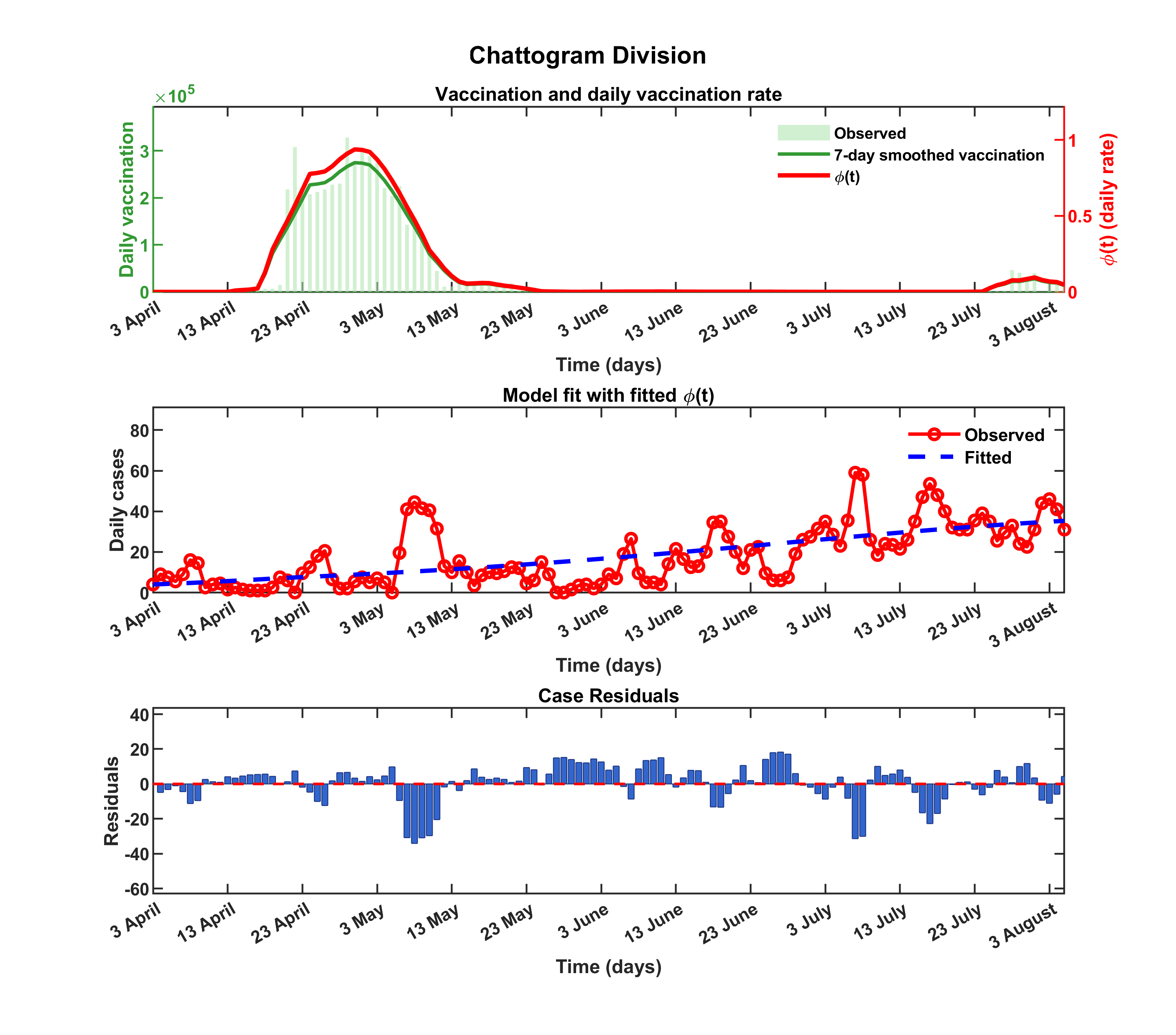}
			\caption{Chattogram Division}
			\label{chattogram}
		\end{subfigure}
		
		\vspace{0.4cm}
		
		% Second row
		\begin{subfigure}[t]{0.48\textwidth}
			\centering
			\includegraphics[width=\linewidth]{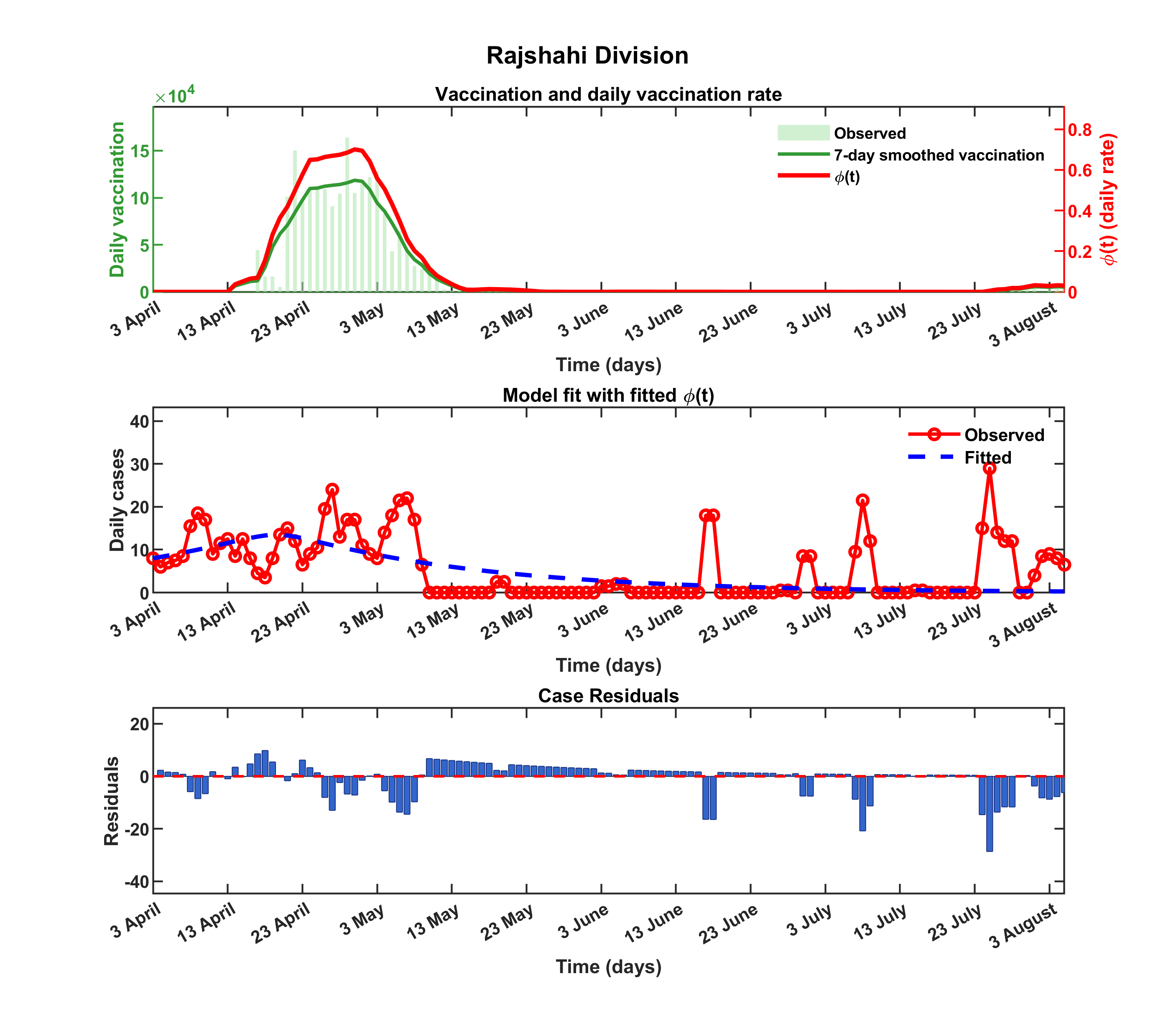}
			\caption{Rajshahi Division}
			\label{rajshahi}
		\end{subfigure}
		\hfill
		\begin{subfigure}[t]{0.48\textwidth}
			\centering
			\includegraphics[width=\linewidth]{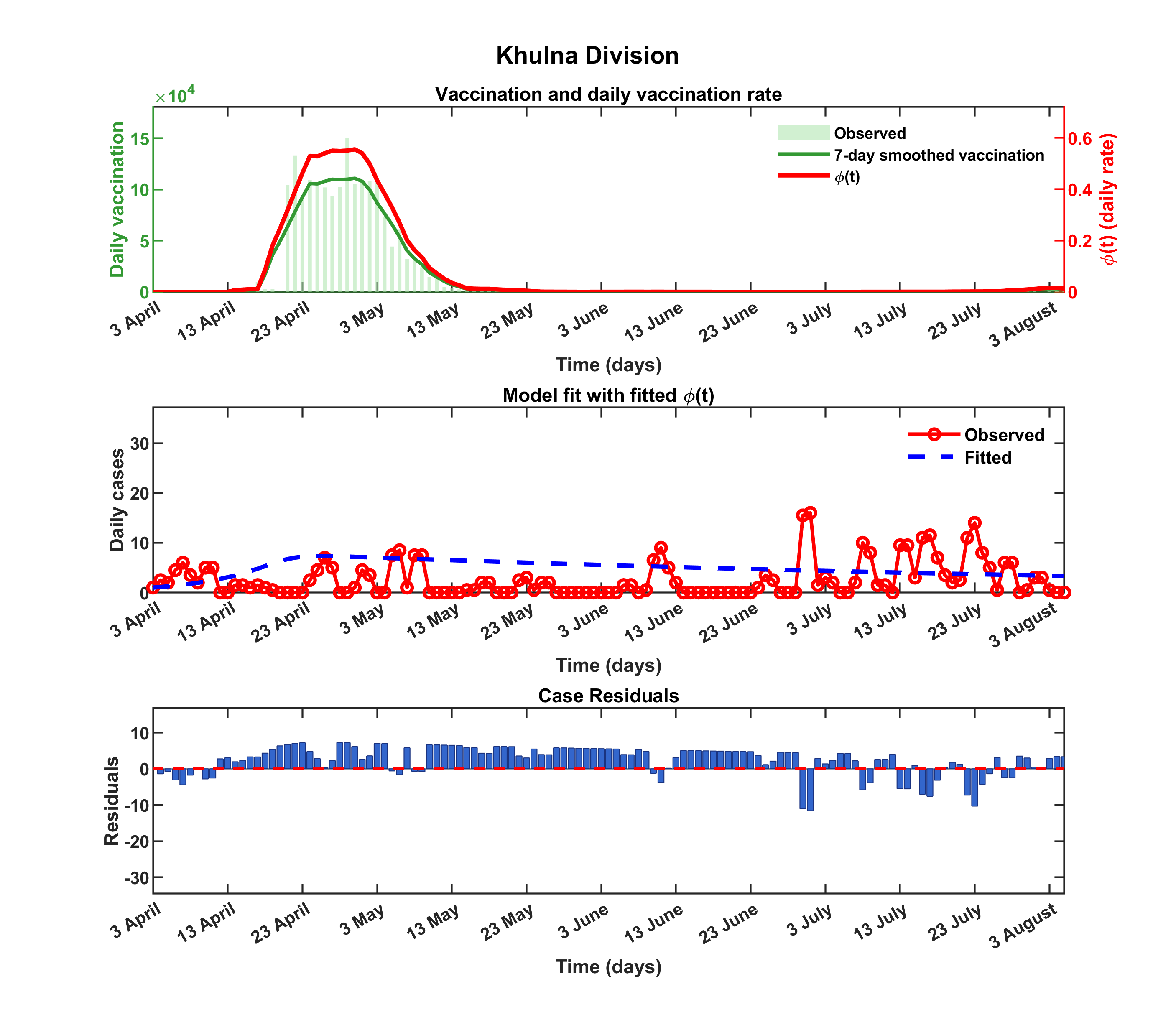}
			\caption{Khulna Division}
			\label{khulna}
		\end{subfigure}
		
		\caption{Division-specific measles dynamics and model-fitting results:
			(a) Dhaka, (b) Chattogram, (c) Rajshahi, and (d) Khulna.}
		\label{division1}
	\end{figure}
	
	The fitted trajectories reproduce the principal division-specific trends but remain smoother than the observed daily series and therefore do not fully capture isolated peaks and short-term case clusters. The residual deviations are generally moderate but become larger during the early high-burden period in Dhaka, the May and July peaks in Chattogram, and the later increases in Rajshahi, Barishal, Rangpur, and Sylhet. These short-term deviations may reflect localized transmission clusters, reporting delays, retrospective case entry, or other processes not explicitly represented in the deterministic model.\\
	The estimated local reproduction numbers show substantial geographical variation, ranging from 1.1386 in Khulna to 4.3492 in Dhaka (table~\ref{tab1}b). All point estimates exceed one, indicating the potential for sustained local transmission. The division-specific reproduction number estimates and their 95\% confidence intervals are reported in (table~\ref{tab1}b).
	\begin{figure}[H]
		\centering
		
		% First row
		\begin{subfigure}[t]{0.48\textwidth}
			\centering
			\includegraphics[width=\linewidth]{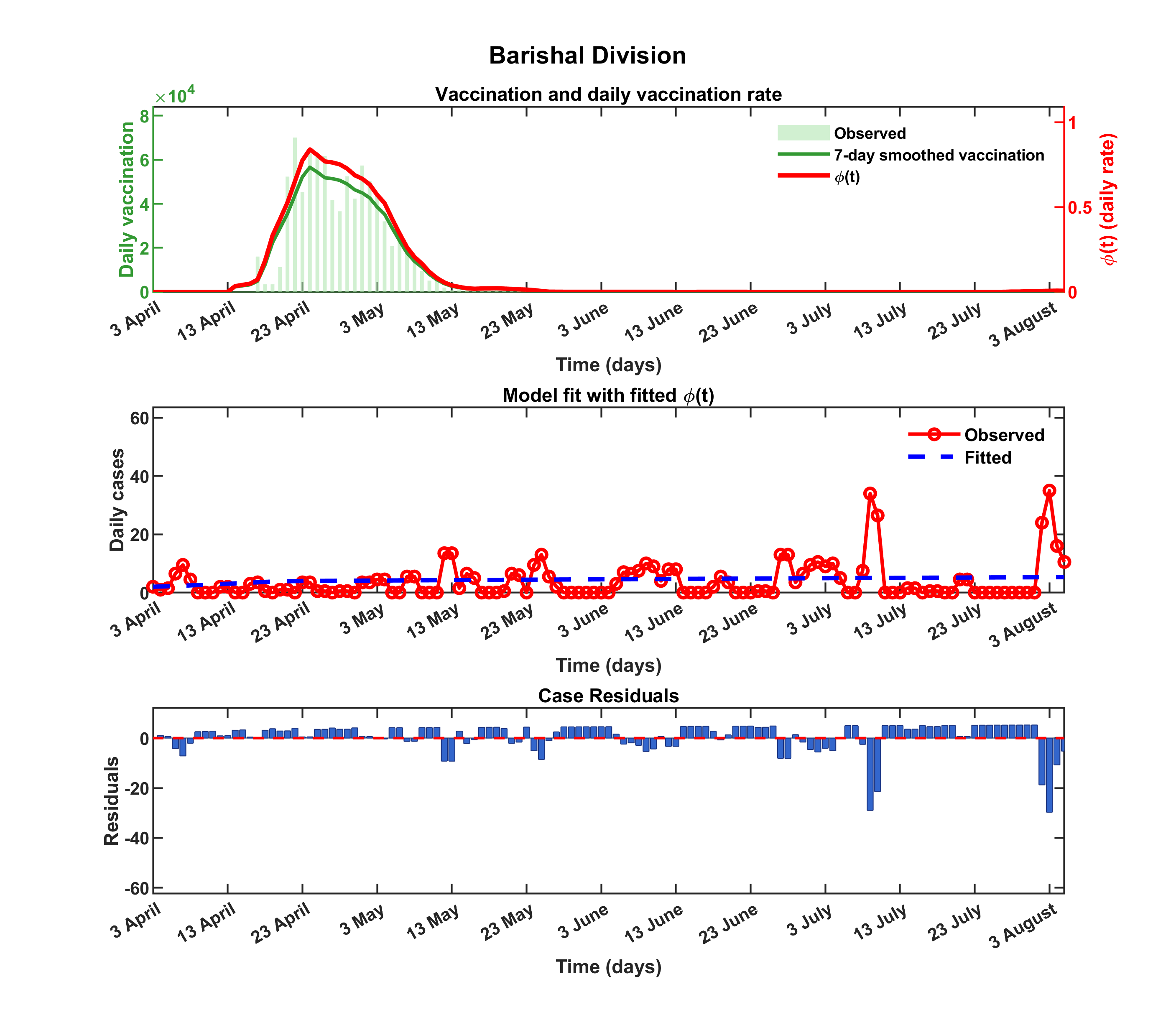}
			\caption{Barishal Division}
			\label{barishal}
		\end{subfigure}
		\hfill
		\begin{subfigure}[t]{0.48\textwidth}
			\centering
			\includegraphics[width=\linewidth]{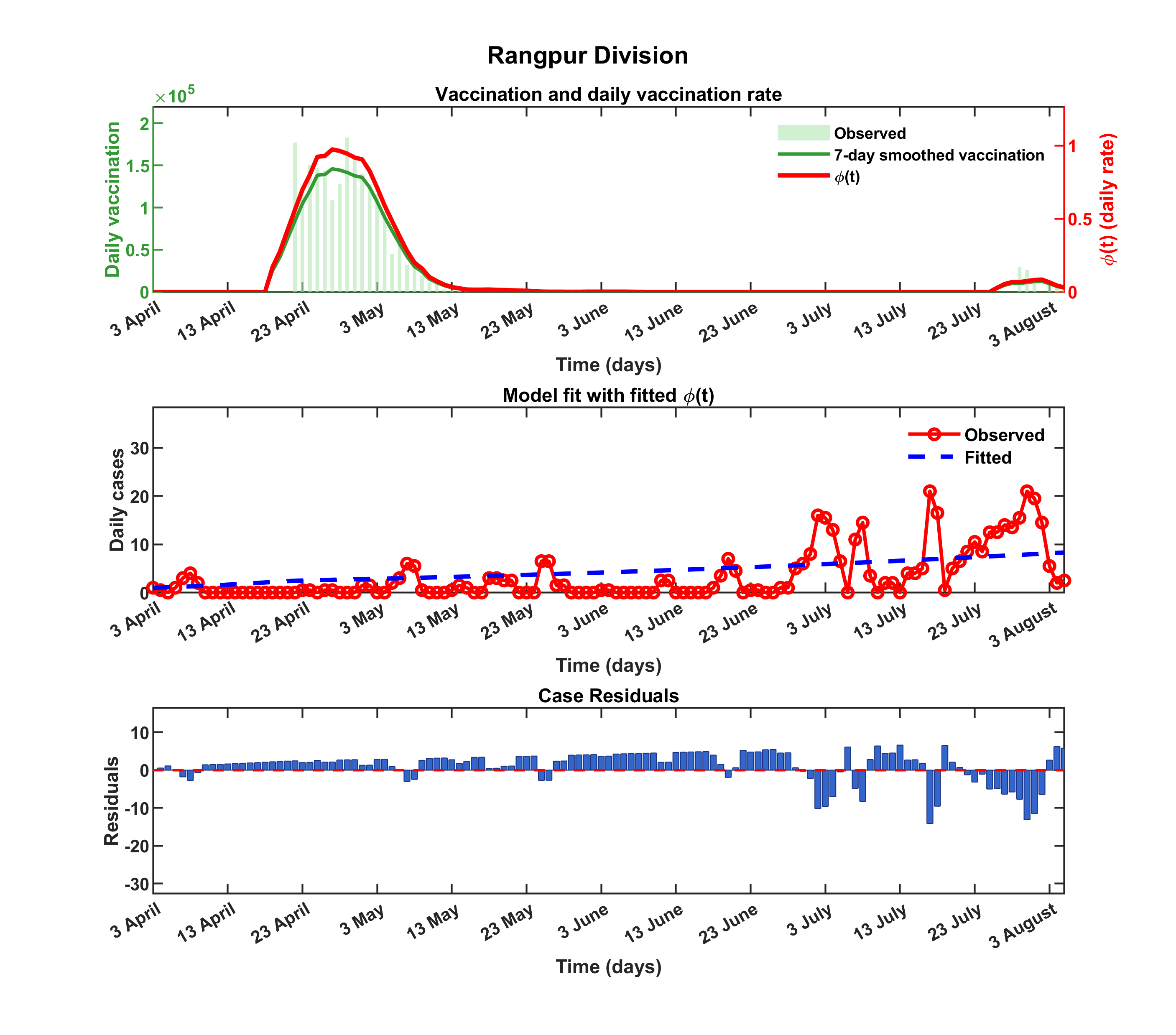}
			\caption{Rangpur Division}
			\label{rangpur}
		\end{subfigure}
		
		\vspace{0.4cm}
		
		% Second row
		\begin{subfigure}[t]{0.48\textwidth}
			\centering
			\includegraphics[width=\linewidth]{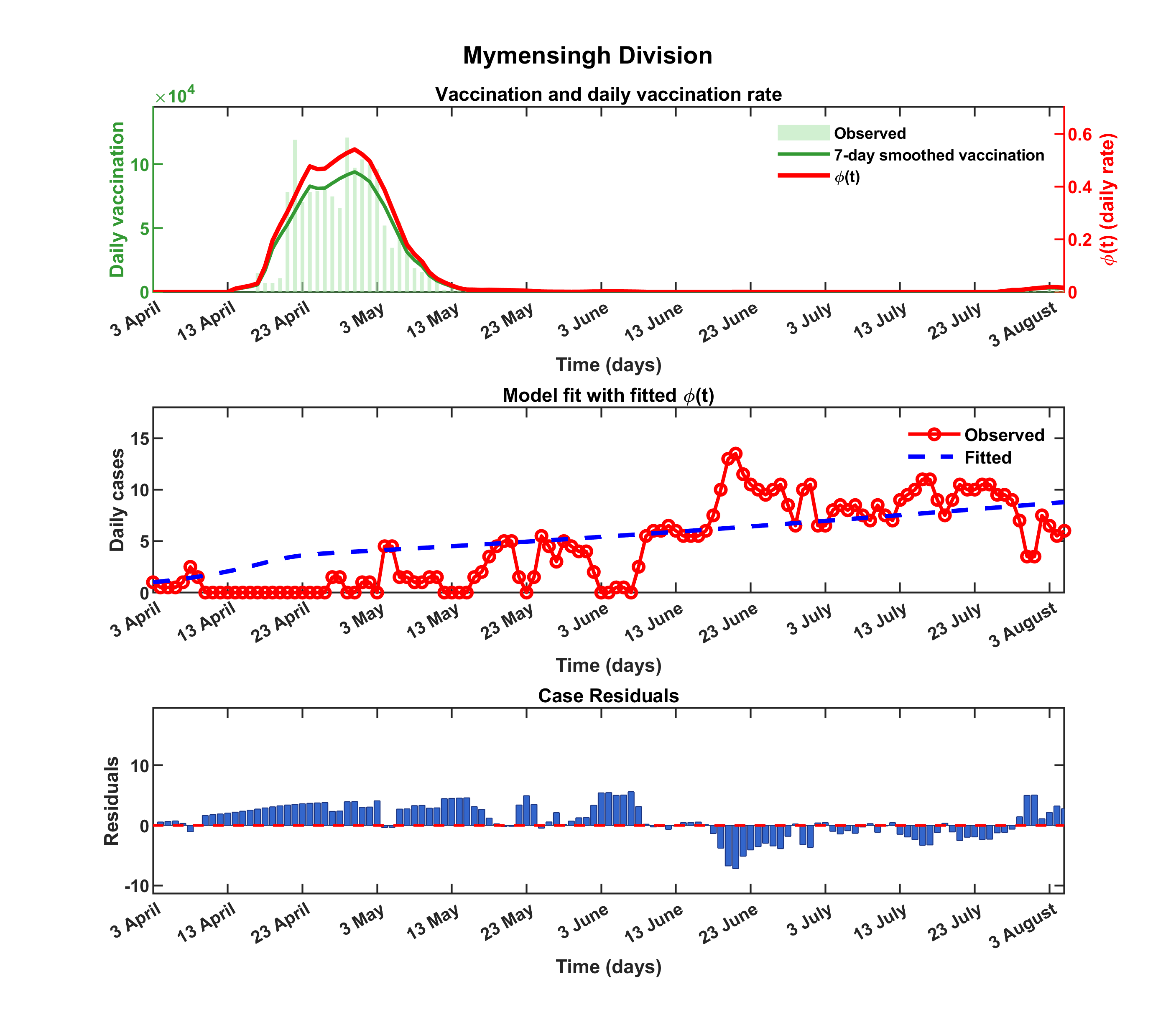}
			\caption{Mymensingh Division}
			\label{mymensingh}
		\end{subfigure}
		\hfill
		\begin{subfigure}[t]{0.48\textwidth}
			\centering
			\includegraphics[width=\linewidth]{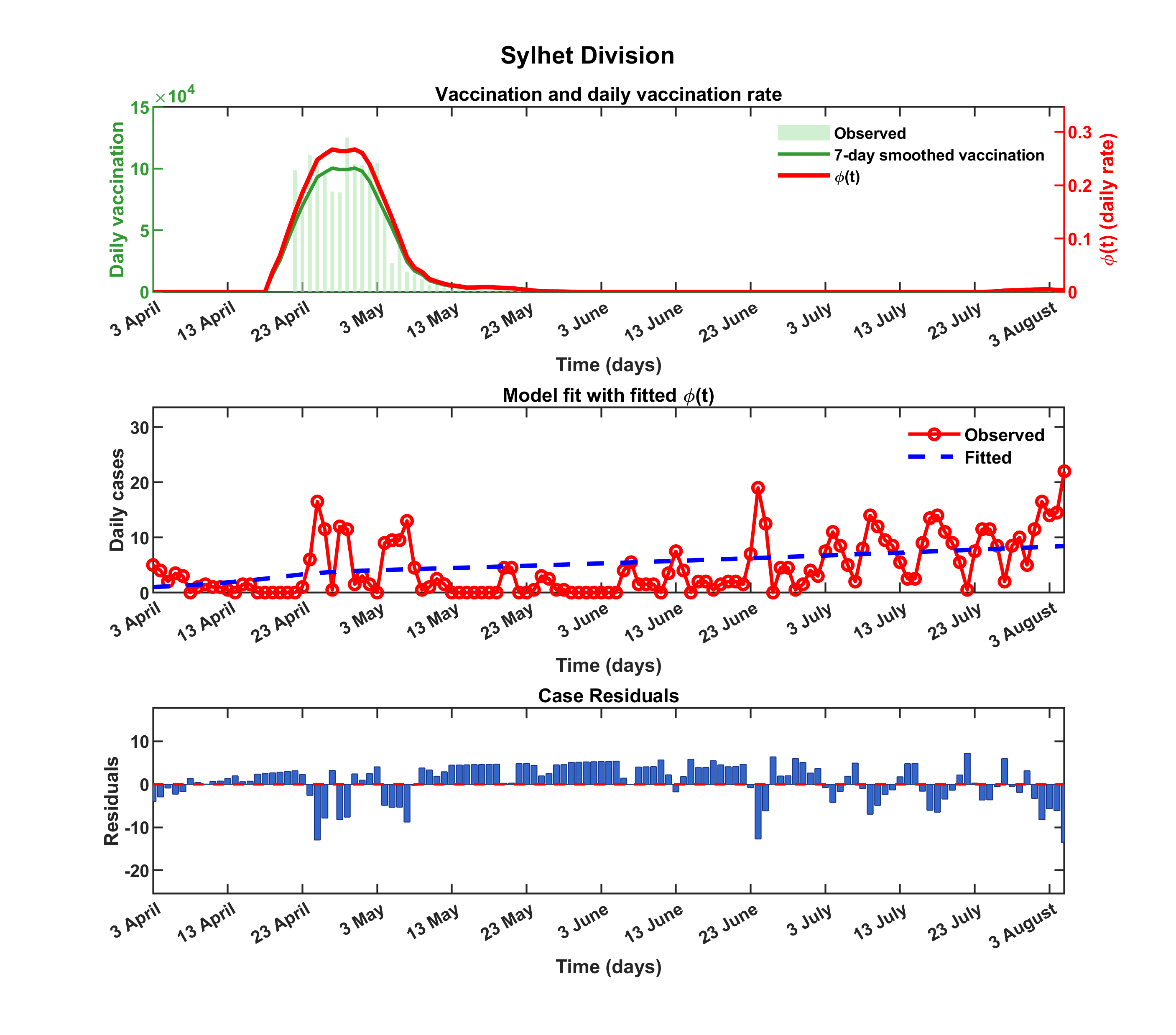}
			\caption{Sylhet Division}
			\label{sylhet}
		\end{subfigure}
		
		\caption{Division-specific measles dynamics and model-fitting results:
			(a) Barishal, (b) Rangpur, (c) Mymensingh, and (d) Sylhet.}
		\label{division2}
	\end{figure}
	
	 As local transmission alone may not fully explain the observed geographical patterns, we next incorporate interdivisional movement and cross-patch infection into the calibrated division-specific models. The fitted epidemiological parameters, initial conditions, and time-dependent vaccination rates are retained for each division, while spatial connections are introduced to represent interactions among the eight divisions.\\
	 A relative connectivity matrix, $Q,$ is constructed using divisional populations, geographical distances, and the greater receiving and outgoing activity of Dhaka and Chattogram as major national hubs (table~\ref{tab1}, table~\ref{tab2}). Under the adopted convention, $Q(i, j)$ represents connectivity from source division $j$ to receiving division $i.$ In the absence of direct compartment-specific mobility and interdivisional contact data, the compartment-specific movement matrices are assumed to be 
	 \begin{align}
	 	d^S=0.003\,Q, d^V=0.0025\,Q, d^I=0.0008\,Q, d^R=0.002\,Q,
     \end{align} 
 describing the physical transfer of susceptible, vaccinated, infectious, and recovered individuals, respectively. This formulation follows the metapopulation framework of Wang et al. \cite{Wang2012}, in which mobility is incorporated separately across epidemiological states, and is further motivated by the observation that distinct subpopulations may exhibit heterogeneous mobility patterns \cite{apolloni}. The assigned scaling coefficients are model assumptions representing compartment-specific differences in mobility, with a lower movement rate assigned to infectious individuals \((d^I)\) to reflect the expected reduction in mobility during illness. In contrast, the additional infection pressure exerted by infectious individuals in source division $j$ on receiving division $i$ is modeled as
	 \begin{align}
	  \rho_{ij}=0.08\,Q(i,j)\left( \frac{\mathcal{R}_0^j}{\displaystyle\max_{k} \mathcal{R}_0^k}
	  \right).
	  \end{align} 
	  This term captures the coupling between spatial connectivity and the transmissibility of the source patch; the force of infection imported from \(j\) into \(i\) scales with the mobility link \(Q(i, j)\) and is modulated by the relative reproduction number of the source division, \(\mathcal{R}_0^j/\displaystyle\max_{k} \mathcal{R}_0^k\). This structure is motivated by metapopulation studies showing that spatial connectivity, mobility, and local reproduction numbers jointly influence epidemic spread across interconnected communities \cite{Lieberthal}. The movement and infection-influence coefficients are treated as baseline modeling assumptions and are varied across the restriction and movement-increase scenarios to assess the robustness of the simulated outcomes.
	  	\begin{figure}[H]
	  	\centering
	  	
	  	% First row
	  	\begin{subfigure}[t]{0.49\textwidth}
	  		\centering
	  		\includegraphics[width=\linewidth]{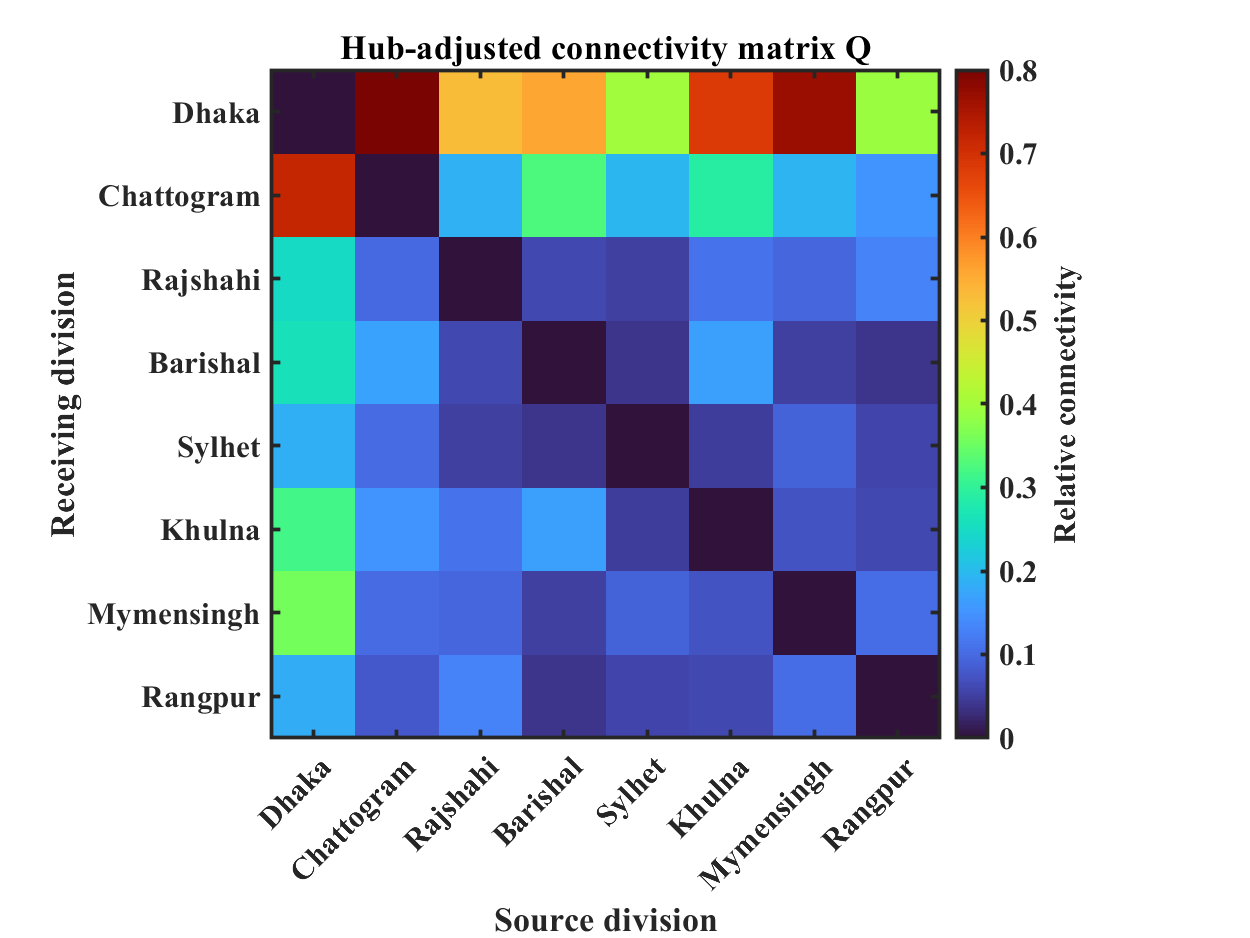}
	  		\caption{}
	  		\label{Q}
	  	\end{subfigure}
	  	\hspace{0.0008\textwidth}
	  	\begin{subfigure}[t]{0.49\textwidth}
	  		\centering
	  		\includegraphics[width=\linewidth]{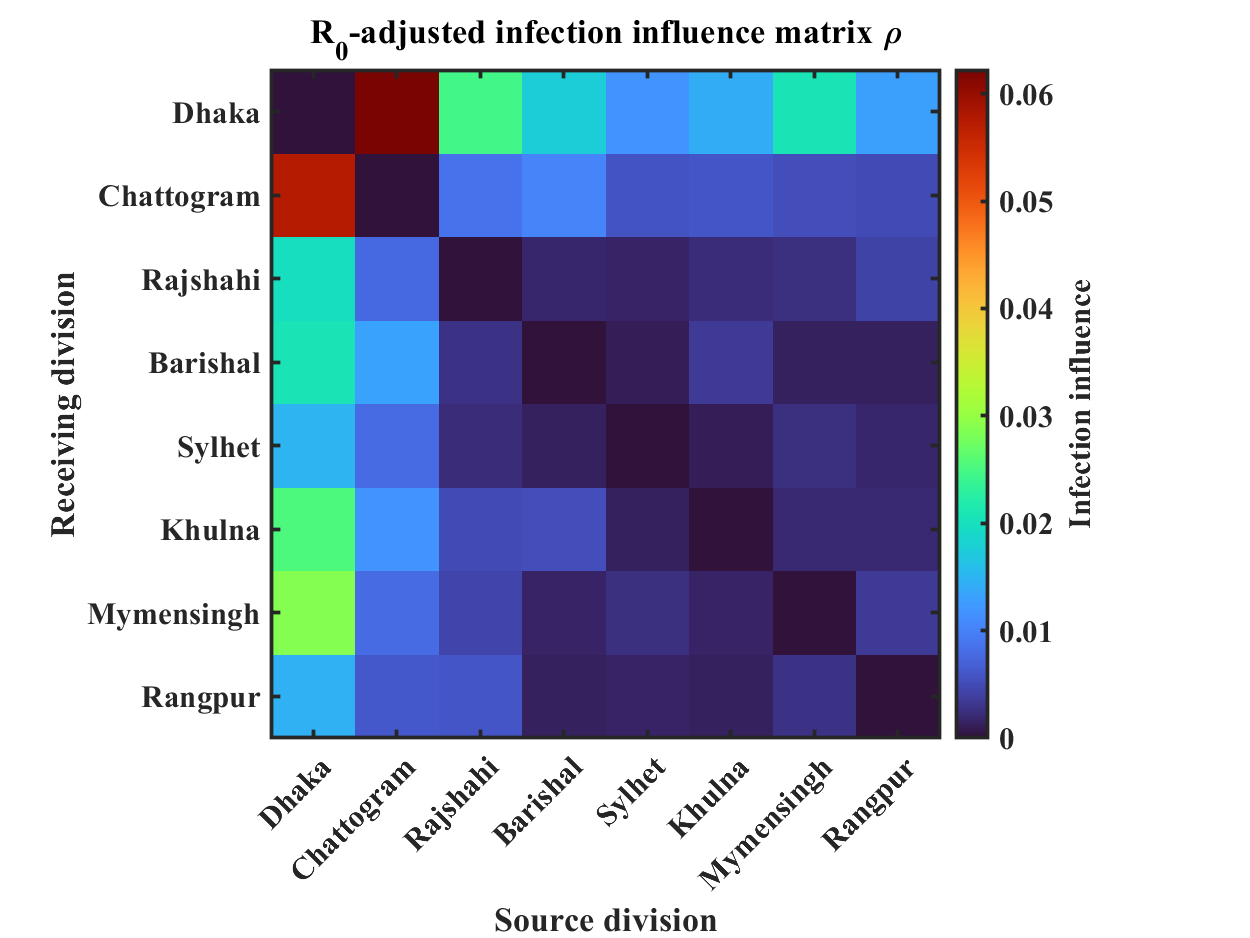}
	  		\caption{}
	  		\label{rho}
	  	\end{subfigure}
	  	
	  	\caption{Model-derived spatial coupling among the eight divisions. (a) Hub-adjusted relative connectivity matrix \(Q\) (table~\ref{tabq}); (b) reproduction number adjusted cross-patch infection influence matrix \(\rho\) (table~\ref{tabr}). Row represent receiving divisions and columns represent source divisions.}
	\label{Q_rho}
	\end{figure}
	Applying this parameterization produces the hub-oriented spatial structure shown in figure~\ref{Q_rho}. As expected from the higher hub weights assigned to Dhaka and Chattogram, the strongest directional connections occur from Chattogram to Dhaka ($Q=0.8$), Mymensingh to Dhaka ($Q=0.769$), Dhaka to Chattogram ($Q=0.718$), and Khulna to Dhaka ($Q=0.682$) (table~\ref{tabq}, figure~\ref{Q_rho}a). Because $Q$ is directional, connectivity between two divisions is not necessarily equal in both directions. After adjustment for the source-specific reproduction number, the greatest cross-patch infection influences occur from Chattogram to Dhaka ($\rho=0.0622$) and from Dhaka to Chattogram ($\rho=0.0574$) (table~\ref{tabr}, figure~\ref{Q_rho}b). The constructed network therefore identifies Dhaka and Chattogram as the principal modelled transmission corridor, while the other divisions maintain weaker but nonzero connections with these major hubs.
	\begin{figure}[H]
		\centering
		
		% First row
		\begin{subfigure}[t]{0.49\textwidth}
			\centering
			\includegraphics[width=\linewidth]{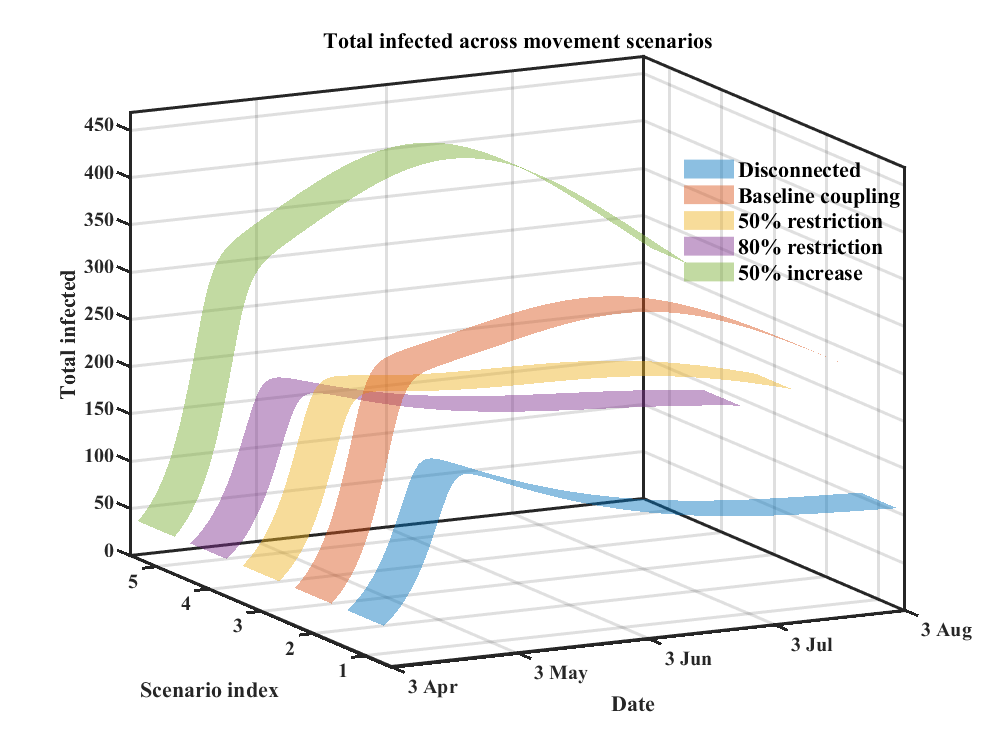}
			\caption{}
			\label{ribbon}
		\end{subfigure}
		%\hspace{0.008\textwidth}
		\begin{subfigure}[t]{0.5\textwidth}
			\centering
			\includegraphics[width=\linewidth]{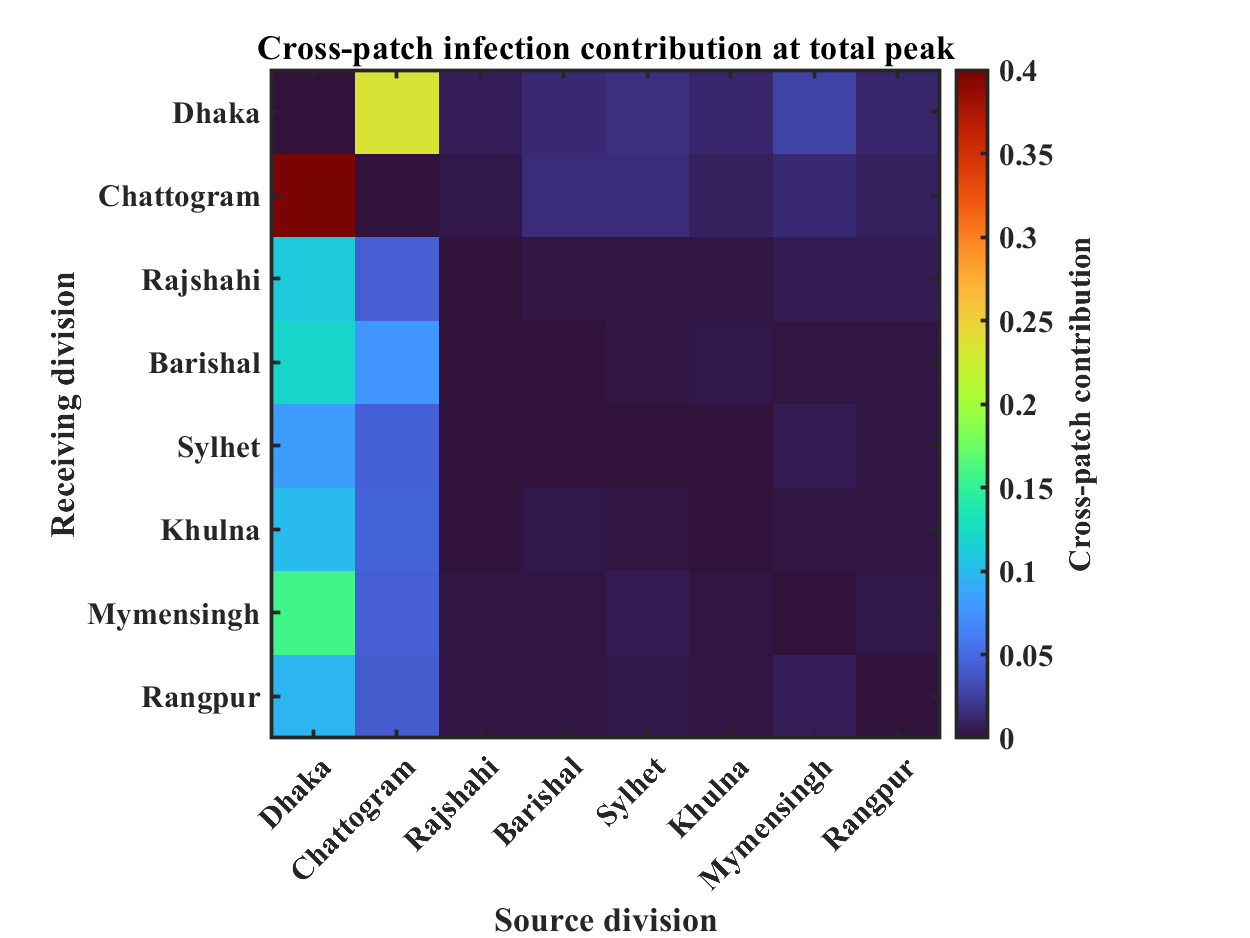}
			\caption{}
			\label{cross}
		\end{subfigure}
		
		\caption{Effects of spatial coupling and cross-patch infection across the eight divisions. (a) Three-dimensional ribbon representation of the total infectious population under no spatial coupling, baseline coupling, 50\% restriction, 80\% restriction, and 50\% increased-coupling scenarios. (b) Directional cross-patch infection contributions at the baseline network-wide epidemic peak; rows represent receiving divisions, columns represent source divisions, and colour intensity indicates the contribution from each source to infection in each receiving division.}
		\label{totalinf}
	\end{figure}
	The simulated epidemic trajectories differ markedly among the movement scenarios (figure~\ref{totalinf}a). Under baseline coupling, the total infectious population reaches a peak of 313.20 on simulation day 69, compared with 190.96 on day 19 in the disconnected system. Baseline coupling therefore increases the epidemic peak by approximately 64.02\% relative to the disconnected case and delays the peak by 50 days. It also increases the final infectious population from 104.49 to 232.86 and the cumulative infections from 1086.42 to 2345.35 (table~\ref{scenarios}). Movement restrictions substantially reduce the simulated burden. Relative to baseline coupling, a 50\% restriction decreases the epidemic peak by 26.69\%, the final infectious population by 21.37\%, and cumulative infections by 25.07\%. An 80\% restriction produces larger reductions of 34.37\%, 38.87\%, and 41.48\%, respectively (table~\ref{tab7}). Conversely, increasing movement and cross-patch infection influence by 50\% raises the peak to 407.59, representing a 30.14\% increase, and increases cumulative infections by 23.54\% to 2897.36. Although the peak occurs slightly earlier, shifting from day 69 to day 65, stronger connectivity produces a larger epidemic and a higher infection burden throughout the study period.
	\begin{table}[H]
		\centering
		\caption{Summary of epidemic outcomes under alternative movement and
			cross-patch infection scenarios.}
		\label{scenarios}
		
		\begin{threeparttable}
			\small
			\renewcommand{\arraystretch}{1}
			\setlength{\tabcolsep}{0.8pt}
			
			\resizebox{0.92\textwidth}{!}{%
			
			\begin{tabular}{
					>{\raggedright\arraybackslash\bfseries}p{2.5cm}
					>{\centering\arraybackslash}p{2.65cm}
					>{\centering\arraybackslash}p{2.5cm}
					>{\centering\arraybackslash}p{2.75cm}
					>{\centering\arraybackslash}p{2.20cm}
					>{\centering\arraybackslash}p{2.25cm}
				}
				
				\toprule
				
				\rowcolor{ScenarioBlue}
				\textcolor{white}{\textbf{Scenario}}
				&
				\multicolumn{2}{c}{
					\cellcolor{ScenarioBlue}
					\textcolor{white}{\textbf{Epidemic peak}}
				}
				&
				\multicolumn{2}{c}{
					\cellcolor{ScenarioBlue}
					\textcolor{white}{\textbf{Epidemic outcomes}}
				}
				&
				\cellcolor{ScenarioBlue}
				\textcolor{white}{\textbf{Relative effect}}
				\\
				
				\rowcolor{SubheaderBlue}
				{}
				&
				\raisebox{2ex}{%
					\makecell[c]{
						\textcolor{ScenarioBlue}{\textbf{Peak infectious}}\\
						\textcolor{ScenarioBlue}{\textbf{population}}\\
						\textcolor{ScenarioBlue}{\textbf{(individuals)}}
					}
				}
				&
				\raisebox{2ex}{%
					\makecell[c]{
						\textcolor{ScenarioBlue}{\textbf{Peak day}}\\
						\textcolor{ScenarioBlue}{\textbf{from 3 April}}\\
						\textcolor{ScenarioBlue}{\textbf{(days)}}
					}
				}
				&
				\raisebox{2ex}{%
					\makecell[c]{
						\textcolor{ScenarioBlue}{\textbf{Final infectious}}\\
						\textcolor{ScenarioBlue}{\textbf{population}}\\
						\textcolor{ScenarioBlue}{\textbf{(individuals)}}
					}
				}
				&
				\raisebox{2ex}{%
					\makecell[c]{
						\textcolor{ScenarioBlue}{\textbf{Cumulative}}\\
						\textcolor{ScenarioBlue}{\textbf{infections}}\\
						\textcolor{ScenarioBlue}{\textbf{(individuals)}}
					}
				}
				&
				\raisebox{2ex}{%    
				    \makecell[c]{
					    \textcolor{ScenarioBlue}{\textbf{Peak change}}\\
					    \textcolor{ScenarioBlue}{\textbf{vs.\ baseline}}\\
					    \textcolor{ScenarioBlue}{\textbf{(\%)}}
				    }
				}
				\\
				\midrule
				
				Disconnected
				& 190.96
				& 19
				& 104.49
				& 1086.42
				& \textcolor{ReductionGreen}{\(-39.03\)} \\
				
				\rowcolor{BaselineFill}
				Baseline coupling
				& \textbf{313.20}
				& \textbf{69}
				& \textbf{232.86}
				& \textbf{2345.35}
				& \textcolor{ScenarioBlue}{\textbf{Reference}} \\
				
				50\% restriction
				& 229.62
				& 22
				& 183.11
				& 1757.30
				& \textcolor{ReductionGreen}{\(-26.69\)} \\
				
				\rowcolor{AlternateRow}
				80\% restriction
				& 205.56
				& 20
				& 142.34
				& 1372.59
				& \textcolor{ReductionGreen}{\(-34.37\)} \\
				
				50\% increase
				& \textbf{407.59}
				& 65
				& \textbf{268.82}
				& \textbf{2897.36}
				& \textcolor{IncreaseRed}{\(\boldsymbol{+30.14}\)} \\
				
				\bottomrule
			\end{tabular}%
		}
			
			\begin{tablenotes}[flushleft]
				\footnotesize
				\item \textbf{\textit{Note~\ref{scenarios}:}} The no-spatial-coupling scenario excludes both
				interdivisional movement and cross-patch infection. Restrictions and increases are applied simultaneously to the movement and cross-patch infection coefficients relative to baseline coupling. Peak day is measured from 3 April, while final infectious population and cumulative
				infections are evaluated on 3 August. Values are continuous model estimates and may therefore be non-integer. Bold values in the final row identify the largest peak, final infectious population, and cumulative infection burden.
			\end{tablenotes}
			
		\end{threeparttable}
	\end{table}

The directional contribution pattern in figure~\ref{totalinf}b indicates that cross-patch infection is concentrated primarily along pathways originating from Dhaka and Chattogram. The pairwise estimates in table~\ref{tab5} identify the Dhaka-to-Chattogram pathway as the strongest contribution (0.4000 infections per day), followed by transmission from Chattogram to Dhaka (0.2342 infections per day); Dhaka also provides the largest external contribution to each of the other six divisions.
%%%%%%%%%%%%%%%%%%%%%%%%%%%%%%%%%%%%%%%%%%%%%%%%%%%%%%%%%%%%%%%%%%%%%%%%%%%
\begin{table}[H]
	\centering
	\caption{Division-specific local and cross-patch infection attribution at the baseline total epidemic peak.}
	\label{tab8}
	
	\renewcommand{\arraystretch}{1.12}
	\setlength{\tabcolsep}{3pt}
	
	\resizebox{\textwidth}{!}{%
		\begin{tabular}{lcccccc}
			\toprule
			
			\rowcolor{TableBlue}
			\textcolor{white}{\textbf{Division}}
			&
			\multicolumn{2}{c}{
				\cellcolor{TableBlue}
				\textcolor{white}{\textbf{Epidemic state}}
			}
			&
			\multicolumn{4}{c}{
				\cellcolor{TableBlue}
				\textcolor{white}{\textbf{Attribution of new infections}}
			}\\
			
			\rowcolor{HeaderFill}
			{}
			&
			\makecell[c]{
				\textcolor{TableBlue}{\textbf{Infectious}}\\
				\textcolor{TableBlue}{\textbf{population}}\\
				\textcolor{TableBlue}{\(\boldsymbol{I_{\mathrm{peak}}}\)}\\
				\textcolor{TableBlue}{\textbf{(individuals)}}
			}
			&
			\makecell[c]{
				\textcolor{TableBlue}{\textbf{Total force}}\\
				\textcolor{TableBlue}{\textbf{of infection}}\\
				\textcolor{TableBlue}{\(\boldsymbol{\lambda_i}\)}\\
				\textcolor{TableBlue}{\textbf{(day\(^{-1}\))}}
			}
			&
			\makecell[c]{
				\textcolor{TableBlue}{\textbf{Local new}}\\
				\textcolor{TableBlue}{\textbf{infections}}\\
				\textcolor{TableBlue}{\textbf{(individuals}}\\
				\textcolor{TableBlue}{\textbf{day\(^{-1}\))}}
			}
			&
			\makecell[c]{
				\textcolor{TableBlue}{\textbf{Cross-patch}}\\
				\textcolor{TableBlue}{\textbf{new infections}}\\
				\textcolor{TableBlue}{\textbf{(individuals}}\\
				\textcolor{TableBlue}{\textbf{day\(^{-1}\))}}
			}
			&
			\makecell[c]{
				\textcolor{TableBlue}{\textbf{Local}}\\
				\textcolor{TableBlue}{\textbf{share}}\\
				\textcolor{TableBlue}{\textbf{(\%)}}
			}
			&
			\makecell[c]{
				\textcolor{TableBlue}{\textbf{Cross-patch}}\\
				\textcolor{TableBlue}{\textbf{share}}\\
				\textcolor{TableBlue}{\textbf{(\%)}}
			}
			\\
			\midrule
			
			\textbf{Dhaka}
			& 90.3831 & \textbf{0.01626026} & 3.726334 & 0.320884
			& 92.0715 & 7.9285\\
			
			\textbf{Chattogram}
			& \textbf{91.3132} & 0.00973413
			& \textbf{7.036637} & \textbf{0.464533}
			& \textbf{93.8072} & 6.1928\\
			
			\textbf{Rajshahi}
			& 6.8711 & 0.00172993 & 0.426159 & 0.172610
			& 71.1725 & \textbf{28.8275}\\
			
			\textbf{Barishal}
			& 19.2807 & 0.00214293 & 1.232822 & 0.210105
			& 85.4389 & 14.5611\\
			
			\textbf{Sylhet}
			& 33.1514 & 0.00413826 & 2.033620 & 0.137566
			& 93.6640 & 6.3360\\
			
			\textbf{Khulna}
			& 19.4340 & 0.00364612 & 0.847878 & 0.159691
			& 84.1508 & 15.8492\\
			
			\textbf{Mymensingh}
			& 31.1460 & 0.00526103 & 1.888468 & 0.215828
			& 89.7435 & 10.2565\\
			
			\textbf{Rangpur}
			& 21.6187 & 0.00260735 & 1.591957 & 0.156038
			& 91.0733 & 8.9267\\
			
			\bottomrule
		\end{tabular}%
	}
	
	\vspace{2mm}
	
	\parbox[t]{\textwidth}{%
		\footnotesize
		
		\textbf{\textit{Note~\ref{tab8}:}} \(I_{\mathrm{peak}}\) denotes the continuous
		model-estimated infectious population in each division when the
		network-wide infectious population reaches its baseline peak.
		The total force of infection, \(\lambda_i\), is expressed per day.
		Local and cross-patch new infections are instantaneous
		model-estimated rates evaluated at the same peak and are expressed
		in individuals per day. Their percentages give the corresponding
		shares of total new infections within each division. Because the
		model is continuous, population and infection-rate estimates may
		be non-integer. Bold type identifies the largest value in each
		column.
	}
\end{table}

%%%%%%%%%%%%%%%%%%%%%%%%%%%%%%%%%%%%%%%%%%%%%%%%%%%%%%%%%%%%%%%%%%%%%%%%%%
 At the baseline network-wide epidemic peak, table~\ref{tab8} shows that Chattogram has the largest infectious population (\(91.31\) individuals), closely followed by Dhaka (\(90.38\)). Chattogram also records the highest local and cross-patch new-infection rates (\(7.0366\) and \(0.4645\) infections per day, respectively), whereas Dhaka exhibits the greatest total force of infection (\(0.0163~\mathrm{day}^{-1}\)). Local transmission remains dominant in every division, contributing \(71.17\%\)--\(93.81\%\) of new infections. Chattogram has the largest local share (\(93.81\%\)) and the smallest cross-patch share (\(6.19\%\)), whereas Rajshahi has the smallest local share (\(71.17\%\)) and the greatest relative dependence on cross-patch infection (\(28.83\%\)). The compartmental movement balances in table~\ref{tab9} further indicate that Dhaka receives the largest net susceptible and vaccinated inflows and has the highest movement-related exposure tendency. In contrast, the infectious-movement balance classifies Dhaka and Chattogram as net exporters, while the remaining divisions are net importers, with Rajshahi receiving the largest net infectious inflow. Consistent with this pattern, table~\ref{tab10} identifies Dhaka and Chattogram as cross-patch sources, with net source rates of \(0.7478\) and \(0.0648\) infections per day, respectively. The other divisions act as cross-patch sinks, with Barishal exhibiting the largest net sink magnitude (\(0.1687\) infections per day).
	
	\begin{figure}[H]
		\centering
		\includegraphics[
		width=\textwidth,
		trim={0 1.5cm 0 0},
		clip
		]{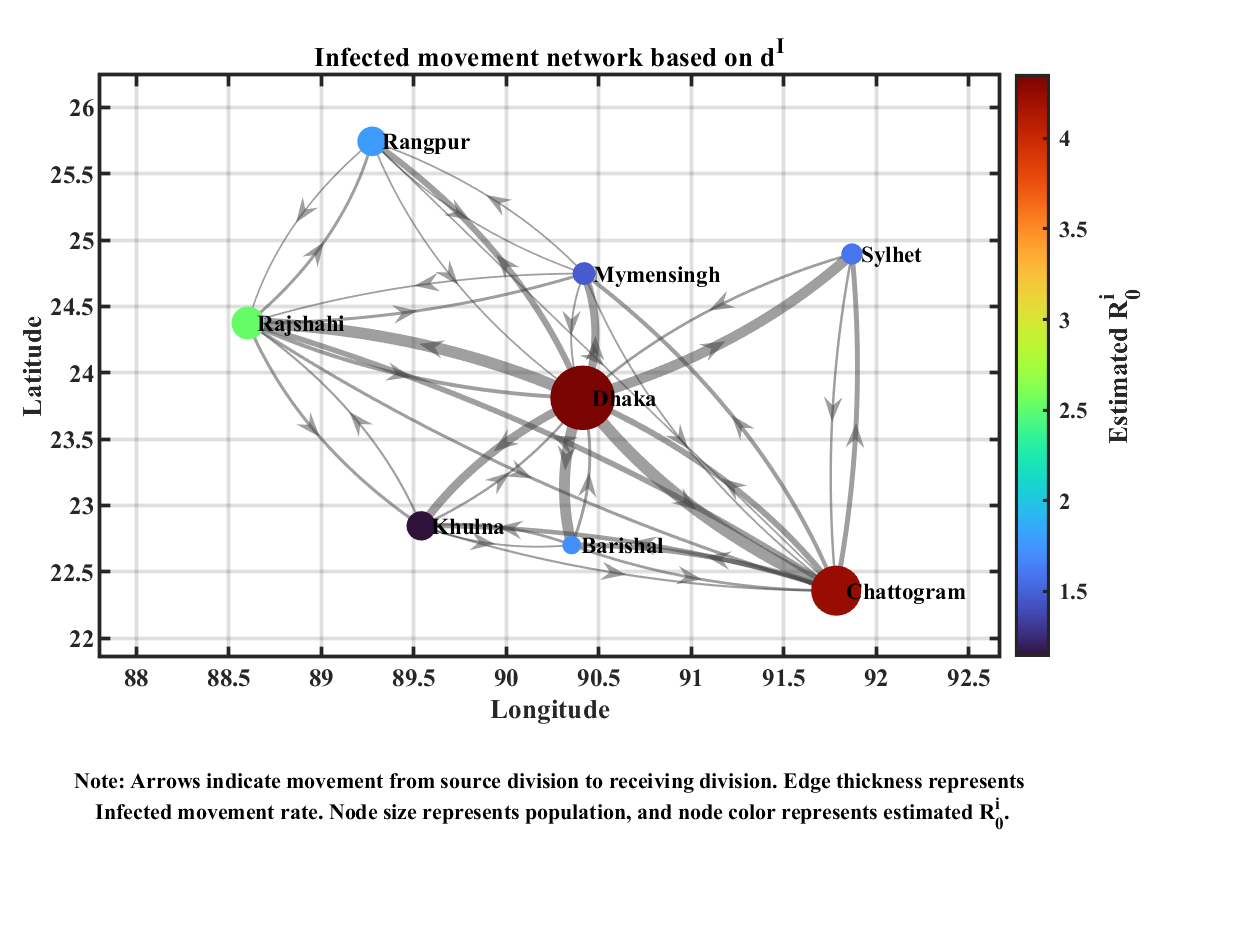}
		\caption{Baseline infectious-movement network among the eight divisions, showing the 35 strongest directed links and their relative movement intensities.}
		\label{directional}
	\end{figure}
	To improve visual clarity, figure~\ref{directional} displays the 35 strongest directed infectious-movement links, ranked according to \(d_{ij}^{I}\); all movement links remain included in the model simulations, while weaker links are omitted only from the network visualization to reduce overcrowding and overlapping arrows. The corresponding source--sink pattern identifies Dhaka and Chattogram as net exporters and the other six divisions as net importers. Together with the infection-attribution results, this indicates that local transmission remains dominant in every division, although cross-patch infection constitutes a larger proportion of new infections in several lower-burden divisions.

The contour analyses compare the effects of network-wide scaling with those of scaling only the links connected to each focal division (\cref{contour_dc,contour_rjb,contour_sk,contour_mr}). As expected, both analyses reproduce the same cumulative burden at the baseline point (1,1), with the largest values occurring in Dhaka 680.72 (figure~\ref{contour_dc}a,b) and Chattogram 647.72 (figure~\ref{contour_dc}c,d). Removing movement and cross-patch infection does not eliminate transmission: cumulative burden remains positive in every division, ranging from 38.86 in Khulna to 458.70 in Dhaka, confirming that local transmission alone can sustain infection. However, the proportional amplification associated with baseline spatial coupling varies considerably. 
	\begin{figure}[H]
		\centering
		
		\begin{subfigure}[t]{0.49\textwidth}
			\centering
			\includegraphics[width=\linewidth]{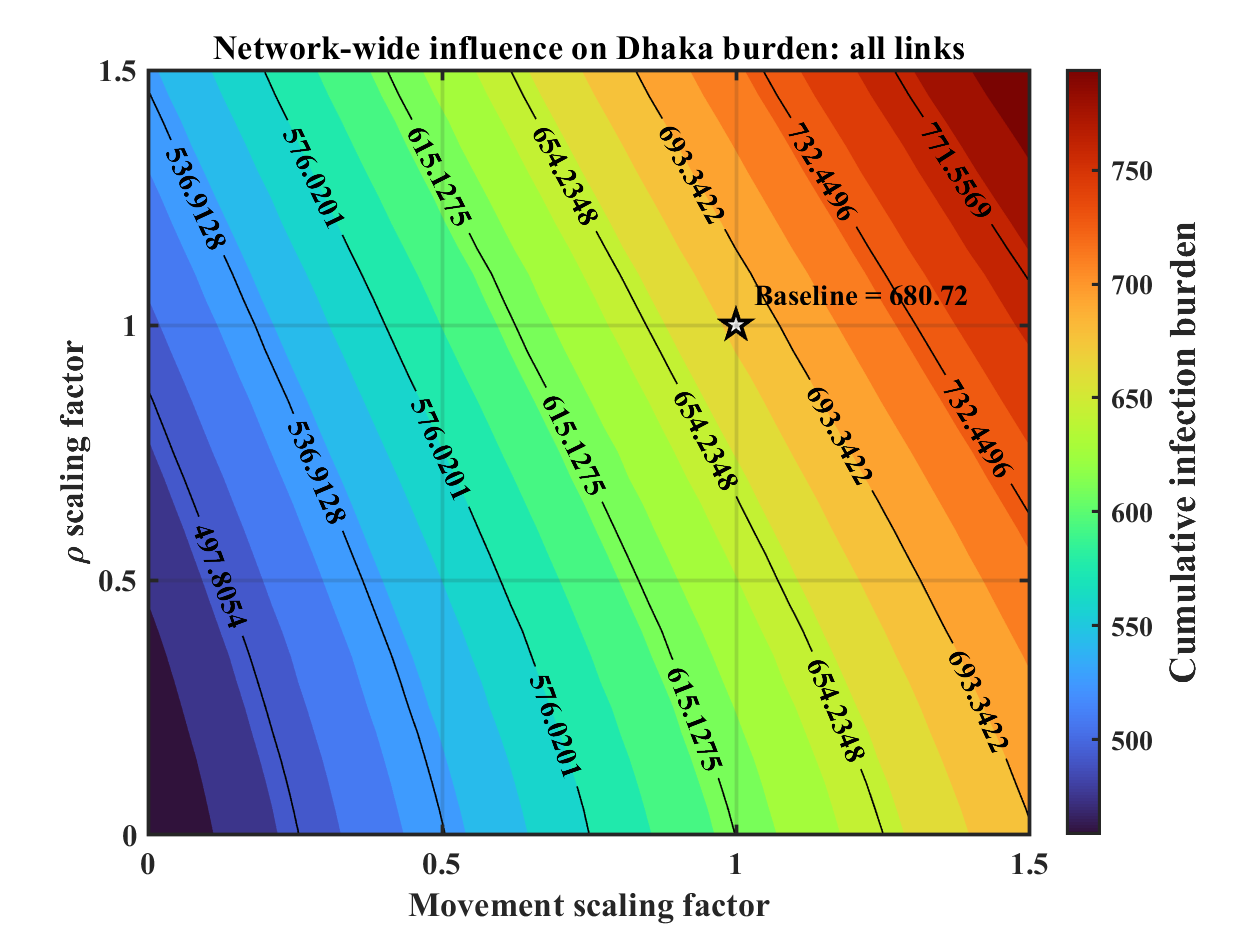}
			\caption{Dhaka: network-wide scaling}
			\label{global_dhaka}
		\end{subfigure}
		\hfill
		\begin{subfigure}[t]{0.49\textwidth}
			\centering
			\includegraphics[width=\linewidth]{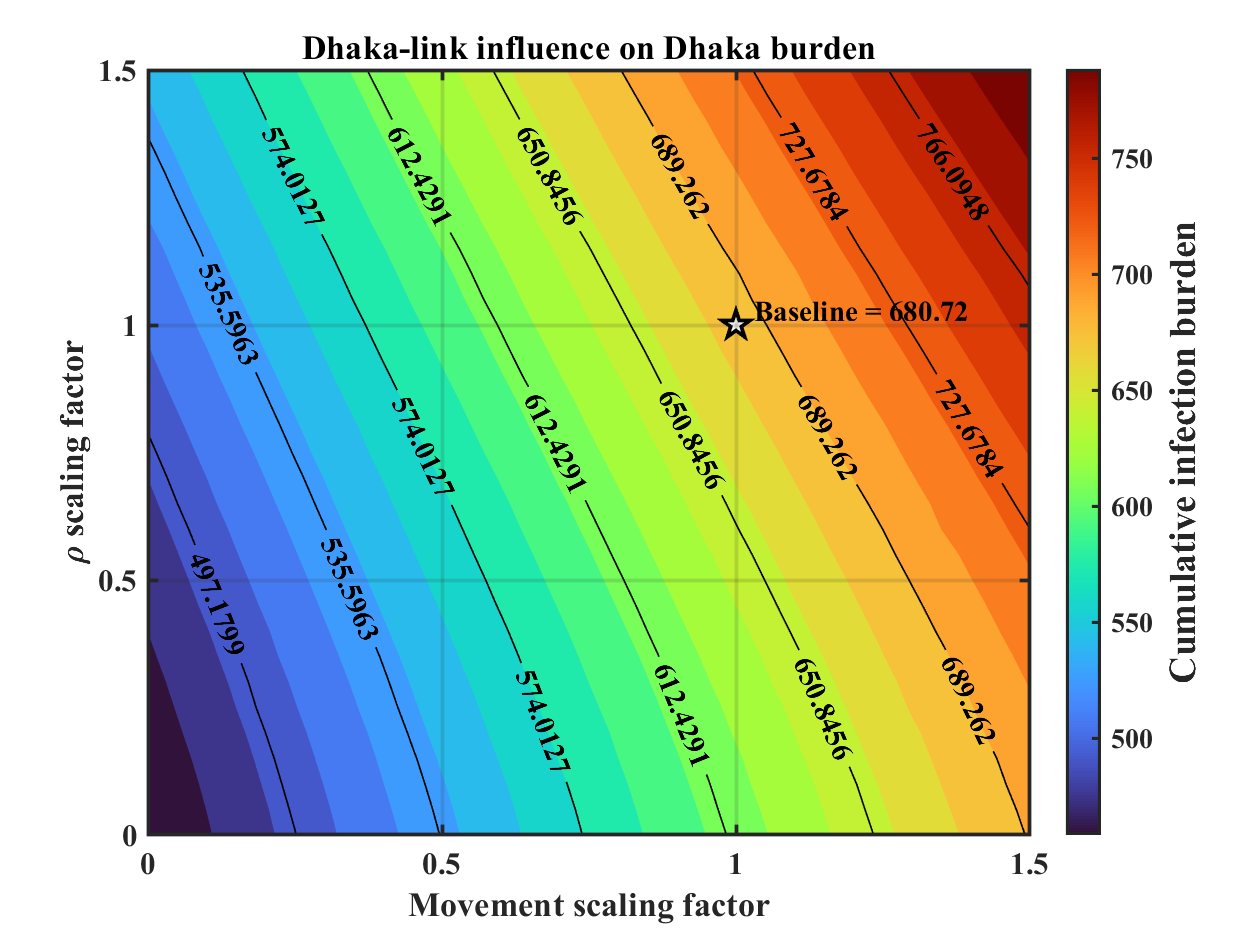}
			\caption{Dhaka-link-specific scaling}
			\label{link_dhaka}
		\end{subfigure}
		
		\vspace{0.3cm}
		
		\begin{subfigure}[t]{0.49\textwidth}
			\centering
			\includegraphics[width=\linewidth]{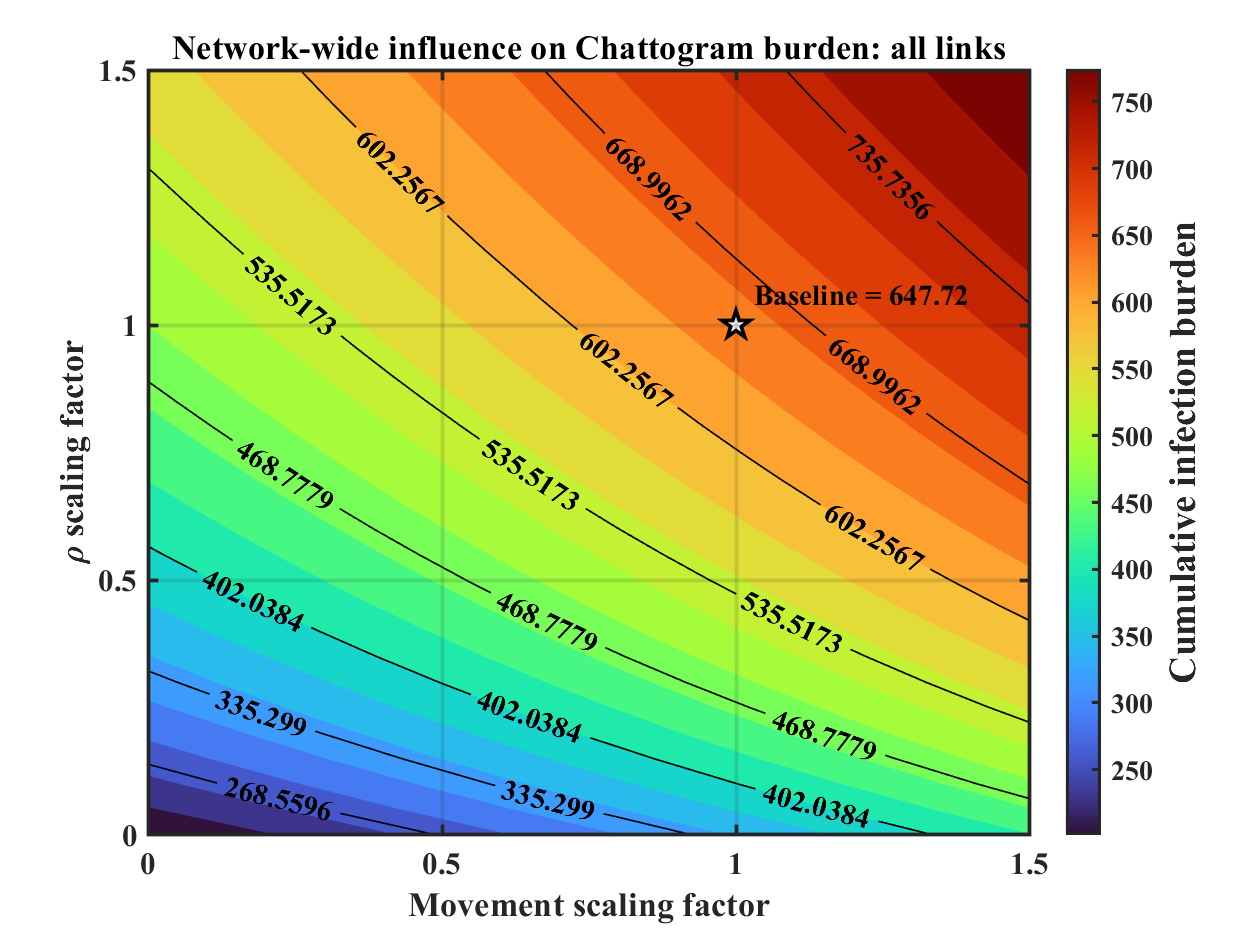}
			\caption{Chattogram: network-wide scaling}
			\label{global_chattogram}
		\end{subfigure}
		\hfill
		\begin{subfigure}[t]{0.49\textwidth}
			\centering
			\includegraphics[width=\linewidth]{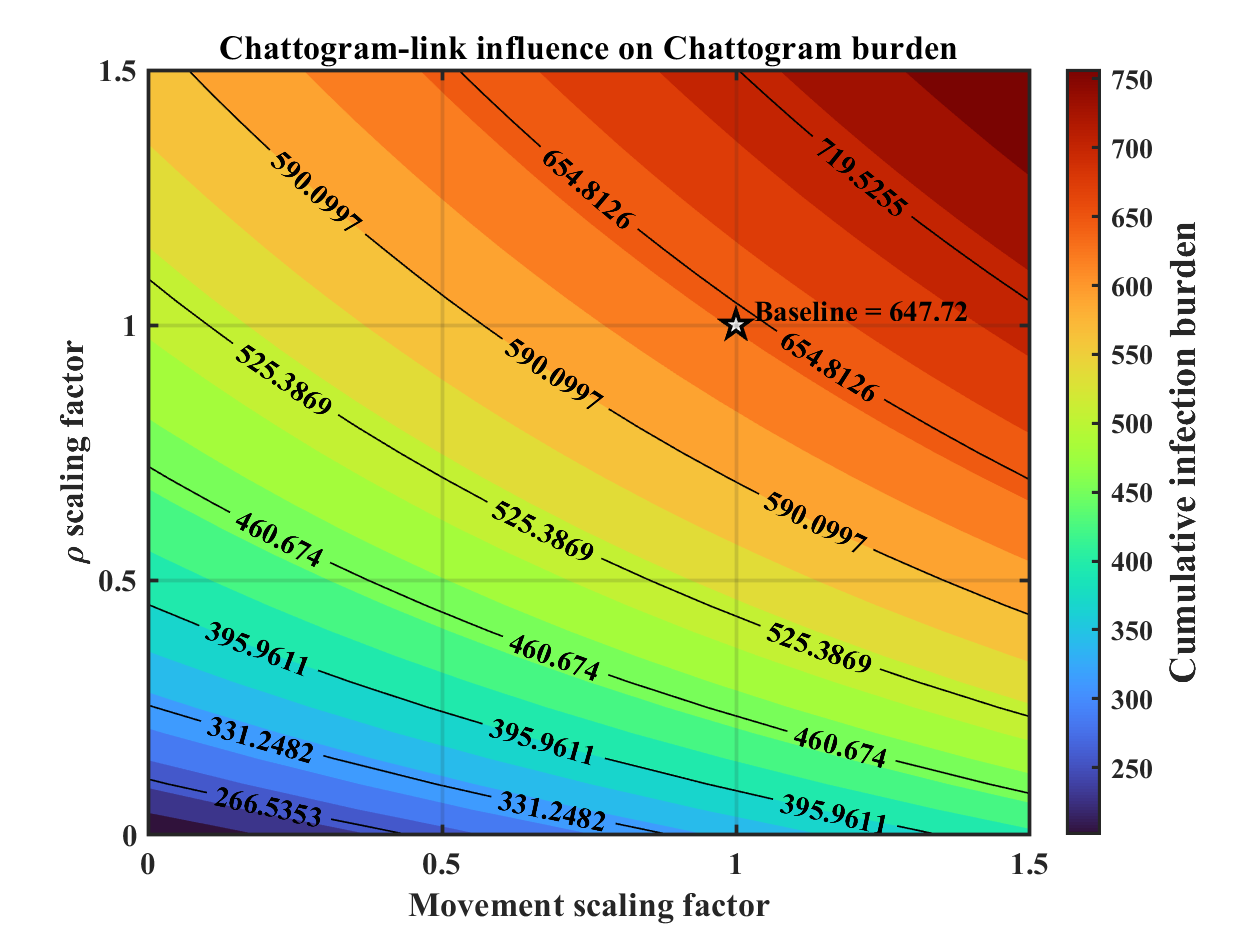}
			\caption{Chattogram-link-specific scaling}
			\label{link_chattogram}
		\end{subfigure}
		
		\caption{Sensitivity of the cumulative infection burden in Dhaka and
			Chattogram to movement and cross-patch infection. The left panels vary
			all interdivisional links simultaneously, whereas the right panels vary
			only links connected to the indicated division while retaining all other
			links at their baseline values. Stars indicate the baseline point
			$(1,1)$.}
		\label{contour_dc}
	\end{figure}
	 Relative to zero coupling, cumulative burden increases most strongly in Mymensingh 319.90\%, Rangpur 269.84\%, Khulna 242.25\%, Barishal 237.26\%, and Chattogram 220.94\%, indicating greater dependence on spatial transmission despite their different absolute burdens. When both scaling factors increase to 1.5, network-wide burden rises by 12.98\% in Sylhet (figure~\ref{contour_sk}a) to 32.95\% in Khulna (figure~\ref{contour_sk}c) relative to baseline. These increases consistently exceed those obtained by scaling only the focal division's links, with differences ranging from 6.15 infections in Dhaka to 18.24 in Chattogram. This indicates that indirect pathways through the wider network contribute appreciably beyond direct connections to the focal division. For Dhaka, Chattogram, and Rajshahi (figure~\ref{contour_rjb}a,b), the minimum and maximum burdens occur at the lowest and highest joint-scaling points, respectively, suggesting a largely monotonic spatial effect. In contrast, the off-corner extrema observed in Barishal (figure~\ref{contour_rjb}c,d), Khulna (figure~\ref{contour_sk}c,d), Mymensingh (figure~\ref{contour_mr}a,b), Rangpur (figure~\ref{contour_mr}c,d), and Sylhet (figure~\ref{contour_sk}a,b) indicate more complex interactions in which movement may redistribute infections among divisions rather than uniformly increasing the burden within each focal division.
	\begin{figure}[H]
		\centering
		
		\begin{subfigure}[t]{0.45\textwidth}
			\centering
			\includegraphics[width=\linewidth]{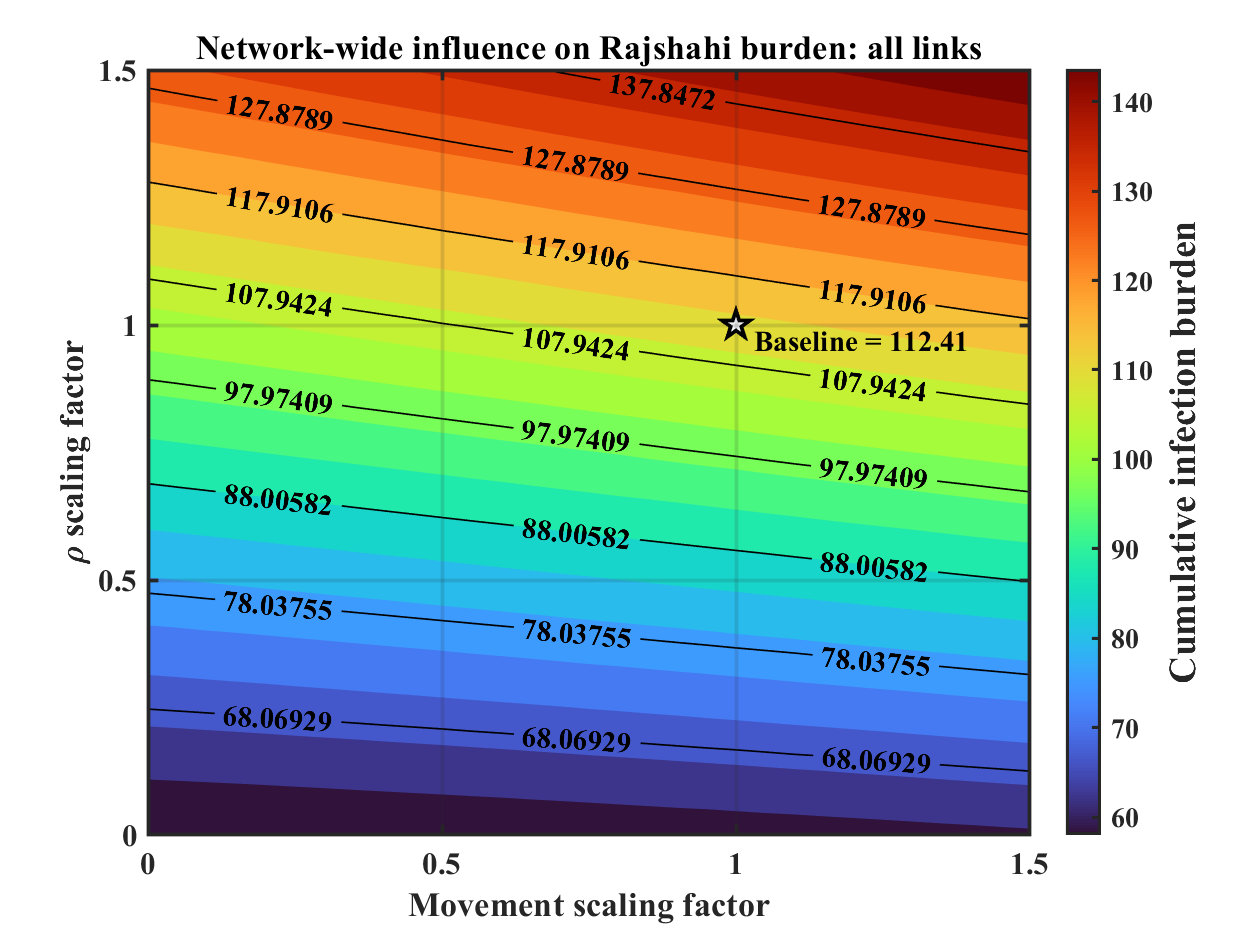}
			\caption{Rajshahi: network-wide scaling}
			\label{global_Rajshahi}
		\end{subfigure}
		\hfill
		\begin{subfigure}[t]{0.45\textwidth}
			\centering
			\includegraphics[width=\linewidth]{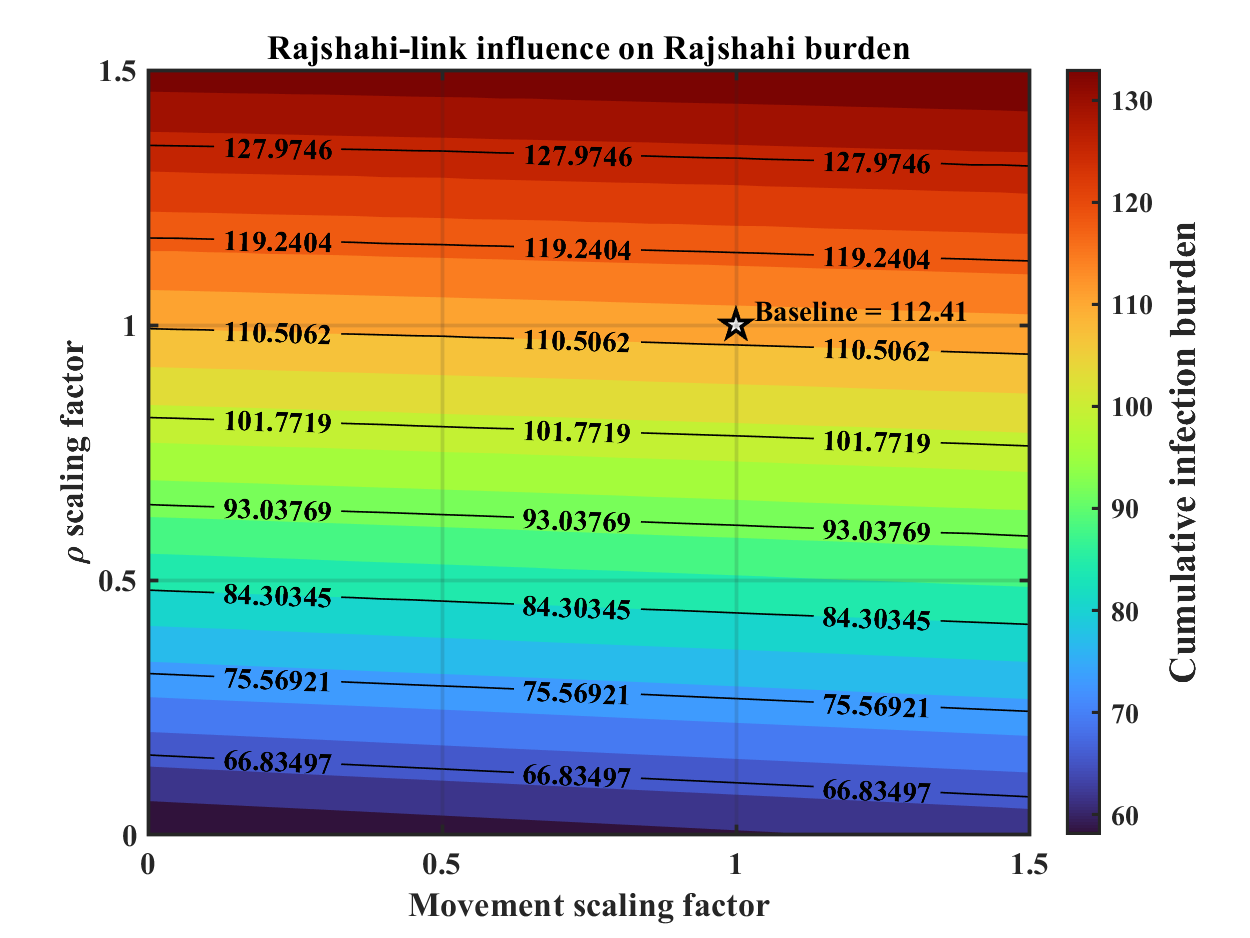}
			\caption{Rajshahi-link-specific scaling}
			\label{link_Rajshahi}
		\end{subfigure}
		
		\vspace{0.3cm}
		
		\begin{subfigure}[t]{0.49\textwidth}
			\centering
			\includegraphics[width=\linewidth]{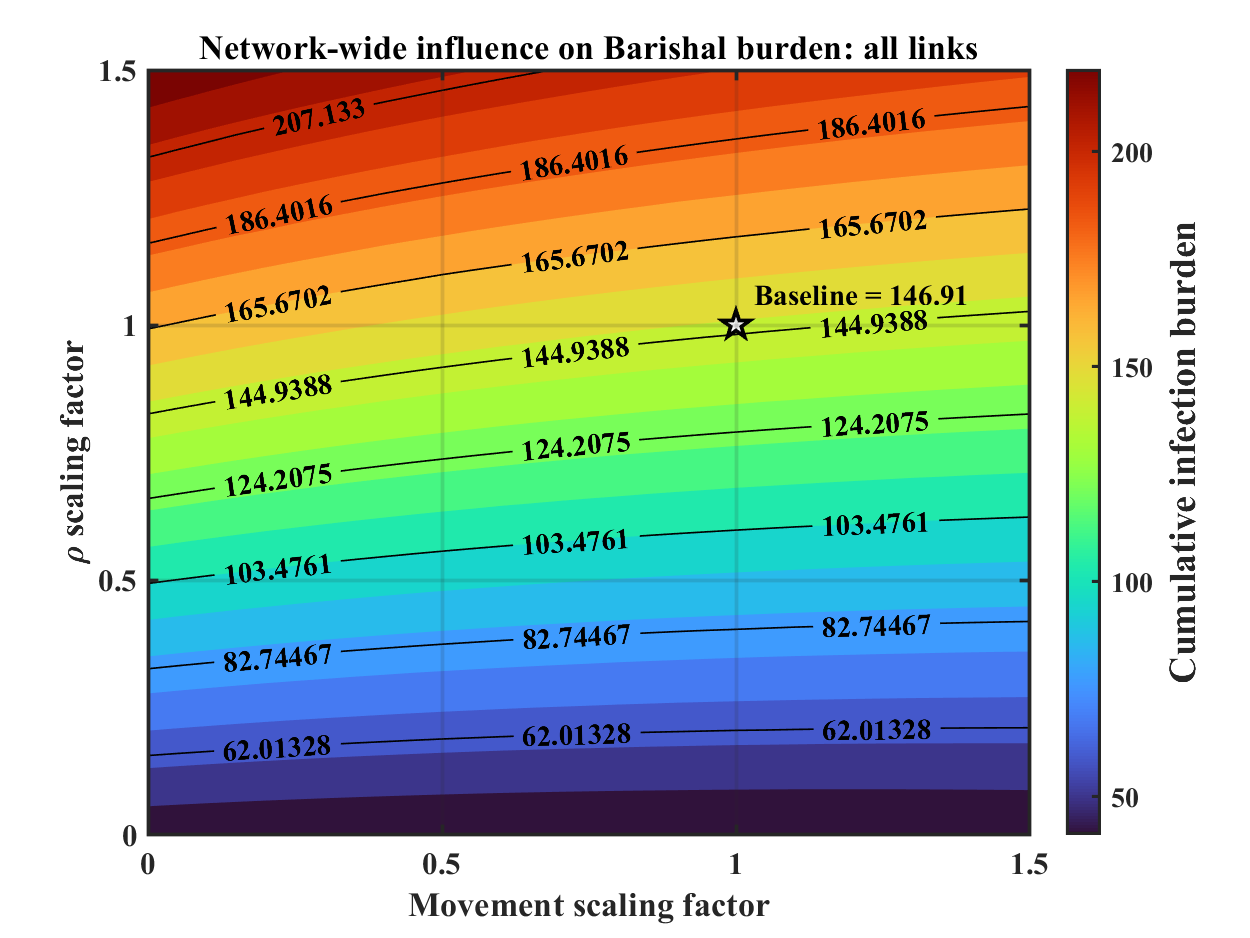}
			\caption{Barishal: network-wide scaling}
			\label{global_Barishal}
		\end{subfigure}
		\hfill
		\begin{subfigure}[t]{0.49\textwidth}
			\centering
			\includegraphics[width=\linewidth]{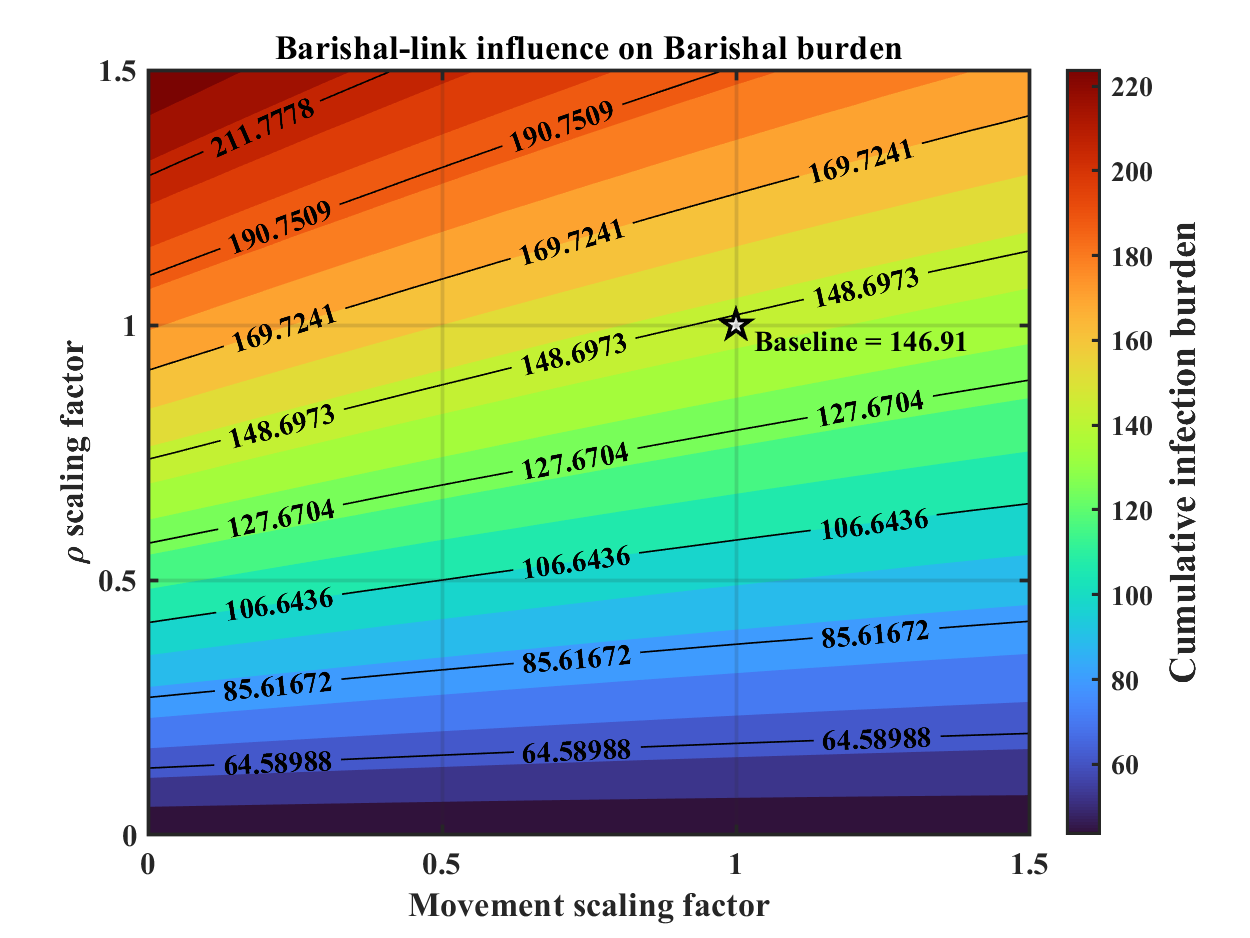}
			\caption{Barishal-link-specific scaling}
			\label{link_Barishal}
		\end{subfigure}
		
		\caption{Sensitivity of the cumulative infection burden in Rajshahi and
			Barishal to movement and cross-patch infection. The left panels vary
			all interdivisional links simultaneously, whereas the right panels vary
			only links connected to the indicated division while retaining all other
			links at their baseline values. Stars indicate the baseline point
			$(1,1)$.}
		\label{contour_rjb}
	\end{figure}

	From a public-health perspective, these patterns do not support movement control as a uniform stand-alone intervention. Dhaka and Chattogram require strong mobility-linked surveillance and rapid information sharing, whereas the receiving divisions would benefit more from early detection of imported infections, coordinated contact tracing, and synchronized vaccination across connected areas. Because a substantial disease burden remains even without spatial coupling, network-based interventions should be implemented alongside local vaccination, case isolation, and outbreak-response measures.

	\section{Conclusion}\label{conclusion}
	The current measles outbreak in Bangladesh is likely driven by several interconnected factors, including immunity gaps, missed routine vaccinations, vaccine-supply disruptions, high population density, and frequent population movement. This study focuses on how differences in local transmission and vaccination activity, together with interdivisional connectivity, shape the outbreak. Dhaka and Chattogram show the greatest local transmission potential, while spatial coupling increases the network-wide epidemic peak from 190.96 to 313.20 infectious individuals and more than doubles cumulative infections. Although most infections arise through local transmission, cross-patch infection contributes up to 28.83\% of new infections in Rajshahi, with Dhaka and Chattogram acting as the main spatial sources. These findings highlight the need to maintain two-dose vaccination coverage, prevent vaccine shortages, and provide catch-up vaccination to children who have missed one or both routine doses. Stronger surveillance, early case detection, and coordinated responses across connected divisions and major travel routes are also important. Focusing control efforts on major source divisions while closing immunity gaps in vulnerable receiving divisions may offer a more effective and sustainable approach than applying the same interventions uniformly across the country.
	
%	\section*{Acknowledgments} 
%	The research was partially supported by the University Grants Commission (UGC), Bangladesh.
	
	%=========================Conflict of interest======================================
	\section*{Conflict of interest}
	The authors declare no conflict of interest. 
	
	\section*{Data Statement}
	{The datasets generated and/or analysed during the current study are available in the Faizunnesa Khondaker repository, \url{https://github.com/Faizunnesa-Khondaker/Measles.git}.}

	\section*{Ethical approval}
	No consent is required to publish this manuscript.
	
	\section*{Author Contributions (CRediT)}
	Faizunnesa Khondaker: Conceptualization, Data curation, Formal analysis, Methodology, Software, Validation, Visualization, Writing–original draft.\\
	Md. Kamrujjaman: Conceptualization, Methodology, Investigation, Validation, Software, Supervision, Writing–review and editing.\\
	All authors have read and agreed to the published version of the manuscript.

%	Appendix~\ref{tables} (Tables~\ref{tab})

	\clearpage
\appendix
\section{Supplementary Figures}\label{figs}

\begin{figure}[H]
	\centering
	
	\begin{subfigure}[t]{0.48\textwidth}
		\centering
		\includegraphics[width=\linewidth]{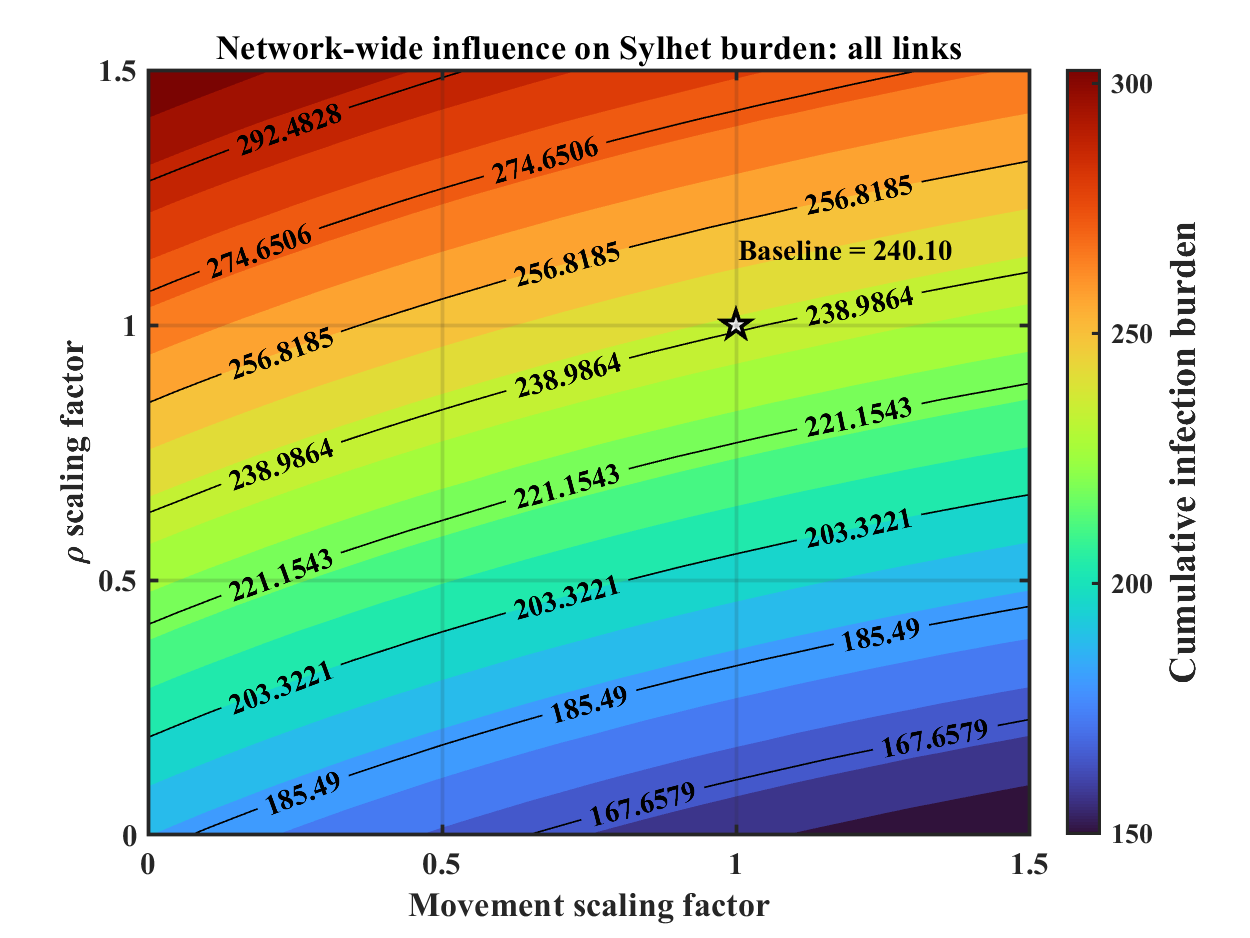}
		\caption{Sylhet: network-wide scaling}
		\label{global_Sylhet}
	\end{subfigure}
	\hfill
	\begin{subfigure}[t]{0.48\textwidth}
		\centering
		\includegraphics[width=\linewidth]{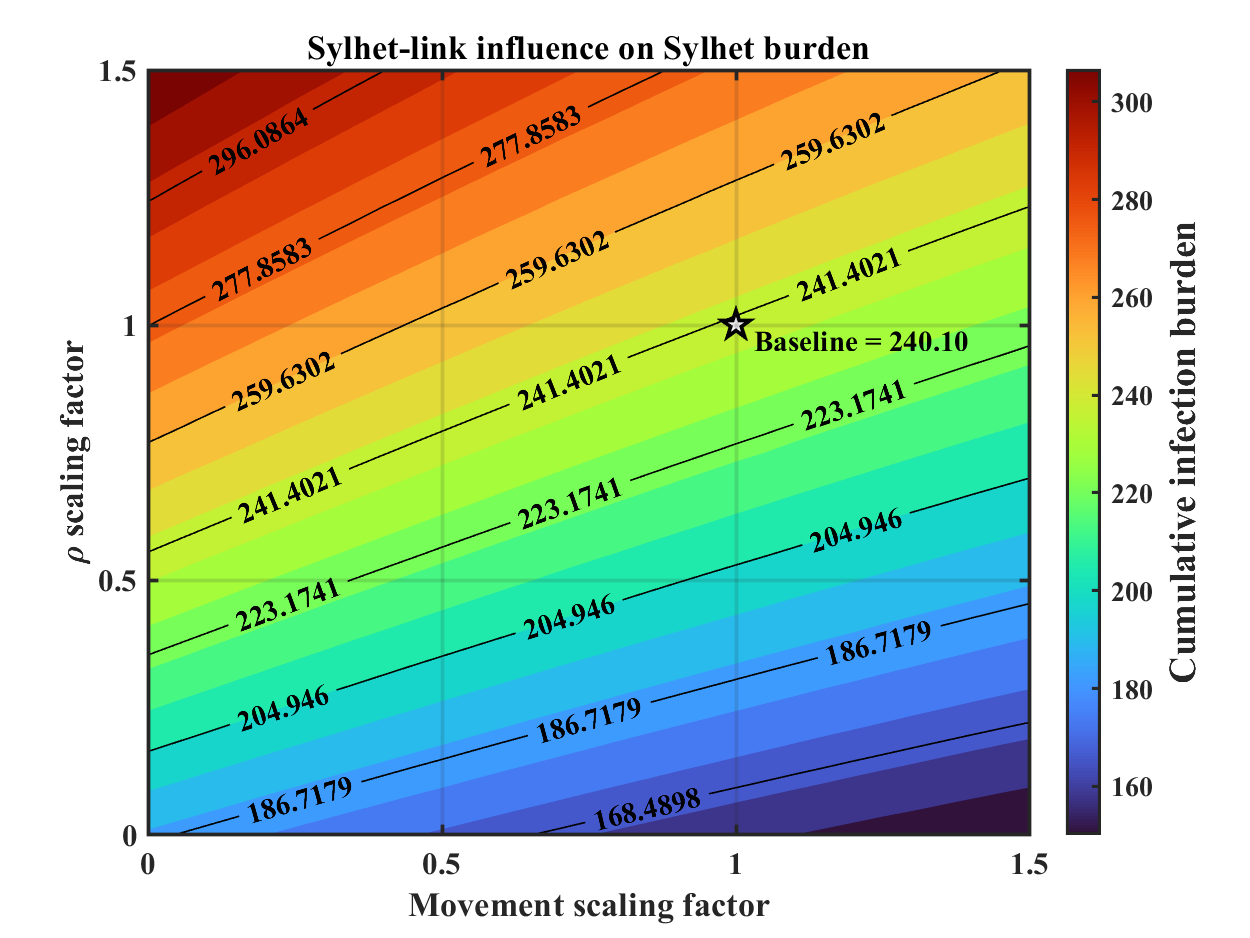}
		\caption{Sylhet-link-specific scaling}
		\label{link_Sylhet}
	\end{subfigure}
	
	\vspace{0.3cm}
	
	\begin{subfigure}[t]{0.48\textwidth}
		\centering
		\includegraphics[width=\linewidth]{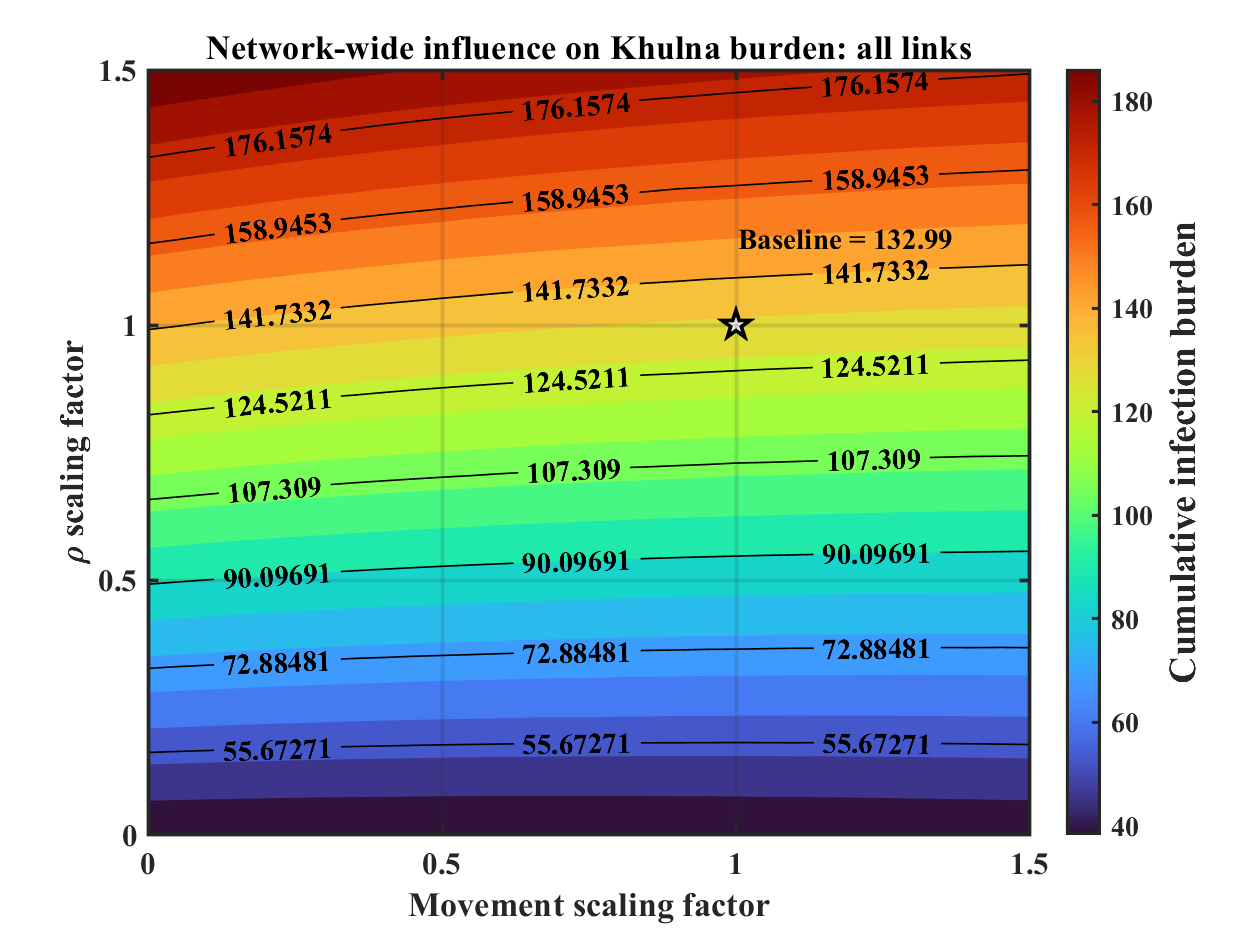}
		\caption{Khulna: network-wide scaling}
		\label{global_Khulna}
	\end{subfigure}
	\hfill
	\begin{subfigure}[t]{0.48\textwidth}
		\centering
		\includegraphics[width=\linewidth]{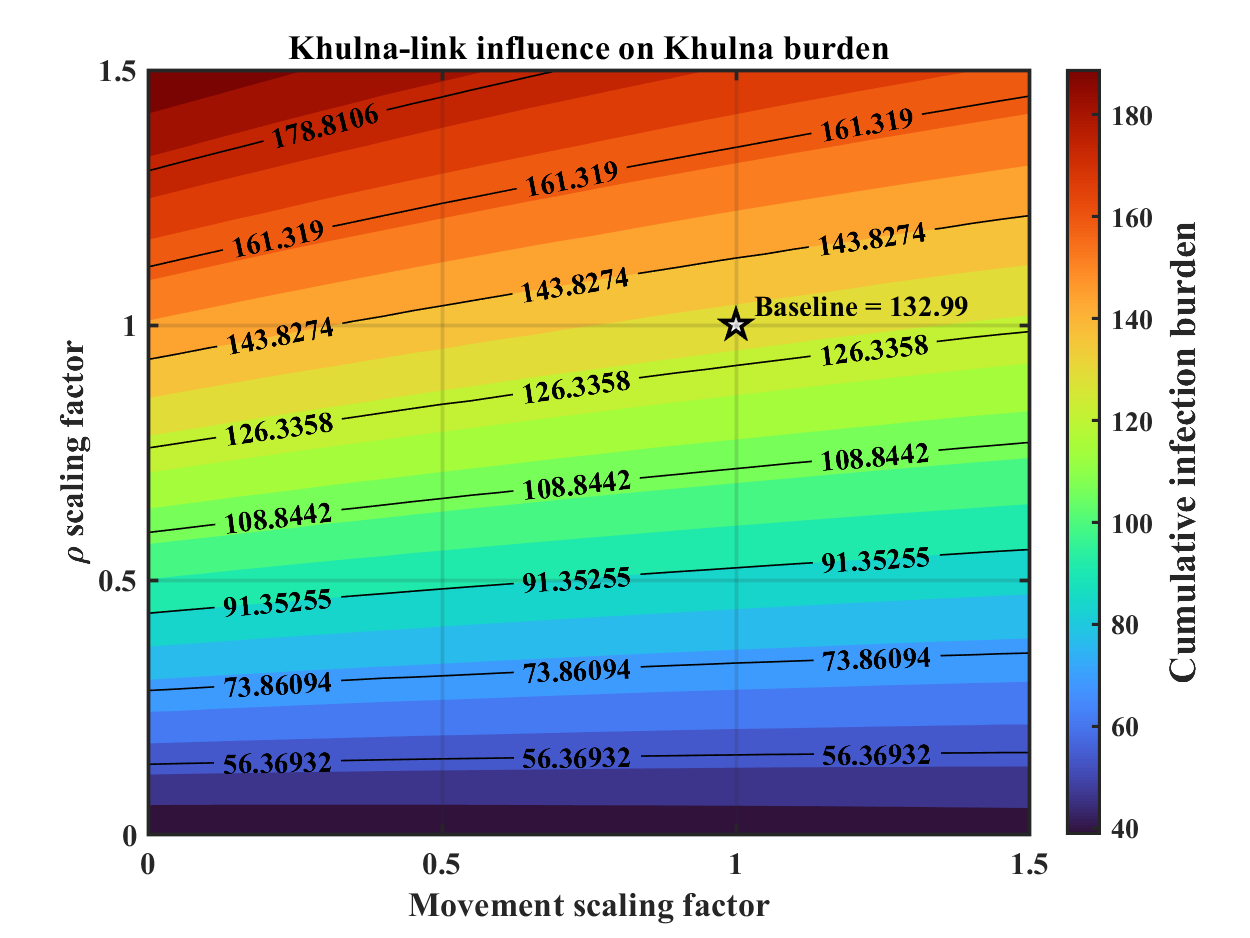}
		\caption{Khulna-link-specific scaling}
		\label{link_Khulna}
	\end{subfigure}
	
	\caption{Sensitivity of the cumulative infection burden in Sylhet and
		Khulna to movement and cross-patch infection. The left panels vary
		all interdivisional links simultaneously, whereas the right panels vary
		only links connected to the indicated division while retaining all other
		links at their baseline values. Stars indicate the baseline point
		$(1,1)$.}
	\label{contour_sk}
\end{figure}
%%%%%%%%%%%%%%%%%%%%%%%%%%%%%%%%%%%%%%%%%%
	\begin{figure}[H]
	\centering
	
	\begin{subfigure}[t]{0.48\textwidth}
		\centering
		\includegraphics[width=\linewidth]{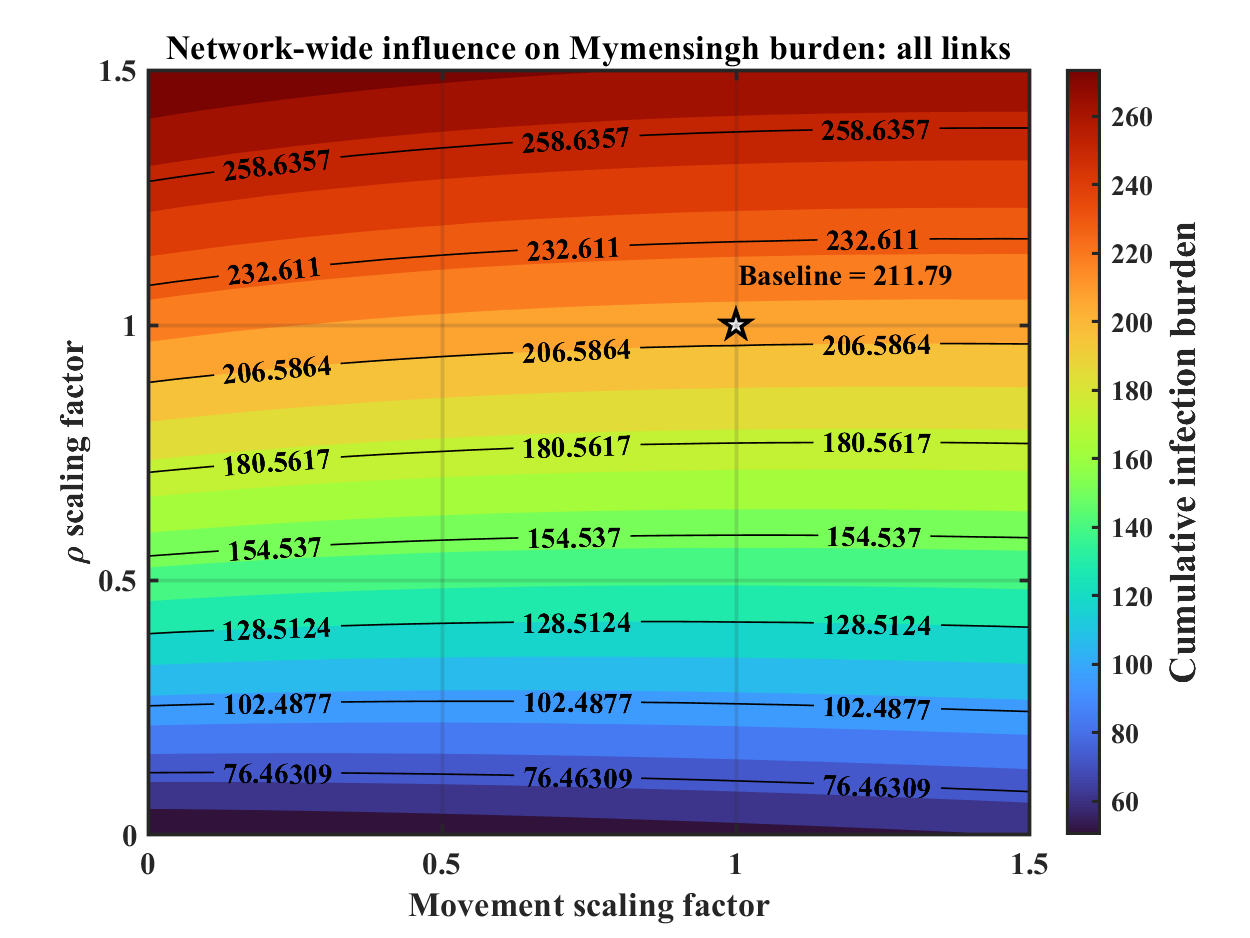}
		\caption{Mymensingh: network-wide scaling}
		\label{global_Mymensingh}
	\end{subfigure}
	\hfill
	\begin{subfigure}[t]{0.48\textwidth}
		\centering
		\includegraphics[width=\linewidth]{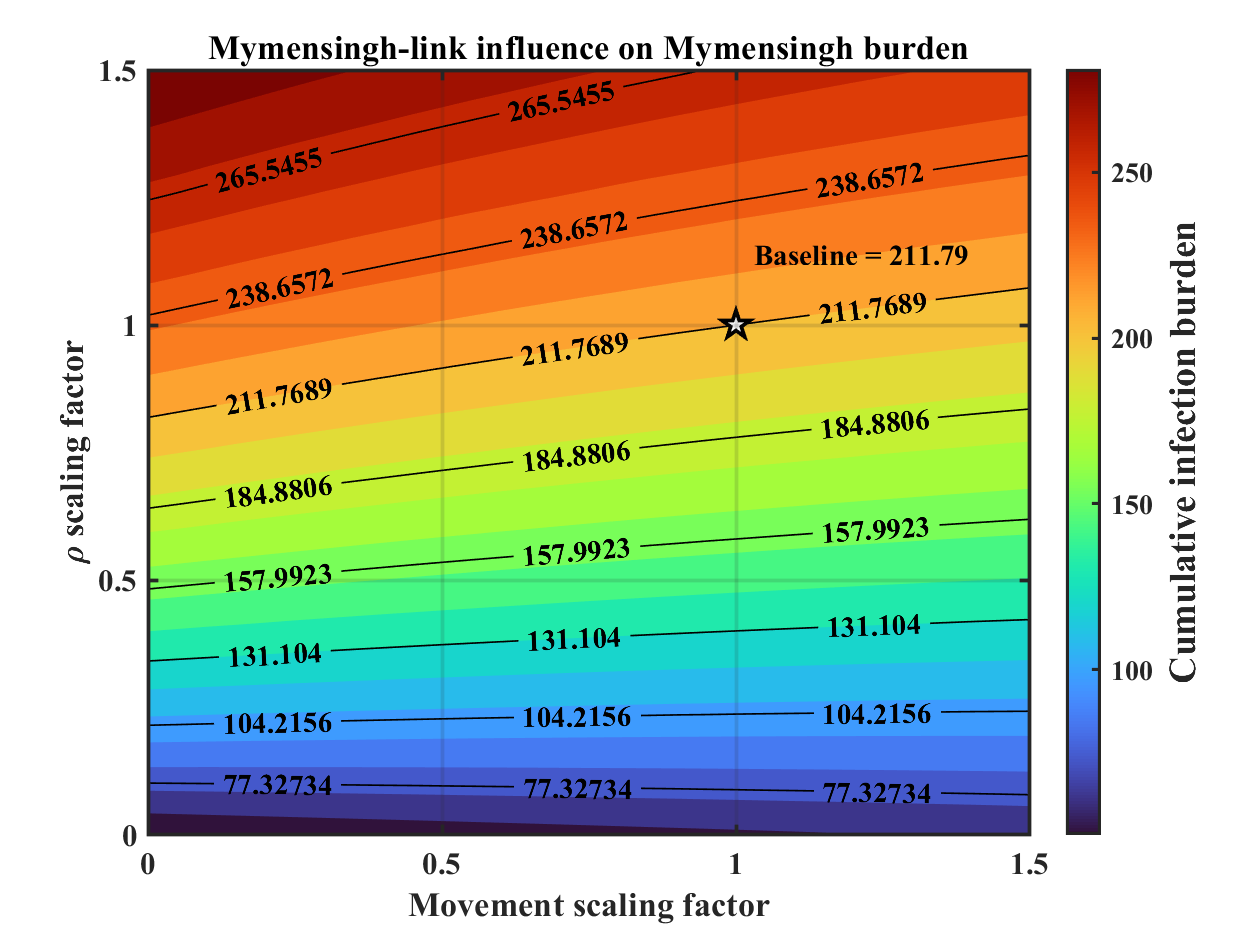}
		\caption{Mymensingh-link-specific scaling}
		\label{link_Mymensingh}
	\end{subfigure}
	
	\vspace{0.3cm}
	
	\begin{subfigure}[t]{0.48\textwidth}
		\centering
		\includegraphics[width=\linewidth]{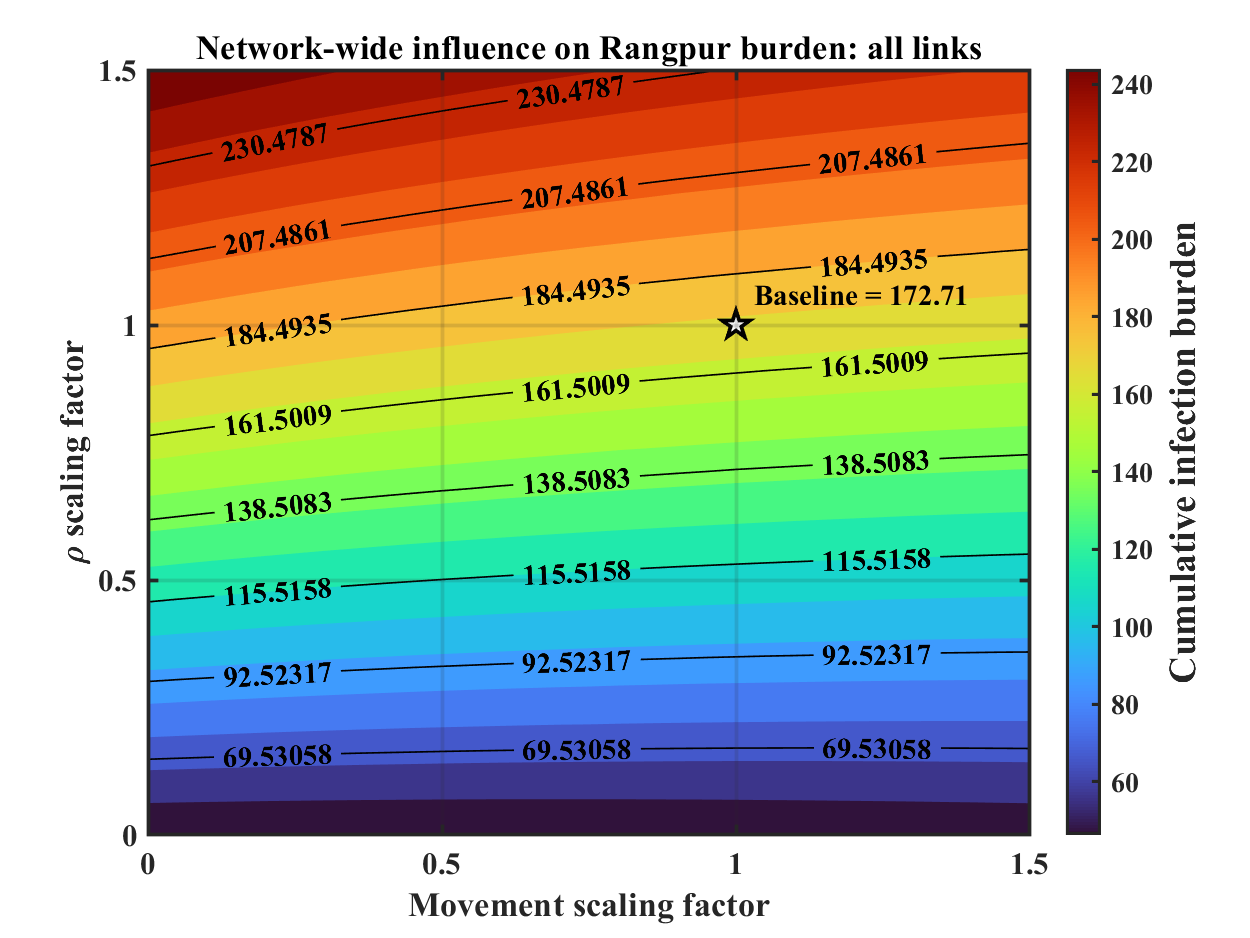}
		\caption{Rangpur: network-wide scaling}
		\label{global_Rangpur}
	\end{subfigure}
	\hfill
	\begin{subfigure}[t]{0.48\textwidth}
		\centering
		\includegraphics[width=\linewidth]{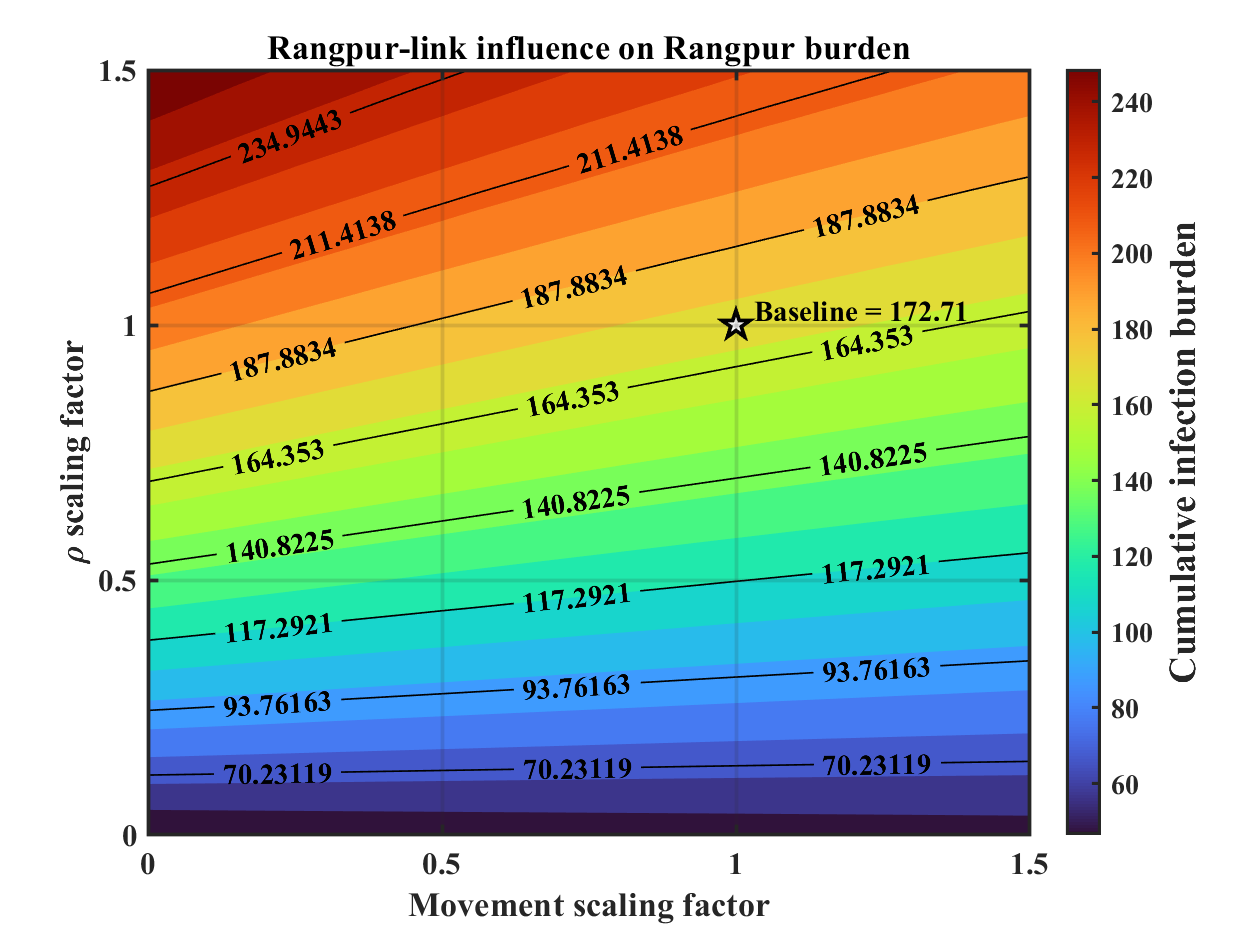}
		\caption{Rangpur-link-specific scaling}
		\label{link_Rangpur}
	\end{subfigure}
	
	\caption{Sensitivity of the cumulative infection burden in Mymensingh and
		Rangpur to movement and cross-patch infection. The left panels vary
		all interdivisional links simultaneously, whereas the right panels vary
		only links connected to the indicated division while retaining all other
		links at their baseline values. Stars indicate the baseline point
		$(1,1)$.}
	\label{contour_mr}
\end{figure}
%%===================================================================

\clearpage
%%%%%%%%%%%%%%%%%%%%%%%%%%%%%%%%%%%%%%%%%%%%%%%%%%%%%%%%%%%%%%%%%
\setcounter{table}{2}

%\begin{landscape}
	
	\section{Supplementary Tables}
	\label{tables}
	%\thispagestyle{empty}

%---------------------Table1---------------------------------------

	\begin{center}
	%\centering
%	\captionof{table}{Division-specific demographic inputs, hub-adjustment factors, fitted epidemiological parameters, and reproduction-number estimates used in the multipatch model.}
	\captionof{table}{Division-specific demographic inputs, hub-adjustment factors, fitted epidemiological parameters, and reproduction-number estimates used in the multipatch model.}
	\label{tab1}
	\begin{adjustbox}{
		max width=\textwidth,
		max totalheight=0.72\textheight,
		center
			}
			\begin{minipage}{\textwidth}
				\centering
				\small
				\renewcommand{\arraystretch}{1.18}
				\setlength{\tabcolsep}{4.5pt}
				
				\textbf{(a) Fixed demographic and model inputs}
				
				\vspace{1.5mm}
		\resizebox{\linewidth}{!}{%
			\begin{tabular}{lcccccc}
				\toprule
				
				\rowcolor{TableBlue}
				\textcolor{white}{\textbf{Division}}
				&
				\textcolor{white}{\textbf{Population}}
				&
				\textcolor{white}{
					\shortstack{\textbf{Receiving}\\\textbf{weight}}
				}
				&
				\textcolor{white}{
					\shortstack{\textbf{Source}\\\textbf{activity}}
				}
				&
				\textcolor{white}{
					\shortstack{\textbf{Recruitment}\\
						$\boldsymbol{\Lambda_i}$\\
						\textbf{(individuals day$^{-1}$)}}
				}
				&
				\textcolor{white}{
					\shortstack{\textbf{Recovery}\\
						$\boldsymbol{\gamma_i}$\\
						\textbf{(day$^{-1}$)}}
				}
				&
				\textcolor{white}{
					\shortstack{\textbf{Relative}\\
						$\boldsymbol{\mathcal{R}_{0,i}^{\mathrm{rel}}}$\\
						\textbf{(unitless)}}
				}
				\\
				
				\midrule
				
				Dhaka
				& 45,644,586 & 3.0000 & 1.4000
				& 0.210000 & 0.050060 & \textbf{1.0000}\\
				
				Chattogram
				& 34,178,612 & 2.5000 & 1.3000
				& 0.130000 & 0.070000 & 0.9720\\
				
				Rajshahi
				& 20,794,019 & 1.0000 & 1.0000
				& 0.110000 & 0.070500 & 0.5813\\
				
				Barishal
				& 9,325,820 & 1.0000 & 1.0000
				& 0.070000 & 0.071000 & 0.3956\\
				
				Sylhet
				& 11,415,113 & 1.0000 & 1.0000
				& 0.056000 & 0.065000 & 0.3679\\
				
				Khulna
				& 17,813,218 & 1.0000 & 1.0000
				& 0.051000 & 0.059000 & 0.2618\\
				
				Mymensingh
				& 12,637,472 & 1.0000 & 1.0000
				& 0.041000 & 0.063000 & 0.3384\\
				
				Rangpur
				& 18,020,071 & 1.0000 & 1.0000
				& 0.060000 & 0.071000 & 0.4105\\
				
				\bottomrule
			\end{tabular}%
		}
		
		\vspace{4mm}
		
		\textbf{(b) Fitted parameters and 95\% bootstrap confidence intervals}
		
		\vspace{1.5mm}
		
		\resizebox{\linewidth}{!}{%
			\begin{tabular}{lccccc}
				\toprule
				
				\rowcolor{TableBlue}
				\textcolor{white}{\textbf{Division}}
				&
				\textcolor{white}{
					\shortstack{\textbf{Effective transmission}\\
						\textbf{coefficient, }$\boldsymbol{\beta_i}$\\
						\textbf{(individual$^{-1}$ day$^{-1}$)}}
				}
				&
				\textcolor{white}{
					\shortstack{\textbf{Disease-induced}\\
						\textbf{mortality, }$\boldsymbol{\delta_i}$\\
						\textbf{(day$^{-1}$)}}
				}
				&
				\textcolor{white}{
					\shortstack{\textbf{Vaccine}\\
						 \textbf{efficacy, }$\boldsymbol{\epsilon_i}$\\
						\textbf{(unitless)}}
				}
				&
				\textcolor{white}{
					\shortstack{\textbf{Vaccination} \\
						\textbf{scaling, }$\boldsymbol{\alpha_i}$\\
						\textbf{(day$^{-1}$)}}
				}
				&
				\textcolor{white}{
					\shortstack{\textbf{Reproduction}\\
						 \textbf{number, }$\boldsymbol{\mathcal{R}_0^i}$\\
						\textbf{(unitless)}}
				}
				\\
				
				\midrule
				
				Dhaka
				& \estci{0.00016564}{0.00012992}{0.00020577}
				& \estci{0.00410708}{0.00081232}{0.00765798}
				& \estci{0.75421625}{0.62667265}{0.92238768}
				& \estci{0.86575158}{0.38380511}{0.99999901}
				& \estci{\textbf{4.349145}}{1.444935}{5.630704}\\
				
				Chattogram
				& \estci{0.00015381}{0.00011107}{0.00019945}
				& \estci{0.00501308}{0.00036558}{0.02280095}
				& \estci{0.11411897}{0.10000004}{0.30724939}
				& \estci{0.93756460}{0.26110391}{0.99996145}
				& \estci{4.227464}{4.001903}{4.318450}\\
				
				Rajshahi
				& \estci{0.00017919}{0.00013258}{0.00024739}
				& \estci{0.03379371}{0.02236295}{0.06066636}
				& \estci{0.51446226}{0.16099108}{0.77664745}
				& \estci{0.70065721}{0.23490762}{0.99335765}
				& \estci{2.528084}{1.413498}{3.209059}\\
				
				Barishal
				& \estci{0.00009496}{0.00008881}{0.00011793}
				& \estci{0.00008181}{0.00000164}{0.00044229}
				& \estci{0.33195783}{0.30000069}{0.50923110}
				& \estci{0.83921630}{0.46486581}{0.99974021}
				& \estci{1.720343}{1.528806}{1.805039}\\
				
				Sylhet
				& \estci{0.00011692}{0.00010167}{0.00014345}
				& \estci{0.00031357}{0.00007406}{0.00054486}
				& \estci{0.42078139}{0.30830902}{0.56896282}
				& \estci{0.26708902}{0.16855624}{0.65137615}
				& \estci{1.600028}{1.354918}{1.798346}\\
				
				Khulna
				& \estci{0.00015788}{0.00012086}{0.00019020}
				& \estci{0.00000550}{0.00000130}{0.00070779}
				& \estci{0.69725947}{0.53491339}{0.86818106}
				& \estci{0.55502642}{0.24096809}{0.91441142}
				& \estci{1.138642}{0.570290}{1.383432}\\
				
				Mymensingh
				& \estci{0.00015159}{0.00013621}{0.00019500}
				& \estci{0.00000037}{0.00000026}{0.00048312}
				& \estci{0.45841133}{0.37059838}{0.63272883}
				& \estci{0.54140787}{0.28561095}{0.94368715}
				& \estci{1.471743}{1.263061}{1.550761}\\
				
				Rangpur
				& \estci{0.00010984}{0.00010715}{0.00012769}
				& \estci{0.00002254}{0.00000010}{0.00045667}
				& \estci{0.30128849}{0.30000063}{0.43843522}
				& \estci{0.97405897}{0.41074438}{1.00000000}
				& \estci{1.785397}{1.637892}{1.840509}\\
				
				\bottomrule
			\end{tabular}%
		}
		
		\vspace{3mm}
		
		\noindent
		\parbox{\linewidth}{%
			\footnotesize
			\raggedright
			\textbf{\textit{Note~\ref{tab1}:}} Population data are obtained from the \textit{Population and Housing Census 2022} \cite{bbs}.
			The natural mortality rate, \(\mu=(75.5\times365)^{-1}
			=0.0000363~\mathrm{day}^{-1}\), is based on the national
			life expectancy reported by Worldometer \cite{world} and is
			applied uniformly across all divisions. Specified inputs are
			not assigned confidence intervals. Bracketed values are
			percentile-based 95\% bootstrap confidence intervals from
			1,000 replicates.
			\(\mathcal{R}_{0,i}^{\mathrm{rel}}=\mathcal{R}_0^i/\max_k\mathcal{R}_0^k\);
			bold type identifies the largest estimate.
		}
		
	\end{minipage}
\end{adjustbox}
\end{center}

	%---------------------------Table 2---------------------------------
	\begin{table}[H]
		\centering
		\caption{Geographical coordinates and pairwise great-circle distances
			among the eight divisions of Bangladesh.}
		\label{tab2}
		
		\renewcommand{\arraystretch}{1.18}
		\setlength{\tabcolsep}{3pt}
		
		%==============================================================
		% Panel (a): Coordinates
		%==============================================================
		\textbf{(a) Geographical coordinates}
		
		\vspace{1.5mm}
		
		\begin{tabular}{lcc}
			\toprule
			
			\rowcolor{TableBlue}
			\textcolor{white}{\textbf{Division}}
			&
			\textcolor{white}{\textbf{Latitude ($^\circ$N)}}
			&
			\textcolor{white}{\textbf{Longitude ($^\circ$E)}}\\
			
			\midrule
			
			Dhaka       & 23.8103 & 90.4125\\
			Chattogram  & 22.3569 & 91.7832\\
			Rajshahi    & 24.3745 & 88.6042\\
			Barishal    & 22.7010 & 90.3535\\
			Sylhet      & 24.8949 & 91.8687\\
			Khulna      & 22.8456 & 89.5403\\
			Mymensingh  & 24.7471 & 90.4203\\
			Rangpur     & 25.7439 & 89.2752\\
			
			\bottomrule
		\end{tabular}
		
		\vspace{4mm}
		
		%==============================================================
		% Panel (b): Pairwise distances
		%==============================================================
		\textbf{(b) Pairwise great-circle distances (km)}
		
		\vspace{1.5mm}
		
		\resizebox{\textwidth}{!}{%
			\begin{tabular}{lcccccccc}
				\toprule
				
				\rowcolor{TableBlue}
				\textcolor{white}{\textbf{Receiving division}}
				&
				\multicolumn{8}{c}{
					\cellcolor{TableBlue}
					\textcolor{white}{\textbf{Source division}}
				}\\
				
				\rowcolor{HeaderFill}
				{}
				& \textcolor{TableBlue}{\textbf{Dhaka}}
				& \textcolor{TableBlue}{\textbf{Chattogram}}
				& \textcolor{TableBlue}{\textbf{Rajshahi}}
				& \textcolor{TableBlue}{\textbf{Barishal}}
				& \textcolor{TableBlue}{\textbf{Sylhet}}
				& \textcolor{TableBlue}{\textbf{Khulna}}
				& \textcolor{TableBlue}{\textbf{Mymensingh}}
				& \textcolor{TableBlue}{\textbf{Rangpur}}\\
				
				\midrule
				
				\textbf{Dhaka}
				& 0.00 & 213.95 & 193.98 & 123.50
				& 190.54 & 139.42 & 104.17 & 243.74\\
				
				\textbf{Chattogram}
				& 213.95 & 0.00 & 394.47 & 151.75
				& 282.35 & 236.57 & 299.89 & 454.60\\
				
				\textbf{Rajshahi}
				& 193.98 & 394.47 & 0.00 & 257.73
				& 334.98 & 194.93 & 188.28 & 166.60\\
				
				\textbf{Barishal}
				& 123.50 & 151.75 & 257.73 & 0.00
				& 288.57 & 84.91 & 227.62 & 355.58\\
				
				\textbf{Sylhet}
				& 190.54 & 282.35 & 334.98 & 288.57
				& 0.00 & 328.59 & 147.10 & 277.24\\
				
				\textbf{Khulna}
				& 139.42 & 236.57 & 194.93 & 84.91
				& 328.59 & 0.00 & 229.61 & 323.39\\
				
				\textbf{Mymensingh}
				& 104.17 & 299.89 & 188.28 & 227.62
				& 147.10 & 229.61 & 0.00 & 159.84\\
				
				\textbf{Rangpur}
				& 243.74 & 454.60 & 166.60 & 355.58
				& 277.24 & 323.39 & 159.84 & 0.00\\
				
				\bottomrule
			\end{tabular}%
		}
		
		\vspace{2mm}
		
		\parbox[t]{\textwidth}{%
			\footnotesize
			\textbf{\textit{Note~\ref{tab2}:}} Division-level latitude and longitude coordinates are
			obtained from the NASA POWER Data Access Viewer \cite{nasa}. Pairwise distances are calculated using the Haversine formula \cite{haversine}, with a mean Earth radius of \(R=6371\) km. The resulting values represent great-circle distances rather than road-travel distances. Consequently, the distance matrix is symmetric, and its diagonal elements are zero.
		}
	\end{table}
	%---------------------Table3---------------------------------------
\begin{table}[H]
	\centering
	\caption{Hub-adjusted directional connectivity matrix \(Q\) among the
		eight divisions.}
	\label{tabq}
	
	\renewcommand{\arraystretch}{1.20}
	\setlength{\tabcolsep}{3pt}
	
	\resizebox{\textwidth}{!}{%
		\begin{tabular}{lcccccccc}
			\toprule
			
			\rowcolor{TableBlue}
			\textcolor{white}{\textbf{Receiving division}}
			&
			\multicolumn{8}{c}{
				\cellcolor{TableBlue}
				\textcolor{white}{\textbf{Source division}}
			}\\
			
			\rowcolor{HeaderFill}
			{}
			& \textcolor{TableBlue}{\textbf{Dhaka}}
			& \textcolor{TableBlue}{\textbf{Chattogram}}
			& \textcolor{TableBlue}{\textbf{Rajshahi}}
			& \textcolor{TableBlue}{\textbf{Barishal}}
			& \textcolor{TableBlue}{\textbf{Sylhet}}
			& \textcolor{TableBlue}{\textbf{Khulna}}
			& \textcolor{TableBlue}{\textbf{Mymensingh}}
			& \textcolor{TableBlue}{\textbf{Rangpur}}\\
			
			\midrule
			
			\textbf{Dhaka}
			& 0.000000 & \textbf{0.800000} & 0.529416 & 0.556901
			& 0.399344 & 0.681769 & 0.768550 & 0.392226\\
			
			\textbf{Chattogram}
			& 0.717949 & 0.000000 & 0.187732 & 0.326825
			& 0.194332 & 0.289736 & 0.192512 & 0.151647\\
			
			\textbf{Rajshahi}
			& 0.247061 & 0.097620 & 0.000000 & 0.060037
			& 0.051104 & 0.109706 & 0.095666 & 0.129108\\
			
			\textbf{Barishal}
			& 0.259887 & 0.169949 & 0.060037 & 0.000000
			& 0.039729 & 0.168664 & 0.052995 & 0.040509\\
			
			\textbf{Sylhet}
			& 0.186361 & 0.101053 & 0.051104 & 0.039729
			& 0.000000 & 0.048220 & 0.090727 & 0.057482\\
			
			\textbf{Khulna}
			& 0.318159 & 0.150663 & 0.109706 & 0.168664
			& 0.048220 & 0.000000 & 0.072607 & 0.061558\\
			
			\textbf{Mymensingh}
			& 0.358656 & 0.100106 & 0.095666 & 0.052995
			& 0.090727 & 0.072607 & 0.000000 & 0.104905\\
			
			\textbf{Rangpur}
			& 0.183039 & 0.078856 & 0.129108 & 0.040509
			& 0.057482 & 0.061558 & 0.104905 & 0.000000\\
			
			\bottomrule
		\end{tabular}%
	}
	
	\vspace{2mm}
	
	\parbox[t]{\textwidth}{%
		\footnotesize
		\textbf{\textit{Note~\ref{tabq}:}} Rows represent receiving divisions and columns represent source divisions. Thus, \(Q(i,j)\) denotes relative connectivity from source division \(j\) to receiving division \(i\). The zero diagonal excludes within-division connectivity from the interdivisional matrix. Because \(Q\) is directional, \(Q(i,j)\) is not necessarily equal to \(Q(j,i)\). The bold value identifies the strongest directional connection.
	}
\end{table}
%---------------------Table4---------------------------------------
\begin{table}[H]
	\centering
	\caption{Reproduction-number-adjusted cross-patch infection influence
		matrix \(\rho\) among the eight divisions.}
	\label{tabr}
	
	\renewcommand{\arraystretch}{1.20}
	\setlength{\tabcolsep}{3pt}
	
	\resizebox{\textwidth}{!}{%
		\begin{tabular}{lcccccccc}
			\toprule
			
			\rowcolor{TableBlue}
			\textcolor{white}{\textbf{Receiving division}}
			&
			\multicolumn{8}{c}{
				\cellcolor{TableBlue}
				\textcolor{white}{\textbf{Source division}}
			}\\
			
			\rowcolor{HeaderFill}
			{}
			& \textcolor{TableBlue}{\textbf{Dhaka}}
			& \textcolor{TableBlue}{\textbf{Chattogram}}
			& \textcolor{TableBlue}{\textbf{Rajshahi}}
			& \textcolor{TableBlue}{\textbf{Barishal}}
			& \textcolor{TableBlue}{\textbf{Sylhet}}
			& \textcolor{TableBlue}{\textbf{Khulna}}
			& \textcolor{TableBlue}{\textbf{Mymensingh}}
			& \textcolor{TableBlue}{\textbf{Rangpur}}\\
			
			\midrule
			
			\textbf{Dhaka}
			& 0.00000000 & \textbf{0.06220940} & 0.02461926
			& 0.01762296 & 0.01175333 & 0.01427943
			& 0.02080607 & 0.01288124\\
			
			\textbf{Chattogram}
			& 0.05743590 & 0.00000000 & 0.00873001
			& 0.01034228 & 0.00571950 & 0.00606842
			& 0.00521166 & 0.00498029\\
			
			\textbf{Rajshahi}
			& 0.01976488 & 0.00759113 & 0.00000000
			& 0.00189985 & 0.00150408 & 0.00229776
			& 0.00258987 & 0.00424009\\
			
			\textbf{Barishal}
			& 0.02079096 & 0.01321553 & 0.00279187
			& 0.00000000 & 0.00116928 & 0.00353261
			& 0.00143468 & 0.00133039\\
			
			\textbf{Sylhet}
			& 0.01490885 & 0.00785803 & 0.00237648
			& 0.00125721 & 0.00000000 & 0.00100996
			& 0.00245614 & 0.00188778\\
			
			\textbf{Khulna}
			& 0.02545272 & 0.01171579 & 0.00510163
			& 0.00533732 & 0.00141920 & 0.00000000
			& 0.00196562 & 0.00202166\\
			
			\textbf{Mymensingh}
			& 0.02869252 & 0.00778443 & 0.00444874
			& 0.00167702 & 0.00267023 & 0.00152074
			& 0.00000000 & 0.00344523\\
			
			\textbf{Rangpur}
			& 0.01464312 & 0.00613201 & 0.00600388
			& 0.00128191 & 0.00169178 & 0.00128932
			& 0.00283998 & 0.00000000\\
			
			\bottomrule
		\end{tabular}%
	}
	
	\vspace{2mm}
	
	\parbox[t]{\textwidth}{%
		\footnotesize
		\textbf{\textit{Note~\ref{tabr}:}} Rows represent receiving divisions and columns represent source divisions. Thus, \(\rho_{ij}\) denotes the additional infection pressure exerted by infectious individuals in source division \(j\) on receiving division \(i\). The zero diagonal excludes within-division transmission, which is represented separately by the local force of infection. Because \(\rho\) is directional, \(\rho_{ij}\) is not necessarily equal to \(\rho_{ji}\). The bold value identifies the strongest cross-patch infection influence.
	}
\end{table}
%---------------------Table5---------------------------------------
	\begin{table}[H]
		\centering
		\caption{Division-to-division cross-patch infection contributions at the baseline total epidemic peak.}
		\label{tab5}
		
		\renewcommand{\arraystretch}{1.20}
		\setlength{\tabcolsep}{3pt}
		
		\resizebox{\textwidth}{!}{%
			\begin{tabular}{lcccccccc}
				\toprule
				
				\rowcolor{TableBlue}
				\textcolor{white}{\textbf{Receiving division}}
				&
				\multicolumn{8}{c}{
					\cellcolor{TableBlue}
					\textcolor{white}{\textbf{Source division}}
				}\\
				
				\rowcolor{HeaderFill}
				{}
				& \textcolor{TableBlue}{\textbf{Dhaka}}
				& \textcolor{TableBlue}{\textbf{Chattogram}}
				& \textcolor{TableBlue}{\textbf{Rajshahi}}
				& \textcolor{TableBlue}{\textbf{Barishal}}
				& \textcolor{TableBlue}{\textbf{Sylhet}}
				& \textcolor{TableBlue}{\textbf{Khulna}}
				& \textcolor{TableBlue}{\textbf{Mymensingh}}
				& \textcolor{TableBlue}{\textbf{Rangpur}}\\
				
				\midrule
				
				\textbf{Dhaka}
				& 0.000000 & 0.234198 & 0.006974 & 0.014009
				& 0.016064 & 0.011441 & 0.026717 & 0.011481\\
				
				\textbf{Chattogram}
				& \textbf{0.400039} & 0.000000 & 0.004622 & 0.015366
				& 0.014611 & 0.009088 & 0.012509 & 0.008297\\
				
				\textbf{Rajshahi}
				& 0.110796 & 0.042992 & 0.000000 & 0.002272
				& 0.003093 & 0.002770 & 0.005003 & 0.005685\\
				
				\textbf{Barishal}
				& 0.120154 & 0.077160 & 0.001227 & 0.000000
				& 0.002479 & 0.004390 & 0.002857 & 0.001839\\
				
				\textbf{Sylhet}
				& 0.082661 & 0.044016 & 0.001002 & 0.001487
				& 0.000000 & 0.001204 & 0.004693 & 0.002504\\
				
				\textbf{Khulna}
				& 0.100368 & 0.046674 & 0.001529 & 0.004490
				& 0.002053 & 0.000000 & 0.002671 & 0.001907\\
				
				\textbf{Mymensingh}
				& 0.157240 & 0.043099 & 0.001853 & 0.001961
				& 0.005367 & 0.001792 & 0.000000 & 0.004516\\
				
				\textbf{Rangpur}
				& 0.097459 & 0.041232 & 0.003038 & 0.001820
				& 0.004130 & 0.001845 & 0.006514 & 0.000000\\
				
				\bottomrule
			\end{tabular}%
		}
		
		\vspace{2mm}
		
		\parbox[t]{\textwidth}{%
			\footnotesize
			\textbf{\textit{Note~\ref{tab5}:}} Rows represent receiving divisions and columns represent source divisions. Each entry gives the model-estimated number of new
			infections per day generated in receiving division \(i\) by cross-patch infection pressure from source division \(j\) at the baseline total epidemic peak. The zero diagonal excludes local transmission, which is calculated separately within each division. Summing the entries across a row gives the total cross-patch new-infection rate received by that division, whereas summing down a column gives the total cross-patch
			new-infection rate generated by the corresponding source division. The bold value identifies the largest pairwise contribution.
		}
	\end{table}
	%---------------------Table6---------------------------------------
	
	%---------------------Table7---------------------------------------
	\begin{table}[H]
		\centering
		\caption{Percentage changes in epidemic outcomes relative to the baseline coupling scenario.}
		\label{tab7}
		
		\renewcommand{\arraystretch}{1.12}
		\setlength{\tabcolsep}{6pt}
		
		\resizebox{\textwidth}{!}{%
			\begin{tabular}{lcccc}
				\toprule
				
				\rowcolor{TableBlue}
				\textcolor{white}{\textbf{Scenario}}
				&
				\multicolumn{4}{c}{
					\cellcolor{TableBlue}
					\textcolor{white}{
						\textbf{Percentage change relative to baseline coupling}
					}
				}\\
				
				\rowcolor{HeaderFill}
				{}
				&
				\makecell[c]{
					\textcolor{TableBlue}{\textbf{Peak infectious}}\\
					\textcolor{TableBlue}{\textbf{population}}\\
					\textcolor{TableBlue}{\textbf{(\%)}}
				}
				&
				\makecell[c]{
					\textcolor{TableBlue}{\textbf{Time to}}\\
					\textcolor{TableBlue}{\textbf{peak}}\\
					\textcolor{TableBlue}{\textbf{(\%)}}
				}
				&
				\makecell[c]{
					\textcolor{TableBlue}{\textbf{Final infectious}}\\
					\textcolor{TableBlue}{\textbf{population}}\\
					\textcolor{TableBlue}{\textbf{(\%)}}
				}
				&
				\makecell[c]{
					\textcolor{TableBlue}{\textbf{Cumulative}}\\
					\textcolor{TableBlue}{\textbf{infections}}\\
					\textcolor{TableBlue}{\textbf{(\%)}}
				}\\
				
				\midrule
				
				No spatial coupling
				& \(-39.0301\) & \(-72.4638\)
				& \(-55.1292\) & \(-53.6777\)\\
				
				\rowcolor{HeaderFill!70}
				\textbf{Baseline coupling}
				& 0.0000 & 0.0000 & 0.0000 & 0.0000\\
				
				50\% restriction
				& \(-26.6854\) & \(-68.1159\)
				& \(-21.3661\) & \(-25.0732\)\\
				
				80\% restriction
				& \(-34.3682\) & \(-71.0145\)
				& \(-38.8732\) & \(-41.4763\)\\
				
				50\% increase
				& \(+30.1366\) & \(-5.7971\)
				& \(+15.4439\) & \(+23.5361\)\\
				
				\bottomrule
			\end{tabular}%
		}
		
		\vspace{2mm}
		
		\parbox[t]{\textwidth}{%
			\footnotesize
			\textbf{\textit{Note~\ref{tab7}:}} Percentage changes are calculated relative to the
			baseline-coupling scenario as
			\(
			100\,(X_{\mathrm{scenario}}-X_{\mathrm{baseline}})
			/X_{\mathrm{baseline}}
			\).
			Negative values indicate reductions relative to baseline, whereas positive values indicate increases. For time to peak, a negative value indicates that the epidemic peak occurs earlier than under baseline coupling. The no-spatial-coupling scenario excludes both interdivisional movement and cross-patch infection.
		}
	\end{table}
	%---------------------Table8---------------------------------------
	
	%---------------------Table9---------------------------------------
\begin{table}[H]
	\centering
	\caption{Division-specific compartment movement balance and movement-related exposure at the baseline network-wide epidemic peak.}
	\label{tab9}
	
	\renewcommand{\arraystretch}{1.12}
	\setlength{\tabcolsep}{5pt}
	
	\resizebox{\textwidth}{!}{%
		\begin{tabular}{lccccc}
			\toprule
			
			\rowcolor{TableBlue}
			\textcolor{white}{\textbf{Division}}
			&
			\multicolumn{4}{c}{
				\cellcolor{TableBlue}
				\textcolor{white}{\textbf{Net movement and exposure}}
			}
			&
			\cellcolor{TableBlue}
			\textcolor{white}{\textbf{Infectious-movement status}}
			\\
			
			\rowcolor{HeaderFill}
			{}
			&
			\makecell[c]{
				\textcolor{TableBlue}{\textbf{Net susceptible}}\\
				\textcolor{TableBlue}{\textbf{movement}}\\
				\textcolor{TableBlue}{\(\boldsymbol{\mathrm{Net}S_i}\)}\\
				\textcolor{TableBlue}{\textbf{(individuals}}\\
				\textcolor{TableBlue}{\textbf{day\(^{-1}\))}}
			}
			&
			\makecell[c]{
				\textcolor{TableBlue}{\textbf{Net vaccinated}}\\
				\textcolor{TableBlue}{\textbf{movement}}\\
				\textcolor{TableBlue}{\(\boldsymbol{\mathrm{Net}V_i}\)}\\
				\textcolor{TableBlue}{\textbf{(individuals}}\\
				\textcolor{TableBlue}{\textbf{day\(^{-1}\))}}
			}
			&
			\makecell[c]{
				\textcolor{TableBlue}{\textbf{Movement-related}}\\
				\textcolor{TableBlue}{\textbf{exposure tendency}}\\
				\textcolor{TableBlue}{\(\boldsymbol{E_{\mathrm{move},i}}\)}\\
				\textcolor{TableBlue}{\textbf{(individuals}}\\
				\textcolor{TableBlue}{\textbf{day\(^{-2}\))}}
			}
			&
			\makecell[c]{
				\textcolor{TableBlue}{\textbf{Net infectious}}\\
				\textcolor{TableBlue}{\textbf{movement}}\\
				\textcolor{TableBlue}{\(\boldsymbol{\mathrm{Net}I_i}\)}\\
				\textcolor{TableBlue}{\textbf{(individuals}}\\
				\textcolor{TableBlue}{\textbf{day\(^{-1}\))}}
			}
			&
			{}
			\\
			
			\midrule
			
			\textbf{Dhaka}
			& \textbf{0.014957}
			& \textbf{3.205284}
			& \textbf{0.01305315}
			& \(-0.047152\)
			& \textcolor{purple}{\textbf{Exporter}}\\
			
			\textbf{Chattogram}
			& 0.014365
			& 1.484500
			& 0.01294109
			& \(-0.034384\)
			& \textcolor{purple}{\textbf{Exporter}}\\
			
			\textbf{Rajshahi}
			& \(-0.002124\)
			& \(-0.232086\)
			& 0.00000000
			& \textbf{0.027207}
			& \textcolor{green!50!black}{\textbf{Importer}}\\
			
			\textbf{Barishal}
			& 0.001977
			& \(-1.347536\)
			& 0.00000424
			& 0.018020
			& \textcolor{green!50!black}{\textbf{Importer}}\\
			
			\textbf{Sylhet}
			& \(-0.045540\)
			& \(-0.636919\)
			& 0.00000000
			& 0.002392
			& \textcolor{green!50!black}{\textbf{Importer}}\\
			
			\textbf{Khulna}
			& 0.004670
			& \(-1.129316\)
			& 0.00001703
			& 0.019101
			& \textcolor{green!50!black}{\textbf{Importer}}\\
			
			\textbf{Mymensingh}
			& 0.007695
			& \(-0.513334\)
			& 0.00004049
			& 0.005604
			& \textcolor{green!50!black}{\textbf{Importer}}\\
			
			\textbf{Rangpur}
			& 0.004000
			& \(-0.830593\)
			& 0.00001043
			& 0.009212
			& \textcolor{green!50!black}{\textbf{Importer}}\\
			
			\bottomrule
		\end{tabular}%
	}
	
	\vspace{2mm}
	
	\parbox[t]{\textwidth}{%
		\footnotesize
		\raggedright
		\textbf{\textit{Note~\ref{tab9}:}} Net movement is defined as inflow minus outflow;
		hence, positive values indicate net import and negative values indicate net export. The exporter/importer classification is based on \(\mathrm{Net}I_i\). The movement-related exposure tendency is
		\(E_{\mathrm{move},i}
		=\lambda_i\max(\mathrm{Net}S_i,0)
		+(1-\epsilon_i)\lambda_i\max(\mathrm{Net}V_i,0)\).
		It measures the instantaneous change in infection-rate potential associated with positive susceptible and vaccinated inflows, rather than the number of infections accumulated over a day. Zero values may reflect numerical rounding. Bold type identifies the largest value in each numerical column.
	}
\end{table}
	%---------------------Table10---------------------------------------
	\begin{table}[H]
		\centering
		\caption{Division-specific cross-patch source--sink classification at the
			baseline total epidemic peak.}
		\label{tab10}
		
		\renewcommand{\arraystretch}{1.12}
		\setlength{\tabcolsep}{6pt}
		
		\resizebox{\textwidth}{!}{%
			\begin{tabular}{lcccc}
				\toprule
				
				\rowcolor{TableBlue}
				\textcolor{white}{\textbf{Division}}
				&
				\multicolumn{3}{c}{
					\cellcolor{TableBlue}
					\textcolor{white}{\textbf{Cross-patch infection balance}}
				}
				&
				\cellcolor{TableBlue}
				\textcolor{white}{\textbf{Spatial role}}\\
				
				\rowcolor{HeaderFill}
				{}
				&
				\makecell[c]{
					\textcolor{TableBlue}{\textbf{Imported}}\\
					\textcolor{TableBlue}{\textbf{new-infection rate}}\\
					\textcolor{TableBlue}{\textbf{(individuals day\(^{-1}\))}}
				}
				&
			\makecell[c]{
				\textcolor{TableBlue}{\textbf{Exported}}\\
				\textcolor{TableBlue}{\textbf{new-infection rate}}\\
				\textcolor{TableBlue}{\textbf{(individuals day\(^{-1}\))}}
			}
				&
				\makecell[c]{
					\textcolor{TableBlue}{\textbf{Net source}}\\
					\textcolor{TableBlue}{\textbf{rate}}\\
					\textcolor{TableBlue}{\textbf{(individuals day\(^{-1}\))}}
				}
				&
				{}\\
				
				\midrule
				
				\textbf{Dhaka}
				& 0.320884 & \textbf{1.068717} & \textbf{0.747833}
				& \textcolor{red!75!black}{\textbf{Cross-patch source}}\\
				
				\textbf{Chattogram}
				& \textbf{0.464533} & 0.529372 & 0.064839
				& \textcolor{red!75!black}{\textbf{Cross-patch source}}\\
				
				\textbf{Rajshahi}
				& 0.172610 & 0.020246 & \(-0.152365\)
				& \textcolor{blue!70!black}{\textbf{Cross-patch sink}}\\
				
				\textbf{Barishal}
				& 0.210105 & 0.041404 & \(-0.168701\)
				& \textcolor{blue!70!black}{\textbf{Cross-patch sink}}\\
				
				\textbf{Sylhet}
				& 0.137566 & 0.047797 & \(-0.089769\)
				& \textcolor{blue!70!black}{\textbf{Cross-patch sink}}\\
				
				\textbf{Khulna}
				& 0.159691 & 0.032529 & \(-0.127162\)
				& \textcolor{blue!70!black}{\textbf{Cross-patch sink}}\\
				
				\textbf{Mymensingh}
				& 0.215828 & 0.060963 & \(-0.154865\)
				& \textcolor{blue!70!black}{\textbf{Cross-patch sink}}\\
				
				\textbf{Rangpur}
				& 0.156038 & 0.036229 & \(-0.119810\)
				& \textcolor{blue!70!black}{\textbf{Cross-patch sink}}\\
				
				\bottomrule
			\end{tabular}%
		}
		
		\vspace{2mm}
		
		\parbox[t]{\textwidth}{%
			\footnotesize
			\textbf{\textit{Note~\ref{tab10}:}} Imported values represent the instantaneous model-estimated rate of new infections generated within a division by cross-patch infection pressure from all other divisions. Exported values represent the corresponding rate of new infections generated by that division in the remaining divisions. The net source rate is calculated as the exported rate minus the imported rate; a positive value identifies a cross-patch source, whereas a negative value identifies a cross-patch sink. All quantities are evaluated at the baseline network-wide epidemic peak and are expressed as newly infected individuals per day. Bold numerical values indicate the largest value in the corresponding column.
		}
	\end{table}
\end{document}